\documentclass[10pt,letterpaper]{amsart}
\usepackage{graphicx}
\usepackage{a4wide}
\usepackage{units}
\usepackage{color}
\usepackage{subfig}
\usepackage{float}
\usepackage{graphicx}
\usepackage[colorlinks]{hyperref}
\usepackage{upref}
\usepackage{mathtools}
\usepackage{todonotes}
\usepackage{fancybox}
\usepackage{amsthm}
\usepackage{xfrac}

\newtheorem{theorem}{Theorem}[section]
\newtheorem{definition}[theorem]{Definition}
\newtheorem{lemma}[theorem]{Lemma}

\newtheorem{remark}[theorem]{Remark}
\newtheorem*{maintheorem*}{Main Theorem}
\allowdisplaybreaks
\numberwithin{equation}{section}

\newcommand{\abs}[1]{\left|#1\right|}
\newcommand{\norm}[1]{\left\| #1 \right\|}
\newcommand{\R}{\mathbb{R}}

\newcommand{\eps}{\epsilon}

\newcommand{\wh}[1]{\widehat{#1}}

\newcommand{\Dt}{{\Delta t}}

\newcommand{\Grad}{\nabla}
\newcommand{\Div}{\operatorname{div}}
\newcommand{\Curl}{\operatorname{curl}}
\newcommand{\vphi}{\varphi}
\newcommand{\dom}{\Omega}

\newcommand{\bd}{\mathbf{d}}
\newcommand{\rand}{\partial\dom}

\newcommand{\dhalf}{\bd^{m+1/2}}
\newcommand{\disp}{\sigma}
\newcommand{\damp}{\alpha}
\newcommand{\oneconst}{k}

\newcommand{\Econst}{\epsilon_1}

\newcommand{\pot}{\varphi}

\newcommand{\epsi}{\epsilon_1}
\newcommand{\epsii}{\epsilon_2}
\newcommand{\bF}{f}
\newcommand{\bw}{\mathbf{w}}

\newcommand{\bi}{\underline{i}}
\newcommand{\bj}{\mathbf{e}_j}

\newcommand{\weak}{\rightharpoonup}

\newcommand{\weakstar}{\overset{\star}\rightharpoonup}

\graphicspath{figures}

\newcommand{\equationbox}[2]{\begin{center}

		\fbox{
			\begin{minipage}{0.96\textwidth}
				#1
			\end{minipage}
		}
	\end{center}
}
\begin{document}

\title[Maxwell Liquid Crystals] {A convergent numerical scheme for a model of liquid crystal dynamics subjected to an electric field}
\date{\today}
\author[F. Weber]{Franziska Weber}
\address[Franziska Weber]{\newline Department of Mathematical Sciences \newline Carnegie Mellon University \newline 5000 Forbes Avenue, Pittsburgh, PA 15213, USA.}
\email[]{franzisw@andrew.cmu.edu}
\thanks{The author's work was supported in part by NSF awards DMS-1912854 and DMR-2020915.}





 \begin{abstract}
 	We present a convergent and constraint-preserving numerical discretization of a mathematical model for the dynamics of a liquid crystal subjected to an electric field. This model can be derived from the Oseen-Frank director field theory, assuming that the dynamics of the electric field are governed by the electrostatics equations with a suitable constitutive relation for the electric displacement field that describes the coupling with the liquid crystal director field. The resulting system of partial differential equations consists of an elliptic equation that is coupled to the wave map equations through a quadratic source term. We show that the discretization preserves the unit length constraint of the director field, is energy-stable and convergent. In numerical experiments, we show that the method is stable even when singularities develop. Moreover, predictions about the alignment of the director field with the electric field are confirmed.
 \end{abstract}

 \maketitle
\section{Introduction}
Liquid crystals are materials that consist of elongated molecules and exhibit states that lie between the liquid and the solid phase, the so called liquid crystal phase.
Some examples include shaving cream, cell membranes, as well as the nematic liquid crystals used in displays (LCD displays). 
The microscopic structure of liquid crystals affects at a macroscopic scale the mechanical response to stress and strain. For instance, the molecules of certain liquid crystals react to electric fields on a microscopic scale, which on a macroscopic scale changes the polarization of the light passing through the material. Monitors take advantage of this property to allow a certain amount of red, green, or blue light through each pixel.
Their sensitivity to electric and magnetic fields which allows to orient the molecules in a certain way, makes liquid crystals attractive for many engineering applications.

In this article, we develop a convergent numerical scheme for a simple model for the dynamics of a liquid crystal director field subjected to electromagnetic forces. The model is derived from the Oseen-Frank theory which suggests that the liquid crystal director field, $\bd:\dom\subset \R^n\rightarrow \mathbb{S}^2$, $n=1,2,3$, which is a unit length vector pointing in the direction of the main orientation of the molecules, behaves in such a way that in an equilibrium, the Oseen-Frank energy~\cite{Oseen1993,frank1958},
\begin{equation}\label{eq:ofe}
E_{OF}(\bd) = \int_{\dom} w_{OF}(\bd,\Grad\bd),
\end{equation}
with the energy density
\begin{equation}
\label{eq:ofed}
w_{OF}(\bd,\Grad \bd)=\frac12 k_1 (\Div\bd)^2+\frac12 k_2 (\bd\cdot\Curl\bd)^2+k_3(\bd\times\Curl\bd)^2+\frac12(k_2+k_4)\Div[(\bd\cdot\Grad)\bd-(\Div\bd)\bd],
\end{equation}
is minimized. A common approach is to set $k_1=k_2=k_3=\oneconst>0$ and $k_4=0$, the so called \emph{one-constant approximation}~\cite{Stewart2004}. In this case, it can be shown that the energy density becomes
\begin{equation}
\label{eq:oneconste}
w_{OF}=\frac12 \oneconst |\Grad \bd|^2.
\end{equation}
 We will also take this approach here, despite it not being realistic for all physical scenarios. The scheme that is proposed here, can be extended in a stable manner to the general case with different constants, but it is unclear whether convergence can be shown in that case also.

Models that incorporate the effects of external fields on the liquid crystal are obtained by adding terms to the Oseen-Frank energy density~\eqref{eq:ofed}. For an electric field $E$, this term has been found to be of the form~\cite{Stewart2004},
\begin{equation}
\label{eq:electricenergy}
w_e = -\frac{1}{2} E\cdot D,\quad D = \eps_0\eps_\perp E+\eps_0\eps_a(\bd\cdot E)\bd, 
\end{equation}
where $D$ is the electric displacement field, 
$\eps_0$ is the permittivity of free space, $\eps_\perp,\eps_a$ are dielectric constants depending on material properties and $\eps_a$ is called dielectric anisotropy of the liquid crystal. Values of $\eps_a$ can be negative or positive, depending on the type of liquid crystal. Typical values are between $-0.7$ and $11$~\cite{Collings1998,Skarp1980,Stephen1974}. When $\eps_a>0$, the director field tends to align with the electric field and when $\eps_a<0$, it prefers to be perpendicular~\cite{Stewart2004,Chandrasekhar1992,deGennes1995}.

In an equilibrium situation, the director field minimizes the free energy~\eqref{eq:ofe}, now modified with term~\eqref{eq:electricenergy}, while the electric field $E$ and the electric displacement field $D$ satisfy the electrostatics equations (Gauss' law and Faraday's law of induction):
\begin{equation}
\label{eq:electrostatics}
\Div D = \rho_{\text{free}},\quad \Curl E= 0,
\end{equation} 
where $\rho_{\text{free}}$ is the free electric charge density. In liquid crystal applications, there are generally no free charges, so one can set this source term to be zero. In a smooth, simply-connected domain, the second equation implies that $E$ can be written as the gradient of a potential, i.e., $E=\Grad\vphi$. Combining~\eqref{eq:electrostatics} with the expression for $D$ in equation~\eqref{eq:electricenergy}, we thus obtain an elliptic equation for the electric potential
\begin{equation}
\label{eq:ellipticE}
-\Div(\Grad\vphi+\epsii(\bd\cdot\Grad\vphi)\bd)=f,
\end{equation}
where $\epsii:=\eps_a\eps_\perp^{-1}$ and $f=-\rho_{\text{free}}/(\eps_0\eps_\perp)$,
that can be augmented with Dirichlet boundary conditions 
\begin{equation}\label{eq:bcpotential}
\vphi(t,x) = g(t,x),\quad x\in \partial\dom,\, t\geq 0,
\end{equation}
for the potential. 
To account for non-equilibrium situations such as the dynamic transition between two equilibrium states, e.g., the Freedericksz transition~\cite{Freedericksz1927,Freedericksz1933,Shelestiuk2011,Aursand2016}, one adds inertial and damping terms to the equations for the director field (which can be obtained as the Euler-Lagrange equations from a suitable action functional)~\cite{Glassey1996,Macdonald2015}, since the liquid crystal molecules move on a slower time scale compared to the electric field:
\begin{equation}\label{eq:directorefield}
\begin{split}
\disp\bd_{tt}+\damp\bd_t =\oneconst\Delta \bd +\gamma\bd+ \epsi(\bd\cdot E)E,
\end{split}
\end{equation}
where $\gamma$ is a Lagrange multiplier enforcing the constraint $|\bd|=1$ and $\eps_1:=\eps_0\eps_a$. In addition, one could add the effects of the fluid flow and would end up with a simplified version of the Ericksen-Leslie system~\cite{Ericksen1976,Leslie1987}. 
Combining~\eqref{eq:ellipticE},~\eqref{eq:bcpotential} and~\eqref{eq:directorefield}, and augmenting it with initial and boundary conditions for $\bd$, we obtain the system
\begin{subequations}
	\label{eq:model1}
	\begin{align}
	\label{eq:deq}
	\disp\bd_{tt}+\damp\bd_t =\oneconst\Delta \bd +\gamma\bd+ \epsi(\bd\cdot \Grad\vphi)\Grad\vphi,\quad &x\in \dom,\, t>0,\\
	-\Div(\Grad\vphi+\epsii(\bd\cdot\Grad\vphi)\bd)=f,\quad &x\in \dom,\, t\geq 0\\
	\vphi = g,\quad \Grad\bd\cdot \nu = 0,\quad &x\in\partial\dom,\, t\geq 0,\\
	\bd(0,x)=\bd_0(x),\quad \partial_t\bd(0,x)=\bd_1(x),\quad &x\in \dom,\\
	|\bd|=1,\quad &\textrm{a.e. }\, x\in\dom,\, t\geq 0,
	\end{align}
\end{subequations}
where $\nu$ is the outward unit normal, and $\bd_0$, $\bd_1$ a given initial data. Due to the nonlinearities and the constraint $|\bd|=1$, it may not be realistic to expect classical solutions to system~\eqref{eq:model1}. In fact, it has been shown that blow-up solutions to the wave maps equation (which is obtained from~\eqref{eq:deq} by setting $\damp=\epsi=0$) exist in two and three spatial dimensions, see e.g.,~\cite{Shatah1994,Tataru2004,Krieger2008,Cazenave1998}. Hence it might be more meaningful to seek out weak solutions. We define these as follows:
\begin{definition}\label{def:weaksol}
	Let $T>0$ and $\dom\subset\R^n$ a bounded domain with uniformly Lipschitz boundary. Assume $f\in W^{1,2}([0,T];L^2(\dom))$ and $g\in C^1([0,T]\times\partial\dom)$ which can be extended to a function $\widetilde{g}\in L^\infty(0,T;H^1(\dom))\cap W^{1,2}([0,T];H^1(\dom))$. We say $(\bd,\pot)$ is a weak solution of \eqref{eq:model1} for initial data $\bd_0\in H^1(\dom)$ with $|\bd_0(x)|=1$ a.e. and $\bd_1\in L^2(\dom)$, if $\bd\in L^\infty((0,T)\times\dom)\cap L^\infty((0,T); H^1(\dom))$, $\bd_t\in L^\infty((0,T); L^2(\dom))$ and $\vphi\in L^\infty((0,T;H^1(\dom)))$ satisfy for all $\phi\in C^\infty_c([0,T)\times\dom)$ and all $\psi\in H^1_0(\dom)$
	\begin{align}\label{eq:weakformulation}
	\int_0^T\int_{\dom} \left(\disp (\bd_t\times\bd)\cdot\partial_t \phi-\damp (\bd_t\times\bd)\cdot\phi\right) dx dt&+\disp\int_{\dom}(\bd_1\times\bd_0)\cdot\phi(0,x) dx\\
	=\int_0^T\int_{\dom}\big((\oneconst\Grad\bd\times \bd):\Grad\phi&-\Econst (\bd\cdot \Grad\vphi) (\Grad\vphi\times \bd)\cdot\phi\big) dxdt,\notag\\
	\int_{\dom}  \Grad\pot\cdot\Grad \psi+ \epsii (\bd\cdot\Grad\pot)\cdot(\bd\cdot\Grad\psi) dx &=
	\int_{\dom}\bF \psi dx,  \quad\text{a.e. } t\in[0,T],\notag\\	
	|\bd|=1,\quad \text{a.e. in } (0,T)\times\dom;&\quad \vphi=g,\quad x\in\partial\dom,\, t\in [0,T],\notag
	\end{align}
	and 
	\begin{equation}\label{eq:contat0}
	\lim_{t\rightarrow 0^+} \norm{\bd(t,\cdot)-\bd_0}_{L^2(\dom)}=0.
	\end{equation}
\end{definition}

To design a stable and convergent numerical scheme for~\eqref{eq:model1}, which preserves the unit length constraint for $\bd$, we choose  to reformulate the equations using the \emph{angular momentum} $\bw:=\bd_t\times \bd$. This approach was taken in~\cite{Karper2014} for the wave map equation and resulted in a constraint-preserving convergent scheme. Formally, assuming smooth solutions,~\eqref{eq:model1} then becomes
\begin{subequations}
	\label{eq:angmomentum}
	\begin{align}
	\label{seq:dt}
	\bd_t = \bd\times\bw,\quad &x\in \dom,\, t>0,\\
	\disp\bw_{t}+\damp\bw =\oneconst\Delta \bd \times\bd+ \epsi(\bd\cdot \Grad\vphi)\Grad\vphi\times\bd,\quad &x\in \dom,\, t>0,\\
	-\Div(\Grad\vphi+\epsii(\bd\cdot\Grad\vphi)\bd)=f,\quad &x\in \dom,\, t>0\\
	\vphi = g,\quad \Grad\bd\cdot \nu = 0,\quad &x\in\partial\dom,\, t\geq 0,\\
	\bd(0,x)=\bd_0(x),\quad \bw(0,x)=\bd_1(x)\times\bd_0(x),\quad &x\in \dom,
	\end{align}
\end{subequations}
We notice that the Lagrange multiplier term which enforces the constraint disappears. Nevertheless, this reformulation preserves the constraint, at least at a formal level. One can see this, by taking the inner product of~\eqref{seq:dt} with $\bd$:
\begin{equation*}
\frac12\partial_t |\bd|^2=\bd\cdot\bd_t = (\bd\times\bw)\cdot \bd =0,
\end{equation*}
where we used the orthogonality property of the cross product. Hence, the modulus of $\bd$ stays constant in time at every spatial point, and if initially $|\bd_0(x)|=1$, this will be true for any $t>0$ as well, at least at a formal level. Using an implicit discretization in time, this constraint can be preserved at the discrete level. The implicit discretization will also ensure energy-stability which in turn allows to obtain the necessary a priori estimates for passing the discretization parameters to zero and showing that the limit is a weak solution of~\eqref{eq:model1}. To discretize in space, we use a combination of finite differences for the variables $\bd$ and $\bw$ and piecewise linear finite elements for $\vphi$. This ensures preservation of the constraint $|\bd|=1$ everywhere on the grid. It is probably also possible to use a finite element discretization for $\bd$ and $\bw$, as it was done in~\cite{Bartels2015} for the wave map equation, but then one can only hope to preserve the constraint at every node of the mesh or on average in each grid cell. In addition, the method we propose here, is rather simple and straightforward to implement. 
Because the algebraic system of equations resulting from the discretization is nonlinear, we show that a unique solution exists using a fixed point iteration.

Most rigorous numerical analysis results for related systems do not conserve the constraint $|\bd|=1$ explicitly and use instead a relaxed version, obtained by adding a term
\begin{equation*}
w_{R}=\frac{1}{4\delta}(|\bd|^2-1)^2,
\end{equation*}
where $\delta>0$ is a small parameter, to the energy density~\eqref{eq:ofed}, e.g.,~\cite{Liu2000,Liu2002,Walkington2011}. This results in the solution having higher regularity and hence better compactness properties. As we will see, for system~\eqref{eq:model1}, this is not necessary to achieve stability or convergence of the numerical scheme. Besides that, works in which the unit length constraint on the director field is preserved, concern simpler systems, such as the wave or heat map flow or the Landau-Lifshitz-Gilbert equation~\cite{Bartels2007,Alouges2006,Barrett2007,Bartels2008,Bartels2015,Banas2010,Bartels2009,Hu2009}, or 1D settings~\cite{Aursand2016,Aursand2015}. The numerical methods of~\cite{Aursand2016,Aursand2015} have not been shown to be convergent, to the best of our knowledge. Therefore, our contribution is the construction of a constraint-preserving and convergent discretization of~\eqref{eq:model1}. The convergence proof additionally proves existence of weak solutions for~\eqref{eq:model1}. Other numerical methods using Hamiltonian discretizations related to the angular momentum discussed above to preserve energy and constraints are e.g.~\cite{Austin1993,Lewis2000,Krishnaprasad1987}. 
\begin{remark}
	Magnetic forces could be added in a similar way as electric fields, using the magnetostatics equations. We do not consider these here because the resulting system is of similar type and no mathematical difficulty is added, except that the system becomes larger, as it involves second elliptic equation that needs to be solved, and hence its numerical solution takes more computational effort. 
\end{remark}

\subsection{Outline of this article}
In the following section, we present the numerical scheme and prove that it preserves the unit length constraint and satisfies an energy stability bound. Then, in Section~\ref{sec:conv}, we prove it converges to a weak solution of~\eqref{eq:model1}, and in Section~\ref{sec:fixed}, we show that the nonlinear algebraic system of equations resulting from the discretization possesses a unique solution which can be obtained via fixed point iteration. We conclude with numerical experiments in Section~\ref{sec:num}.
\section{The numerical scheme}
In this section, we describe the numerical scheme to approximate system~\eqref{eq:angmomentum} and prove that it preserves the unit length constraint for the director field, and is energy-stable. We start by introducing the time discretization.
\subsection{Time discretization}
We let $T>0$ the final time and $\Dt>0$ some small number, conditions on it will be specified later on. We define $t^m:=m\Dt$, the time steps, and $N_T=T/\Dt\in \mathbb{N}$ ($\Dt$ is w.l.o.g. such that $T/\Dt$ is an integer). Then we define the time differences and averages,
\begin{equation*}
D_t a^m :=\frac{a^{m+1}-a^m}{\Dt},\quad a^{m+1/2}=\frac{a^m+a^{m+1}}{2},
\end{equation*}
for vectors $a=(a^0,\dots,a^{N_T})$.
We consider the following implicit discretization of~\eqref{eq:angmomentum}:
\begin{subequations}\label{eq:m1}
	\begin{align}
	&D_t \bd^m=\dhalf\times \bw^{m+1/2},\\
	&\disp D_t \bw^{m}=\oneconst\Delta \dhalf\times \dhalf-\damp \bw^{m+1/2}\\
	&\hphantom{\disp D_t \bw^{m}=\oneconst}+\frac{\epsi}{2} \bigl[(\Grad\vphi^{m+1}\cdot \dhalf) \Grad\vphi^m+ (\Grad\vphi^{m}\cdot \dhalf) \Grad\vphi^{m+1} \bigr]\times \dhalf,  \\
	\label{seq3:m1fi}
	&-\Div ( \Grad \vphi^m+ \epsii(\bd^m\cdot \Grad \vphi^m) \bd^m )=\bF^m,
	\end{align}
\end{subequations}
$m=0,1,\dots$, and $x\in\dom$, with boundary conditions 
\begin{equation}\label{eq:bctd}
\Grad\bd^{m}\cdot \nu=0,\quad \vphi^m=g^m, \quad x\in \rand,
\end{equation}
where $\nu$ is the outward unit normal and $g^m(x)=g(t^m,x)$, $m=0,1,2,\dots$. We assume that $g$ has an extension $\widetilde{g}\in L^\infty([0,T],;H^1(\dom))\cap W^{1,2}([0,T];H^1(\dom))$ to the whole space.
One can derive an energy balance for this method:
\begin{multline*}
\frac{1}{2}D_t \int_{\dom} \left(\oneconst|\Grad \bd^m|^2+\frac{\disp}{2}|\bw^m|^2+2|\Grad\vphi^m|^2+2\epsii (\bd^m\cdot \Grad\vphi^m)^2-2\vphi^m f^{m}\right)dx\\
-\left(1+\frac{\epsi}{2\epsii}\right)D_t\int_{\dom}\left(\Grad\vphi^m+\epsii(\Grad\vphi\cdot\bd^m)\bd^m\right)\Grad\widetilde{g}^m dx+D_t\int_{\dom}f^m\widetilde{g}^m dx\\
= -\frac{\epsi}{2\epsii}\int_{\dom}D_t\Grad\widetilde{g}^m\left(2\Grad\vphi^{m+\sfrac{1}{2}}+\epsii(\Grad\vphi^m\cdot\bd^m)\bd^m+\epsii(\Grad\vphi^{m+1}\cdot\bd^{m+1})\bd^{m+1}\right) dx\\
      +\frac{\epsi}{\epsii}\int_{\dom} (\vphi^{m+\sfrac{1}{2}}-\widetilde{g}^{m+\sfrac{1}{2}}) D_t f^m \, dx-\damp\int_{\dom}|\bw^{m+\sfrac{1}{2}}|^2 dx.
\end{multline*}
Using Poincar\'e's and Gr\"{o}nwall's inequality, one can derive bounds on $\bd^m, \bw^m$ and $\vphi^m$ from this. Since we will do this in detail for the fully discrete method, we do not outline the details here. 
\subsection{Spatial discretization}\label{sec:sdisc}
To discretize space, we use first order finite differences for the variables $\bd$ and $\bw$ (and approximate them using piecewise constants), and a finite element formulation for the variable $\vphi$ with continuous piecewise linear basis functions. This way, the constraint $|\bd|=1$ is conserved at any point in space and time, and it is easy to formulate a convergent fixed point iteration, as we will see later.
We assume for this section that $\dom=[0,1]^n$, the extension of the numerical method and analysis to other rectangular domains is straightforward. Using a generalized version of the 9-point stencil Laplacian (or 5-point in 2D), it may be possible to extend to more complicated domains also. 
For the variables $\bd$ and $\bw$ we define the following grid: We let $h:=1/N>0$ for some $N\in\mathbb{N}$, the mesh width, define multiindices 
\begin{equation*}
\bi:=(i_1,\dots,i_n)\in \mathcal{I}_{N}:=\{0,\dots, N+1\}^n
\end{equation*}
and grid cells $\mathcal{C}_{\bi} = ((i_1-1)h,i_1 h]\times\dots\times ((i_n-1)h,i_n h]$ for $\bi\in \mathcal{I}_N$. 
We let $\mathring{\mathcal{I}}_N:=\{1,\dots, N\}^n$ the indices for the interior cells, and $\partial\mathcal{I}_N:=\{\bi\,|\, \exists j\,\text{with }\, i_j\in \{0,N+1 \}\}$ the indices of the boundary cells.
 The approximations of $\bd$ and $\bw$ in grid cell $\mathcal{C}_{\bi}$ at time $t^m$ will be denoted by $\bd_{\bi}^m$ and $\bw_{\bi}^m$. 
 It will also be useful to define the piecewise constant functions 
 \begin{align}
 \label{eq:pwconstd}
 \bd_h^m(x)=\bd_{\bi}^m,\quad x\in \mathcal{C}_{\bi},\qquad& \bd_h(t,x)  = \bd_{h}^m(x),\quad t\in [t^m,t^{m+1}),\, x\in \dom,\\
 \bw_h^m(x)=\bw_{\bi}^m,\quad x\in \mathcal{C}_{\bi},\qquad&\bw_h(t,x)  = \bw_{h}^m(x),\quad t\in [t^m,t^{m+1}),\, x\in\dom,\\
 \bd_h^{m+\nicefrac{1}{2}}(x)=\bd_{\bi}^{m+\nicefrac{1}{2}},\quad x\in \mathcal{C}_{\bi},\qquad&\wh{\bd}_h(t,x)  = \bd_{h}^{m+\nicefrac{1}{2}}(x),\quad t\in [t^m,t^{m+1}),\, x\in \dom,\\
 \label{eq:whhat}
 \bw_h^{m+\nicefrac{1}{2}}(x)=\bw_{\bi}^{m+\nicefrac{1}{2}},\quad x\in \mathcal{C}_{\bi},\qquad&\wh{\bw}_h(t,x)  = \bw_{h}^{m+\nicefrac{1}{2}}(x),\quad t\in [t^m,t^{m+1}),\, x\in \dom.
 \end{align}
  The variable $\vphi$ will be approximated in space using continuous piecewise linear Lagrangian finite elements on a simplicial uniformly shape-regular triangulation $\mathcal{T}=\{K_i\}_{i=1}^{N_{\mathcal{T}}}$ for $\dom$ that could be independent of the discretization for $\bd$ and $\bw$ as long as $h_{\mathcal{T}}:=\max_{i=1,\dots,N_{\mathcal{T}}}\text{diam} K_i\approx h$. We let $\mathcal{N}(\mathcal{T})$ be the set of nodes of the triangulation and $\mathcal{N}(\mathring{\mathcal{T}})$ the set of interior nodes. To keep things simple, we can pick a regular triangulation in which the nodes coincide with the corners of the cells $\mathcal{C}_{\bi}$. We denote $\vphi_{i}^m$ the approximation of $\vphi$ at time $t^m$ at node $w_i\in\mathcal{N}(\mathcal{T})$ of the mesh and $b_i$ the hat function associated with the same node. We denote $\widetilde{V}_h:=\text{span}(b_1,\dots,b_{\mathcal{N}(\mathcal{T})})$ and $V_h=\text{span}(b_i\,|\, w_i\in \mathcal{N}(\mathring{\mathcal{T}}))$. We then let
 \begin{equation}\label{eq:pwconstphi}
\begin{split}
 \vphi^m_h(x)&=\sum_{w_i\in\mathcal{N}(\mathcal{T})}\vphi^m_i b_i(x),\quad x\in\dom,\\
  \vphi_h(t,x)&=\vphi^m_h(x),\quad  t\in [t^m,t^{m+1}),\,x\in\dom,
\end{split}
 \end{equation}
 that is, a piecewise constant interpolation of $\vphi^m_h$ in time.
 Suppose $\widetilde{g}\in L^\infty(0,T;H^1(\dom))\cap W^{1,2}(0,T;H^1(\dom))$ is an extension of the boundary data $g$ which exists if $g$ is sufficiently smooth and the boundary of $\dom$ is uniformly Lipschitz~\cite{Necas2012}. Then we define its piecewise linear in space and piecewise constant in time interpolation on $\mathcal{T}$ by
 \begin{equation}\label{eq:ghtilde}
 \begin{split}
 \widetilde{g}_h^m(x)& = \sum_{w_i\in\mathcal{N}(\mathcal{T})}\widetilde{g}(t^m,w_i) b_i(x),\quad  x\in\dom,\\
  \widetilde{g}_h(t,x)& = \sum_{w_i\in\mathcal{N}(\mathcal{T})}\widetilde{g}(t^m,w_i) b_i(x),\quad  t\in [t^m,t^{m+1}),\,x\in\dom. 
 \end{split}
 \end{equation}
 Then we let $u^m_{i}=\vphi^m_i-\widetilde{g}(t^m,w_i)$ for $w_i\in \mathcal{N}(\mathcal{T})$ and define
 \begin{equation}
 \label{eq:uoh}
 \begin{split}
 u_{0,h}^m(x)&=\vphi_h^m(x)-\widetilde{g}_h(t^m,x)=\sum_{w_i\in\mathcal{N}(\mathring{\mathcal{T}})}u^m_i b_i(x),\quad x\in\dom,\, m=0,1,2,\dots,\\
 u_{0,h}(t,x)&=\vphi_h(t,x)-\widetilde{g}_h(t,x)=u_{0,h}^m(x),\quad t\in [t^m,t^{m+1}),\,x\in\dom.
 \end{split}
 \end{equation}
 Clearly, if $\vphi_h(t,\cdot)\in H^1(\dom)$, then $u_{0,h}(t,\cdot)\in H^1_0(\dom)$ and we can recover $\vphi_h$ from $u_{0,h}$ by adding $\widetilde{g}_h$.

To define and analyze the scheme, we will need the following difference operators for $\bd$ and $\bw$:  
 For a quantity $\sigma_{\bi}$ defined on the grid $\{\mathcal{C}_{\bi}\}_{\bi\in\mathcal{I}_N}$, we denote
\begin{equation*}
D_j^\pm \sigma_{\bi} := \pm \frac{\sigma_{\bi\pm\bj}-\sigma_{\bi}}{h},\quad \Grad_h^{\pm}:= (D_1^\pm,\dots,D_n^\pm)^\top,\quad 
\Delta_h:=\sum_{j=1}^n D_j^\pm D_j^\mp.
\end{equation*} 
We denoted by $\bj$ the $j$th unit vector. We also need the bilinear forms
\begin{equation*}
\begin{split}
a(u,v)(t)&=\int_{\dom}\left[\Grad u\cdot\Grad v+\epsii(\bd(t,x)\cdot\Grad u)(\bd(t,x)\cdot \Grad v)\right] dx,\quad t\in [0,T],\\
a^m(u,v)&=\int_{\dom}\left[\Grad u\cdot\Grad v+\epsii(\bd(t^m,x)\cdot\Grad u)(\bd(t^m,x)\cdot \Grad v)\right] dx,\quad m=0,1,2,\dots,
\end{split}
\end{equation*}
which are well-defined for functions $u,v\in H^1(\dom)$ as long as $\bd\in L^\infty([0,T]\times\dom)$, and the discrete versions
\begin{equation*}
\begin{split}
a_h(u,v)(t)&=\int_{\dom}\left[\Grad u\cdot\Grad v+\epsii(\bd_h(t,x)\cdot\Grad u)(\bd_h(t,x)\cdot \Grad v)\right] dx,\quad t\in [0,T],\\
a_h^m(u,v)&=\int_{\dom}\left[\Grad u\cdot\Grad v+\epsii(\bd_h(t^m,x)\cdot\Grad u)(\bd_h(t^m,x)\cdot \Grad v)\right] dx,\quad m=0,1,2,\dots
\end{split}
\end{equation*}
for $u,v\in \widetilde{V}_h$, and
\begin{equation*}
F(v)(t): = \int_{\dom}f(t,x)v(x) dx,\quad t\in [0,T],\quad  F^m(v):=\int_{\dom}f(t^m,x)v(x)dx,\quad m=1,2,\dots,
\end{equation*}
for $v\in H^1(\dom)$.
Based on \eqref{eq:m1}, we define the fully discrete scheme
\begin{subequations}\label{eq:m2}
	\begin{align}
	\label{seq:dd}
	&D_t \bd^m_{\bi}=\dhalf_{\bi}\times \bw^{m+1/2}_{\bi},\\
	\label{seq:wd}
	&\disp D_t \bw^{m}_{\bi}=\oneconst\Delta_h \dhalf_{\bi}\times \dhalf_{\bi}-\damp \bw^{m+1/2}_{\bi}\\
	&\qquad +\frac{\epsi}{2|\mathcal{C}_{\bi}|}\int_{\mathcal{C}_{\bi}}\left((\Grad\vphi_h^{m+1} \cdot\dhalf_{\bi})\Grad\vphi^m_h +(\Grad\vphi_h^{m} \cdot\dhalf_{\bi})\Grad\vphi^{m+1}_h\right) dx\times \dhalf_{\bi},	\notag  \\
&	u_{0,h}^m\in V_h,\quad\text{s.t.}\quad a_h^m(u_{0,h}^m,v)=F^m(v)-a_h^m(\widetilde{g}_h^m,v),\quad \forall v\in V_h,\quad \vphi^m_h := u_{0,h}^m+\widetilde{g}_h^m.
	\label{seq:phid}
	\end{align}
\end{subequations}
$m=0,1,\dots$, $\bi\in\{1,\dots,N\}^n$, with homogeneous boundary conditions for the variable $\bd$ 
	\begin{equation*}
\bd_{\bi}^m=\bd_{\bi+\bj}^m,\quad i_j=0,\quad \bd_{\bi+\bj}^m=\bd_{\bi}^m,\quad i_j=N,
\end{equation*}
and initial conditions
\begin{equation*}
\bd^0_{\bi}=\frac{1}{|\mathcal{C}_{\bi}|}\int_{\mathcal{C}_{\bi}}\bd_0(x) dx ,\quad \bw^0_{\bi}=\frac{1}{|\mathcal{C}_{\bi}|}\int_{\mathcal{C}_{\bi}} \bw_0(x) dx .
\end{equation*}
\subsubsection{Constraint preservation and discrete energy stability}
We first note that the numerical scheme~\eqref{eq:m2} preserves the constraint $|\bd^m_{\bi}|=1$:
\begin{lemma}\label{lem:constraint}
	If $|\bd_{\bi}^0|=1$ for all $\bi\in \mathcal{I}_N$, then $|\bd_{\bi}^m|=1$ for all $\bi\in \mathcal{I}_N$ and $m\in\mathbb{N}$.
\end{lemma}
\begin{proof}
	We take the inner product of equation \eqref{seq:dd} with $\bd^{m+1/2}_{\bi}$. The right hand side is orthogonal to $\bd_{\bi}^{m+1/2}$, hence it vanishes. For the left hand side we have
	\begin{equation*}
	\bd^{m+1/2}_{\bi}\cdot D_t \bd^m_{\bi} = \frac{1}{2\Dt} \left(\bd^m_{\bi} + \bd^{m+1}_{\bi}\right) \cdot \left(\bd^{m+1}_{\bi}- \bd^m_{\bi}\right) = \frac{1}{2\Dt}\left( \left| \bd^{m+1}_{\bi}\right|^2 - \left|\bd^m_{\bi}\right|^2 \right).  
	\end{equation*}	
	Hence $|\bd_{\bi}^{m+1}|^2 = |\bd_{\bi}^m|^2 = \dots = |\bd_{\bi}^0|^2 =1$.
\end{proof}	
The numerical scheme also satisfies a discrete energy inequality which is crucial for stability and proving convergence:
\begin{lemma}\label{lem:discreteenergy}
	Assume $\Delta t<\tau$ for some sufficiently small constant $\tau>0$, $\oneconst>0$, $\disp>0$ and $\damp\geq 0$ (or $\damp>0$ and $\disp\geq 0$), and $\epsii>-1$. Then, the approximations $(\bd_h,\bw_h,\vphi_h)$ defined through scheme~\eqref{eq:m2} and~\eqref{eq:pwconstd}-\eqref{eq:pwconstphi}, satisfy for all $M\in\mathbb{N}$
	\begin{align}\label{eq:discreteenergyestimate}
	&\frac{\oneconst}{2}\norm{\Grad_h\bd_h(t^M)}_{L^2}^2+\frac{\disp}{2}\norm{\bw_h(t^M)}_{L^2}^2+
	\norm{\Grad\vphi_h(t^M)}_{L^2}^2\\
	&\leq 
	C\left(\norm{\widetilde{g}_h(t^M)}_{H^1}^2+\norm{f(t^M)}_{L^2}^2\right)+C\exp(Ct^M)\Dt\sum_{m=0}^{M-1}\norm{f(t^m)}_{L^2}^2\notag\\
&\quad	+C(t^M\exp(Ct^M)+1)\left(\norm{\Grad\bd_0}^2_{L^2}+ \norm{\bw_0}_{L^2}^2
	+	\norm{\Grad\vphi_h^0}^2_{L^2(\dom)}+\norm{f(0)}_{L^2}^2+\norm{\widetilde{g}_h^0}_{H^1(\dom)}^2\right)\notag\\
&\quad	+C(1+t^M\exp(C t^M))\Dt\left(\sum_{m=0}^{M-1}\norm{D_t f(t^m)}_{L^2}^2+\sum_{m=0}^{M-1}\norm{D_t\Grad\widetilde{g}_h(t^m)}^2_{L^2}\right)\notag\\
&\quad	+ C\exp(Ct^M)(1+t^M)\Dt\sum_{m=0}^M \norm{\widetilde{g}_h(t^m)}^2_{H^1}-2\damp\norm{\wh{\bw}_h}_{L^2([0,t^M)\times\dom)}^2.\notag
	\end{align}
	for some constant $C>0$ depending on the domain and the parameters $\epsi$, $\epsii$ but not on $h$ or $\Delta t$.
\end{lemma}
\begin{proof}
	We multiply \eqref{seq:dd} by $-\oneconst\Delta_h \bd^{m+\nicefrac{1}{2}}_{\bi}$, \eqref{seq:wd} by $\bw_{\bi}^{m+\nicefrac{1}{2}}$ and sum over $\bi\in\mathring{\mathcal{I}}_N$:
		\begin{align}
	\label{eq:lhs}
	&\sum_{\bi} \big(-\oneconst\Delta_h \bd_{\bi}^{m+\nicefrac{1}{2}}D_t \bd^m_{\bi} +\disp\bw^{m+\nicefrac{1}{2}}_{\bi} D_t \bw^m_{\bi}\big) =  \frac12 D_t \sum_{\bi}\left(\oneconst|\Grad_h\bd_{\bi}^m|^2+\disp |\bw_{\bi}^m|^2\right)\\
	&= \sum_{\bi} \bigg(  -\oneconst\Delta \dhalf_{\bi}\cdot\left(\dhalf_{\bi}\times \bw^{m+1/2}_{\bi}\right) + \bw^{m+\nicefrac{1}{2}}_{\bi}\cdot \left(\oneconst\Delta_h \dhalf_{\bi}\times \dhalf_{\bi}\right)-\damp|\bw_{\bi}^{m+\nicefrac{1}{2}}|^2\notag\\
	&\hphantom{=\sum_{\bi} \bigg(  -}
	+ \bw^{m+\nicefrac{1}{2}}_{\bi}\cdot\frac{\epsi}{2|\mathcal{C}_{\bi}|}\int_{\mathcal{C}_{\bi}} \bigl[(\Grad\vphi^{m+1}_{h}\cdot \dhalf_{\bi}) \Grad\vphi^m_{h}+ (\Grad\vphi^{m}_{h}\cdot \dhalf_{\bi}) \Grad\vphi^{m+1}_{h} \bigr] dx \times \dhalf_{\bi} \bigg)\notag\\
	&=\sum_{\bi} \bigg(  
\frac{\epsi}{2|\mathcal{C}_{\bi}|}\int_{\mathcal{C}_{\bi}}	\bigl[(\Grad\vphi^{m+1}_{h}\cdot \dhalf_{\bi}) \Grad\vphi^m_{h}+ (\Grad\vphi^{m}_{h}\cdot \dhalf_{\bi}) \Grad\vphi^{m+1}_{h} \bigr] dx\cdot (\dhalf_{\bi}\times\bw^{m+\nicefrac{1}{2}}_{\bi})
	-\damp |\bw^{m+\nicefrac{1}{2}}_{\bi}|^2 \bigg)\notag\\
		&=\sum_{\bi} \bigg(  
	\frac{\epsi}{2|\mathcal{C}_{\bi}|}\int_{\mathcal{C}_{\bi}}	\bigl[(\Grad\vphi^{m+1}_{h}\cdot \dhalf_{\bi}) \Grad\vphi^m_{h}+ (\Grad\vphi^{m}_{h}\cdot \dhalf_{\bi}) \Grad\vphi^{m+1}_{h} \bigr] dx\cdot D_t \bd^m_{\bi}
	-\damp |\bw^{m+\nicefrac{1}{2}}_{\bi}|^2 \bigg)\notag
	\end{align}
	where we have used the vector identity $A\cdot (B\times C)=B\cdot (C\times A) =C\cdot (A\times B)$ and in addition equation \eqref{seq:dd} for the last equality.
	Next, we take the average of equation~\eqref{seq:phid} at time step $m$ and $m+1$ and use $D_t u^m_{0,h}\in V_h$ as a test function	
	\begin{equation*}
	a_h^{m+\sfrac{1}{2}}(u_{0,h},D_t u_{0,h}^m)=F^{m+\sfrac{1}{2}}(D_t u_{0,h}^m)-a_h^{m+\sfrac{1}{2}}(\widetilde{g}_h,D_t u_{0,h}^m),
	\end{equation*}
	where we used $m+\sfrac{1}{2}$ to denote the averages of the respective quantities at time levels $m$ and $m+1$. We can rewrite this as
	\begin{equation*}
	a_h^{m+\sfrac{1}{2}}(\vphi_h,D_t(\vphi^m_h-\widetilde{g}_h^{m})) = F^{m+\sfrac{1}{2}}(D_t (\vphi^m_h-\widetilde{g}_h^{m})),
	\end{equation*}
	since $u_{0,h}=\vphi_h-\widetilde{g}_h$. Inserting the definitions of the linear and bilinear forms $F^m$ and $a^m$, this is in fact
	\begin{multline}\label{eq:elliptic}
	\frac12 D_t\int_{\dom}|\Grad \vphi_h^m|^2 dx - \int_{\dom} \Grad\vphi_h^{m+\sfrac{1}{2}}  D_t \Grad\widetilde{g}_h^m dx\\
	 + \frac{\epsii}{2}\int_{\dom} (\bd_h^m\cdot \Grad\vphi_h^m)(\bd_h^m\cdot\Grad D_t\vphi_h^m) dx + \frac{\epsii}{2}\int_{\dom} (\bd_h^{m+1}\cdot \Grad\vphi_h^{m+1})(\bd_h^{m+1}\cdot\Grad D_t\vphi_h^m) dx\\
	-\frac{\epsii}{2}\int_{\dom} (\bd_h^m\cdot \Grad\vphi_h^m)(\bd_h^m\cdot D_t\Grad\widetilde{g}_h^m) dx  -\frac{\epsii}{2}\int_{\dom} (\bd_h^{m+1}\cdot \Grad\vphi_h^{m+1})(\bd_h^{m+1}\cdot D_t\Grad\widetilde{g}_h^m) dx\\
	= \frac12 D_t\int_{\dom} f(t^m)\vphi_h^m dx - \int_{\dom} D_t f(t^m)\vphi_h^{m+\sfrac{1}{2}} dx -\int_{\dom}\frac12(f(t^m)+f(t^{m+1}))  D_t\widetilde{g}_h^m dx.
	\end{multline}
	We multiply this identity by $-\epsi/\epsii$ and multiply~\eqref{eq:lhs} by $h^n$, use the definitions~\eqref{eq:pwconstd}-\eqref{eq:whhat}, and then add them up:
	\begin{multline}\label{eq:monster}
	\frac{1}{2} D_t \int_{\dom}\left(\oneconst|\Grad_h\bd_{h}^m|^2+\disp |\bw_{h}^m|^2-\frac{\epsi}{\epsii}|\Grad\vphi_h^m|^2\right) dx \\
	= -\damp\int_{\dom}|{\bw}_{h}^{m+\sfrac{1}{2}}|^2  dx +
	\frac{\epsi}{2}\int_{\dom}	\bigl[(\Grad\vphi^{m+1}_{h}\cdot \dhalf_{h}) \Grad\vphi^m_{h}+ (\Grad\vphi^{m}_{h}\cdot \dhalf_{h}) \Grad\vphi^{m+1}_{h}  \bigr]  D_t \bd^m_{h} dx \\
	+\frac{\epsi}{2}\int_{\dom} (\bd_h^m\cdot \Grad\vphi_h^m)(\bd_h^m\cdot\Grad D_t\vphi_h^m) dx + \frac{\epsi}{2}\int_{\dom} (\bd_h^{m+1}\cdot \Grad\vphi_h^{m+1})(\bd_h^{m+1}\cdot\Grad D_t\vphi_h^m) dx\\
	-\frac{\epsi}{2}\int_{\dom} (\bd_h^m\cdot \Grad\vphi_h^m)(\bd_h^m\cdot D_t\Grad\widetilde{g}_h^m) dx
	  -\frac{\epsi}{2}\int_{\dom} (\bd_h^{m+1}\cdot \Grad\vphi_h^{m+1})(\bd_h^{m+1}\cdot D_t\Grad\widetilde{g}_h^m) dx	\\
		- \frac{\epsi}{\epsii}\int_{\dom} \Grad\vphi_h^{m+\sfrac{1}{2}}  D_t \Grad\widetilde{g}_h^m dx	 -\frac{\epsi}{2\epsii} D_t\int_{\dom} f(t^m)\vphi_h^m dx\\
		+ \frac{\epsi}{\epsii}\int_{\dom} D_t f(t^m)\vphi_h^{m+\sfrac{1}{2}} dx +\frac{\epsi}{2\epsii}\int_{\dom}(f(t^m)+f(t^{m+1}))  D_t\widetilde{g}_h^m dx
	\end{multline}
	Manipulating the terms in the middle, we see:
	\begin{multline*}
		\frac{\epsi}{2}\int_{\dom}	\bigl[(\Grad\vphi^{m+1}_{h}\cdot \dhalf_{h}) \Grad\vphi^m_{h}+ (\Grad\vphi^{m}_{h}\cdot \dhalf_{h}) \Grad\vphi^{m+1}_{h}  \bigr]  D_t \bd^m_{h} dx \\
	+ \frac{\epsi}{2}\int_{\dom} (\bd_h^m\cdot \Grad\vphi_h^m)(\bd_h^m\cdot\Grad D_t\vphi_h^m) dx + \frac{\epsi}{2}\int_{\dom} (\bd_h^{m+1}\cdot \Grad\vphi_h^{m+1})(\bd_h^{m+1}\cdot\Grad D_t\vphi_h^m) dx\\
		=	\frac{\epsi}{4\Delta t}\int_{\dom}	\bigl[(\Grad\vphi^{m+1}_{h}\cdot \bd^m_{h}) (\Grad\vphi^m_{h}\cdot\bd^{m+1}_{h})+ (\Grad\vphi^{m}_{h}\cdot \bd^m_{h}) (\Grad\vphi^{m+1}_{h} \cdot\bd^{m+1}_{h}) \bigr]   dx \\
+\frac{\epsi}{4\Delta t}\int_{\dom}	\bigl[(\Grad\vphi^{m+1}_{h}\cdot \bd^{m+1}_{h}) (\Grad\vphi^m_{h}\cdot\bd^{m+1}_{h})+ (\Grad\vphi^{m}_{h}\cdot \bd^{m+1}_{h}) (\Grad\vphi^{m+1}_{h} \cdot\bd^{m+1}_{h}) \bigr]   dx \\
	-	\frac{\epsi}{4\Delta t}\int_{\dom}	\bigl[(\Grad\vphi^{m+1}_{h}\cdot \bd^m_{h}) (\Grad\vphi^m_{h}\cdot\bd^{m}_{h})+ (\Grad\vphi^{m}_{h}\cdot \bd^m_{h}) (\Grad\vphi^{m+1}_{h} \cdot\bd^{m}_{h}) \bigr]   dx\\
	-	\frac{\epsi}{4\Delta t}\int_{\dom}	\bigl[(\Grad\vphi^{m+1}_{h}\cdot \bd^{m+1}_{h}) (\Grad\vphi^m_{h}\cdot\bd^{m}_{h})+ (\Grad\vphi^{m}_{h}\cdot \bd^{m+1}_{h}) (\Grad\vphi^{m+1}_{h} \cdot\bd^{m}_{h}) \bigr]   dx\\
	+\frac{\epsi}{2\Delta t}\int_{\dom} (\bd_h^m\cdot \Grad\vphi_h^m)(\bd_h^m\cdot\Grad \vphi_h^{m+1}) dx 
	-\frac{\epsi}{2\Delta t}\int_{\dom} (\bd_h^m\cdot \Grad\vphi_h^m)(\bd_h^m\cdot\Grad \vphi_h^m) dx \\
		+ \frac{\epsi}{2\Delta t}\int_{\dom} (\bd_h^{m+1}\cdot \Grad\vphi_h^{m+1})(\bd_h^{m+1}\cdot\Grad \vphi_h^{m+1}) dx
	- \frac{\epsi}{2\Delta t}\int_{\dom} (\bd_h^{m+1}\cdot \Grad\vphi_h^{m+1})(\bd_h^{m+1}\cdot\Grad \vphi_h^m) dx\\
	=\frac{\epsi}{2}D_t\int_{\dom} (\bd_h^m\cdot \Grad\vphi_h^m)^2 dx. 
	\end{multline*}
	Hence,~\eqref{eq:monster} becomes
	\begin{multline}\label{eq:monster2}
	\frac{1}{2} D_t \int_{\dom}\left(\oneconst|\Grad_h\bd_{h}^m|^2+\disp |\bw_{h}^m|^2-\frac{\epsi}{\epsii}|\Grad\vphi_h^m|^2-\epsi (\bd^m_h\cdot\Grad\vphi_h^m)^2+\frac{\epsi}{\epsii}  f(t^m)\vphi_h^m \right) dx \\
	= -\damp\int_{\dom}|{\bw}_{h}^{m+\sfrac{1}{2}}|^2  dx	- \frac{\epsi}{\epsii}\int_{\dom} \Grad\vphi_h^{m+\sfrac{1}{2}}  D_t \Grad\widetilde{g}_h^m dx \\
	-\frac{\epsi}{2}\int_{\dom} (\bd_h^m\cdot \Grad\vphi_h^m)(\bd_h^m\cdot D_t\Grad\widetilde{g}_h^m) dx
	-\frac{\epsi}{2}\int_{\dom} (\bd_h^{m+1}\cdot \Grad\vphi_h^{m+1})(\bd_h^{m+1}\cdot D_t\Grad\widetilde{g}_h^m) dx	\\
	+\frac{\epsi}{\epsii}\int_{\dom} D_t f(t^m)\vphi_h^{m+\sfrac{1}{2}} dx +\frac{\epsi}{2\epsii}\int_{\dom}(f(t^m)+f(t^{m+1}))  D_t\widetilde{g}_h^m dx.
	\end{multline}
	Next, we use $u^m_{0,h}=\vphi^m_h-\widetilde{g}^m_h\in V_h$ as a test function in~\eqref{seq:phid}:
	\begin{equation*}
	a_h^m(\vphi_h^m,\vphi_h^m-\widetilde{g}^m_h) = F^m(\vphi_h^m-\widetilde{g}^m_h),
	\end{equation*}
	which is
	\begin{multline*}
	\int_{\dom}\left[|\Grad\vphi^m_h|^2+\epsii (\bd^m_h\cdot\Grad\vphi_h^m)^2\right] dx - \int_{\dom} (\Grad\vphi_h^m+\epsii(\bd_h^m\cdot\Grad\vphi_h^m)\bd_h^m)\cdot\Grad\widetilde{g}^m_h dx \\
	=\int_{\dom}f(t^m)\vphi_h^m dx -\int_{\dom}f(t^m)\widetilde{g}^m_h dx.
	\end{multline*}
	Taking the time difference $D_t$ of this identity and multiplying by $(1+\epsi/(2\epsii))$, we have
		\begin{multline*}
	\left(1+\frac{\epsi}{2\epsii}\right)D_t\int_{\dom}\left[|\Grad\vphi^m_h|^2+\epsii (\bd^m_h\cdot\Grad\vphi_h^m)^2\right] dx\\
	 - \left(1+\frac{\epsi}{2\epsii}\right)D_t\int_{\dom} (\Grad\vphi_h^m+\epsii(\bd_h^m\cdot\Grad\vphi_h^m)\bd_h^m)\cdot\Grad\widetilde{g}^m_h dx\\ =\left(1+\frac{\epsi}{2\epsii}\right)D_t\int_{\dom}f(t^m)\vphi_h^m dx -\left(1+\frac{\epsi}{2\epsii}\right)D_t\int_{\dom}f(t^m)\widetilde{g}^m_h dx.
	\end{multline*}
	We add this to~\eqref{eq:monster2}:
		\begin{multline}\label{eq:monster3}
	\frac{1}{2} D_t \int_{\dom}\left(\oneconst|\Grad_h\bd_{h}^m|^2+\disp |\bw_{h}^m|^2+2|\Grad\vphi_h^m|^2+2\epsii (\bd^m_h\cdot\Grad\vphi_h^m)^2-2 f(t^m)\vphi_h^m \right) dx \\
	= -\damp\int_{\dom}|{\bw}_{h}^{m+\sfrac{1}{2}}|^2  dx	- \frac{\epsi}{\epsii}\int_{\dom} \Grad\vphi_h^{m+\sfrac{1}{2}}  D_t \Grad\widetilde{g}_h^m dx \\
	-\frac{\epsi}{2}\int_{\dom} (\bd_h^m\cdot \Grad\vphi_h^m)(\bd_h^m\cdot D_t\Grad\widetilde{g}_h^m) dx
	-\frac{\epsi}{2}\int_{\dom} (\bd_h^{m+1}\cdot \Grad\vphi_h^{m+1})(\bd_h^{m+1}\cdot D_t\Grad\widetilde{g}_h^m) dx	\\
	+ \frac{\epsi}{\epsii}\int_{\dom} D_t f(t^m)\vphi_h^{m+\sfrac{1}{2}} dx -\frac{\epsi}{\epsii}\int_{\dom}D_t f(t^m) \widetilde{g}_h^{m+\sfrac{1}{2}} dx\\
	 + \left(1+\frac{\epsi}{2\epsii}\right)D_t\int_{\dom} (\Grad\vphi_h^m+\epsii(\bd_h^m\cdot\Grad\vphi_h^m)\bd_h^m)\cdot\Grad\widetilde{g}^m_h dx
	 -D_t\int_{\dom}f(t^m)\widetilde{g}^m_h dx
	\end{multline}
	Now, we multiply the last identity by $\Delta t$ and sum over $m=0,1,2,\dots, M-1$:
	
		\begin{align*}
	&\frac{1}{2}  \int_{\dom}\left(\oneconst|\Grad_h\bd_{h}^M|^2+\disp |\bw_{h}^M|^2+2|\Grad\vphi_h^M|^2+2\epsii (\bd^M_h\cdot\Grad\vphi_h^M)^2-2 f(t^M)\vphi_h^M \right) dx \\
	&-	 \left(1+\frac{\epsi}{2\epsii}\right)\int_{\dom} (\Grad\vphi_h^M+\epsii(\bd_h^M\cdot\Grad\vphi_h^M)\bd_h^M)\cdot\Grad\widetilde{g}^M_h dx\\
	&= 	\frac{1}{2}  \int_{\dom}\left(\oneconst|\Grad_h\bd_{h}^0|^2+\disp |\bw_{h}^0|^2+2|\Grad\vphi_h^0|^2+2\epsii (\bd^0_h\cdot\Grad\vphi_h^0)^2-2 f(0)\vphi_h^0 \right) dx\\
	&-\Dt\damp\sum_{m=0}^{M-1}\int_{\dom}|{\bw}_{h}^{m+\sfrac{1}{2}}|^2  dx	- \frac{\epsi}{\epsii}\Dt\sum_{m=0}^{M-1}\int_{\dom} \Grad\vphi_h^{m+\sfrac{1}{2}}  D_t \Grad\widetilde{g}_h^m dx \\
&	-\frac{\epsi}{2}\Dt\sum_{m=0}^{M-1}\int_{\dom} (\bd_h^m\cdot \Grad\vphi_h^m)(\bd_h^m\cdot D_t\Grad\widetilde{g}_h^m) dx
	-\frac{\epsi}{2}\Dt\sum_{m=0}^{M-1}\int_{\dom} (\bd_h^{m+1}\cdot \Grad\vphi_h^{m+1})(\bd_h^{m+1}\cdot D_t\Grad\widetilde{g}_h^m) dx	\\
&	+ \frac{\epsi}{\epsii}\Dt\sum_{m=0}^{M-1}\int_{\dom} D_t f(t^m)\vphi_h^{m+\sfrac{1}{2}} dx -\frac{\epsi}{\epsii}\Dt\sum_{m=0}^{M-1}\int_{\dom}D_t f(t^m) \widetilde{g}_h^{m+\sfrac{1}{2}} dx\\
&	-\int_{\dom}f(t^M)\widetilde{g}^M_h dx+ \left(1+\frac{\epsi}{2\epsii}\right)\int_{\dom} (\Grad\vphi_h^0+\epsii(\bd_h^0\cdot\Grad\vphi_h^0)\bd_h^0)\cdot\Grad\widetilde{g}^0_h dx
	+\int_{\dom}f(0)\widetilde{g}^0_h dx.
	\end{align*}
	We use H\"{o}lder and Cauchy-Schwarz inequality and the constraint $|\bd^m_h|=1$ a couple of times to estimate the right hand side:
		\begin{align*}
	&\frac{1}{2}  \int_{\dom}\left(\oneconst|\Grad_h\bd_{h}^M|^2+\disp |\bw_{h}^M|^2+2|\Grad\vphi_h^M|^2+2\epsii (\bd^M_h\cdot\Grad\vphi_h^M)^2-2 f(t^M)\vphi_h^M \right) dx \\
	&-	 \left(1+\frac{\epsi}{2\epsii}\right)\int_{\dom} (\Grad\vphi_h^M+\epsii(\bd_h^M\cdot\Grad\vphi_h^M)\bd_h^M)\cdot\Grad\widetilde{g}^M_h dx\\
	&\leq 	\frac{1}{2}  \left(\oneconst\norm{\Grad\bd_0}^2_{L^2(\dom)}+\disp \norm{\bw_0}_{L^2(\dom)}^2+2(1+\max\{0,\epsii\})\norm{\Grad\vphi_h^0}^2_{L^2(\dom)}+\norm{f(0)}_{L^2}^2+\norm{\vphi_h^0}_{L^2}^2\right)\\
	&-\Dt\damp\sum_{m=0}^{M-1}\int_{\dom}|{\bw}_{h}^{m+\sfrac{1}{2}}|^2  dx	+ \left(\left|\frac{\epsi}{\epsii}\right|+|\epsi|\right)\Dt\left(\sum_{m=0}^{M}\norm{\Grad\vphi_h^m}_{L^2}^2+\sum_{m=0}^{M-1}\norm{ D_t \Grad\widetilde{g}_h^m }_{L^2}^2\right) \\
	&	+\left|\frac{\epsi}{\epsii}\right|\Dt\left(\sum_{m=0}^{M-1}\norm{D_t f(t^m)}_{L^2}^2+ \frac12\sum_{m=0}^{M}\norm{\widetilde{g}_h^{m}}_{L^2}^2+\frac12\sum_{m=0}^{M}\norm{\vphi_h^{m}}_{L^2}^2\right)\\
	&	+\frac12\norm{f(t^M)}_{L^2}^2+\frac12\norm{\widetilde{g}^M_h}_{L^2}^2+ \left(\frac12+\left|\frac{\epsi}{4\epsii}\right|\right)\left(\norm{\Grad\vphi_h^0}_{L^2}^2+\norm{\Grad\widetilde{g}^0_h}_{L^2}^2\right)
	+\frac12\norm{f(0)}_{L^2}^2+\frac12\norm{\widetilde{g}^0_h}^2_{L^2}\\
	&\leq 
		\frac{\oneconst}{2}\norm{\Grad\bd_0}^2_{L^2(\dom)}+\frac{\disp}{2} \norm{\bw_0}_{L^2(\dom)}^2+
		C\norm{\Grad\vphi_h^0}^2_{L^2(\dom)}+C\norm{f(0)}_{L^2}^2+\frac12\norm{\vphi_h^0}_{L^2}^2+C\norm{\widetilde{g}_h^0}_{H^1}^2\\
	&-\Dt\damp\sum_{m=0}^{M-1}\int_{\dom}|{\bw}_{h}^{m+\sfrac{1}{2}}|^2  dx	+ \left(\left|\frac{\epsi}{\epsii}\right|+|\epsi|\right)\Dt\left(\sum_{m=0}^{M}\norm{\Grad\vphi_h^m}_{L^2}^2+\sum_{m=0}^{M-1}\norm{ D_t \Grad\widetilde{g}_h^m }_{L^2}^2\right) \\
	&	+\left|\frac{\epsi}{\epsii}\right|\Dt\left(\sum_{m=0}^{M-1}\norm{D_t f(t^m)}_{L^2}^2+ \frac12\sum_{m=0}^{M}\norm{\widetilde{g}_h^{m}}_{L^2}^2+\frac12\sum_{m=0}^{M}\norm{\vphi_h^{m}}_{L^2}^2\right)\\
	&	+\frac12\norm{f(t^M)}_{L^2}^2+\frac12\norm{\widetilde{g}^M_h}_{L^2}^2.
	\end{align*}
	Using Poincar\'e's inequality, this becomes
		\begin{align}\label{eq:monster4}
	&\frac{1}{2}  \int_{\dom}\left(\oneconst|\Grad_h\bd_{h}^M|^2+\disp |\bw_{h}^M|^2+2|\Grad\vphi_h^M|^2+2\epsii (\bd^M_h\cdot\Grad\vphi_h^M)^2-2 f(t^M)\vphi_h^M \right) dx \\
	&-	 \left(1+\frac{\epsi}{2\epsii}\right)\int_{\dom} (\Grad\vphi_h^M+\epsii(\bd_h^M\cdot\Grad\vphi_h^M)\bd_h^M)\cdot\Grad\widetilde{g}^M_h dx\notag\\
	&\leq 	C\left(\norm{\Grad\bd_0}^2_{L^2(\dom)}+ \norm{\bw_0}_{L^2(\dom)}^2+
	\norm{\Grad\vphi_h^0}^2_{L^2(\dom)}+\norm{f(0)}_{L^2}^2+\norm{\widetilde{g}_h^0}^2_{H^1(\dom)}\right)\notag\\
	&-\Dt\damp\sum_{m=0}^{M-1}\int_{\dom}|{\bw}_{h}^{m+\sfrac{1}{2}}|^2  dx	+ C\Dt\left(\sum_{m=0}^{M}\norm{\Grad\vphi_h^m}_{L^2}^2\right)+\frac12 \norm{\widetilde{g}^M_h}_{L^2(\dom)}^2\notag \\
	&	+C\Dt\left(\sum_{m=0}^{M-1}\norm{D_t f(t^m)}_{L^2}^2+ \sum_{m=0}^{M}\norm{\widetilde{g}_h^{m}}_{H^1}^2+\sum_{m=0}^{M-1}\norm{ D_t \Grad\widetilde{g}_h^m }_{L^2}^2\right)+\frac12\norm{f(t^M)}_{L^2}^2.\notag	
	\end{align}
	Now for the left hand side, we have
	\begin{equation*}
	\int_{\dom}\left(|\Grad\vphi_h^M|^2+\epsii (\bd^M_h\cdot\Grad\vphi_h^M)^2\right) dx\geq (1+\min\{0,\epsii\})\norm{\Grad\vphi_h^M}_{L^2}^2,
	\end{equation*}
	and, using Poincar\'e's inequality for $u_{0,h}$, then for arbitrary $\alpha_1,\alpha_2>0$,
	\begin{equation*}
	\begin{split}
	\left|\int_{\dom} f(t^M)\vphi^M_h dx\right|&\leq C(\dom)\norm{f(t^M)}_{L^2}\left(\norm{\widetilde{g}^M_h}_{H^1}+\norm{\Grad\vphi^M_h}_{L^2}\right)\\
	& \leq  \frac{C(\dom)^2}{2\alpha_1}\norm{f(t^M)}_{L^2}^2+\frac{\alpha_1}{2}\norm{\Grad\vphi^M_h}^2_{L^2}+C(\dom)\norm{f(t^M)}_{L^2}\norm{\widetilde{g}^M_h}_{H^1}
	\end{split}
	\end{equation*}
	and
	\begin{equation*}
	\begin{split}
\left|\left(1+\frac{\epsi}{2\epsii}\right)\int_{\dom} (\Grad\vphi_h^M+\epsii(\bd_h^M\cdot\Grad\vphi_h^M)\bd_h^M)\cdot\Grad\widetilde{g}^M_h dx\right|	& \leq C(\epsi,\epsii)\norm{\Grad\vphi_h^M}_{L^2}\norm{\Grad\widetilde{g}^M_h}_{L^2}\\
	&\leq \frac{\alpha_2}{2}\norm{\Grad\vphi_h^M}_{L^2}^2+\frac{C(\epsi,\epsii)^2}{2\alpha_2}\norm{\Grad\widetilde{g}^M_h}_{L^2}^2
	\end{split}
	\end{equation*}
	Picking $\alpha_1=\alpha_2= 0.5(1+\min\{0,\epsii\})$, we can estimate the left hand side of~\eqref{eq:monster4}
\begin{align}
\label{eq:forfaen}
&\frac{1}{2}  \int_{\dom}\left(\oneconst|\Grad_h\bd_{h}^M|^2+\disp |\bw_{h}^M|^2+2|\Grad\vphi_h^M|^2+2\epsii (\bd^M_h\cdot\Grad\vphi_h^M)^2-2 f(t^M)\vphi_h^M \right) dx \\
&\qquad-	 \left(1+\frac{\epsi}{2\epsii}\right)\int_{\dom} (\Grad\vphi_h^M+\epsii(\bd_h^M\cdot\Grad\vphi_h^M)\bd_h^M)\cdot\Grad\widetilde{g}^M_h dx\notag\\
&\geq \frac{\oneconst}{2}\norm{\Grad_h\bd_h^M}_{L^2}^2+\frac{\disp}{2}\norm{\bw^M_h}_{L^2}^2+\frac{1+\min\{0,\epsii\}}{2}\norm{\Grad\vphi^M_h}_{L^2}^2-C(\dom)\norm{f(t^M)}_{L^2}\norm{\widetilde{g}^M_h}_{H^1}\notag\\
&\qquad-\frac{C(\dom)^2}{1+\min\{0,\epsii\}}\norm{f(t^M)}_{L^2}^2 -\frac{C(\epsi,\epsii)^2}{1+\min\{0,\epsii\}}\norm{\Grad\widetilde{g}^M_h}_{L^2}^2.\notag
\end{align}
	Combining this with~\eqref{eq:monster4}, we find
		\begin{align}\label{eq:monster5}
	&\frac{\oneconst}{2}\norm{\Grad_h\bd_h^M}_{L^2}^2+\frac{\disp}{2}\norm{\bw^M_h}_{L^2}^2+\frac{1+\min\{0,\epsii\}}{2}\norm{\Grad\vphi^M_h}_{L^2}^2\\
	&\leq
	C(\dom)\norm{f(t^M)}_{L^2}\norm{\widetilde{g}^M_h}_{H^1}+\frac12\norm{f(t^M)}_{L^2}^2+\frac12\norm{\widetilde{g}^M_h}_{L^2(\dom)}^2\notag\\
	&+\frac{C(\dom)^2}{1+\min\{0,\epsii\}}\norm{f(t^M)}_{L^2}^2 +\frac{C(\epsi,\epsii)^2}{1+\min\{0,\epsii\}}\norm{\Grad\widetilde{g}^M_h}_{L^2}^2\notag\\
	&+	C\left(\norm{\Grad\bd_0}^2_{L^2(\dom)}+ \norm{\bw_0}_{L^2(\dom)}^2+
	\norm{\Grad\vphi_h^0}^2_{L^2(\dom)}+\norm{f(0)}_{L^2}^2+\norm{\widetilde{g}_h^0}^2_{H^1(\dom)}\right)\notag\\
	&-\Dt\damp\sum_{m=0}^{M-1}\int_{\dom}|{\bw}_{h}^{m+\sfrac{1}{2}}|^2  dx	+ C\Dt\left(\sum_{m=0}^{M}\norm{\Grad\vphi_h^m}_{L^2}^2\right)\notag \\
	&	+C\Dt\left(\sum_{m=0}^{M-1}\norm{D_t f(t^m)}_{L^2}^2+ \sum_{m=0}^{M}\norm{\widetilde{g}_h^{m}}_{H^1}^2+\sum_{m=0}^{M-1}\norm{ D_t \Grad\widetilde{g}_h^m }_{L^2}^2\right).\notag	
	\end{align}
	This yields an estimate of the form
	\begin{equation*}
	V(t^M)\leq A(t^M)+\Dt\sum_{m=0}^{M-1} B(t^m)V(t^m),
	\end{equation*}
	where
	\begin{equation*}
	V(t^m)=\frac{\oneconst}{2}\norm{\Grad_h\bd_h^M}_{L^2}^2+\frac{\disp}{2}\norm{\bw^M_h}_{L^2}^2+
	\norm{\Grad\vphi^M_h}_{L^2}^2,
	\end{equation*}
	and
	\begin{equation*}
	B(t^m)= \frac{2C}{\min\{1+\min\{0,\epsii\}-2C\Dt,\sfrac{1}{2}\}}
	\end{equation*}
	(assume $\Dt\leq 1+\min\{0,\epsii\}/(2C)$)
	and 
	\begin{align*}
	A(t^M)&=	C(\dom,\epsi,\epsii)\norm{\widetilde{g}^M_h}_{H^1}^2+C(\dom,\epsi,\epsii)\norm{f(t^M)}_{L^2}^2\\
	&+	C\left(\norm{\Grad\bd_0}^2_{L^2(\dom)}+ \norm{\bw_0}_{L^2(\dom)}^2+
	\norm{\Grad\vphi_h^0}^2_{L^2(\dom)}+\norm{f(0)}_{L^2}^2+\norm{\widetilde{g}_h^0}^2_{H^1(\dom)}\right)\\
	&	+C\Dt\left(\sum_{m=0}^{M-1}\norm{D_t f(t^m)}_{L^2}^2+ \sum_{m=0}^{M}\norm{\widetilde{g}_h^{m}}_{H^1}^2+\sum_{m=0}^{M-1}\norm{ D_t \Grad\widetilde{g}_h^m }_{L^2}^2\right).	
	\end{align*}
	Using the `discrete Gr\"onwall inequality' (e.g.~\cite{Holte2009}),
	\begin{equation}\label{eq:Gronwalldisc}
	V(t^M)\leq A(t^M)+\Delta t\sum_{m=0}^{M-1}A(t^m) B(t^m)\exp\left(\Delta t\sum_{k=m+1}^{M-1}B(t^k)\right),
	\end{equation}
	we obtain
	\begin{multline*}
	\frac{\oneconst}{2}\norm{\Grad_h\bd_h^M}_{L^2}^2+\frac{\disp}{2}\norm{\bw^M_h}_{L^2}^2+
	\norm{\Grad\vphi^M_h}_{L^2}^2\\
	 \leq A(t^M)+\frac{2C\Dt}{\min\{1+\min\{0,\epsii\}-2C\Dt,\sfrac{1}{2}\}}\sum_{m=0}^{M-1}A(t^m)\exp\left(\frac{2C\Dt (M-1-m)}{\min\{1+\min\{0,\epsii\}-2C\Dt,\sfrac{1}{2}\}}\right)\\
	\leq A(t^M)+C\exp(C t^M)\Dt\sum_{m=0}^{M-1}A(t^m),
	\end{multline*}
	with the expression for $A(t^m)$ given above. Hence we obtain the estimate
	\begin{multline*}
		\frac{\oneconst}{2}\norm{\Grad_h\bd_h^M}_{L^2}^2+\frac{\disp}{2}\norm{\bw^M_h}_{L^2}^2+
	\norm{\Grad\vphi^M_h}_{L^2}^2\\
	\leq	
	C\left(\norm{\widetilde{g}^M_h}_{H^1}^2+\norm{f(t^M)}_{L^2}^2\right)+C\exp(Ct^M)\Dt\sum_{m=0}^{M-1}\norm{f(t^m)}_{L^2}^2\\
	+C(t^M\exp(Ct^M)+1)\left(\norm{\Grad\bd_0}^2_{L^2(\dom)}+ \norm{\bw_0}_{L^2(\dom)}^2
	+	\norm{\Grad\vphi_h^0}^2_{L^2(\dom)}+\norm{f(0)}_{L^2}^2+\norm{\widetilde{g}_h^0}^2_{H^1(\dom)}\right)\\
	+C(1+t^M\exp(C t^M))\Dt\left(\sum_{m=0}^{M-1}\norm{D_t f(t^m)}_{L^2}^2+\sum_{m=0}^{M-1}\norm{D_t\Grad\widetilde{g}^m_h}^2_{L^2}\right)\\
		+ C\exp(Ct^M)(1+t^M)\Dt\sum_{m=0}^M \norm{\widetilde{g}^m_h}^2_{H^1}.
	\end{multline*}
Comparing~\eqref{eq:monster5}, we also obtain the same bound on the $L^2_{t,x}$-norm of $\wh{\bw}_h$.
which proves the result.
\end{proof}
\section{Convergence}\label{sec:conv}
Using the discrete energy estimate~\eqref{eq:discreteenergyestimate}, we now proceed to showing convergence of the scheme. We assume that the initial data satisfies
\begin{equation}\label{eq:initdata}
\bd_0\in H^1(\dom),\quad |\bd_0(x)|=1,\quad \text{a.e.}\, x\in \dom, \quad\partial_t\bd(0,\cdot)=\bd_1\in L^2(\dom),
\end{equation}
and define $\bw_0:=\bd_1\times\bd_0$ and set $\bw_h(0,\cdot)=\bw_0$. This implies for the piecewise constant functions~\eqref{eq:pwconstd}--\eqref{eq:whhat}
\begin{equation}\label{eq:discreteinit}
\norm{\bd_h(0,\cdot)}_{H^1(\dom)}\leq C, \quad \norm{\bw_h(0,\cdot)}_{L^2(\dom)}\leq C,\quad |\bd_h(0,x)|=1,\quad x\in \dom.
\end{equation}
Moreover, we assume that
\begin{equation}
\label{eq:CFL}
\Dt\leq C h,\quad \text{for some constant}\, C>0. 
\end{equation}
Then we can show
\begin{theorem}
	Assume $\oneconst>0$, $\disp>0$ and $\damp\geq 0$, and $\epsii>-1$. Assume $f\in W^{1,2}([0,T];L^2(\dom))$ and $g\in C^1([0,T]\times\partial\dom)$ which can be extended to a function $\widetilde{g}\in L^\infty(0,T;H^1(\dom))\cap W^{1,2}([0,T];H^1(\dom))$. Moreover, assume that the initial data $(\bd_0,\bw_0)$ satisfies~\eqref{eq:initdata} and that the time step $\Dt$ and the grid size $h$ are related by~\eqref{eq:CFL}. Then, as $h,\Dt\to 0$, the approximations $(\bd_h,\bw_h,\vphi_h)$ as defined in~\eqref{eq:pwconstd}--\eqref{eq:pwconstphi} converge up to a subsequence to a weak solution of~\eqref{eq:model1} as in Definition~\ref{def:weaksol}. 
\end{theorem}
\begin{proof}
	Using Lemma~\ref{lem:constraint} and Lemma~\ref{lem:discreteenergy}, the Poincar\'{e} inequality for $u_{0,h}$, and that $\norm{\widetilde{g}_h(t)}_{H^1}\leq C\norm{\widetilde{g}(t)}_{H^1}\leq C$ uniformly in $h$, 
 we obtain the following uniform in $h>0$ bounds for all $t\in [0,T]$, 
\begin{subequations}
	\label{eq:uniformbounds}
	\begin{align}
	&\norm{\bd_h(t)}_{L^2}\leq C,\quad \norm{\Grad_h \bd_h}_{L^2}\leq C,\label{eq:waek}\\
	&|\bd_h(t,x)|=1,\quad (t,x)\in [0,T]\times\dom,\\
	&\norm{\bw_h(t)}_{L^2}\leq C,\\
	&\norm{\vphi_h(t)}_{L^2}\leq C,\quad \norm{\Grad \vphi_h}_{L^2}\leq C.
	\end{align}
\end{subequations}
Using the numerical scheme, we also get a priori estimates on the forward time differences of $\bd_h$:
Since $D_t^+\bd^m_{\bi} = \bd_{\bi}^{m+\nicefrac{1}{2}}\times \bw_{\bi}^{m+\nicefrac{1}{2}}$, and $|\bd_{\bi}^m|=1$ by Lemma~\ref{lem:constraint}, we obtain,
\begin{equation}\label{eq:timederbounded}
\norm{D_t^+ \bd_h(t)}_{L^2}\leq C\norm{\wh{\bw}_h(t)}_{L^2}\leq C.
\end{equation}
The Banach-Alaoglu theorem then implies the convergence of subsequences, for simplicity of notation denoted by $\{(\bd_h,\bw_h,\vphi_h)\}_{h>0}$,
\begin{subequations}\label{eq:weakconv}
\begin{align}
	&\bd_h\weakstar \bd,\quad\text{in }\,  L^\infty([0,T);L^2(\dom)),\quad \wh{\bd}_h\weakstar \bd,\quad\text{in }\,  L^\infty([0,T);L^2(\dom)),\\
	&D_t^+\bd\weakstar \partial_t \bd,\quad\text{in }\,  L^\infty([0,T);L^2(\dom)),\quad \Grad_h\bd_h\weakstar\Grad\bd,\quad\text{in }\,  L^\infty([0,T);L^2(\dom)),\\
	&\bw_h\weakstar \bw,\quad\text{in }\, L^\infty([0,T);L^2(\dom)),\quad \wh{\bw}_h\weakstar \wh{\bw},\quad\text{in }\, L^\infty([0,T);L^2(\dom)),\\
		&\vphi_h\weakstar \vphi,\quad\text{in }\, L^\infty([0,T);H^1(\dom)),\quad u_{0,h}\weakstar u_0,\quad\text{in }\, L^\infty([0,T);H^1(\dom))
		\label{eq:weakphi}
\end{align}
\end{subequations}
as $h\to 0$. The second item follows since
\begin{equation*}
\wh{\bd}_h(t,x) = \bd_h(t,x) +\frac{1}{2}\Dt D_t^+\bd(t,x)),
\end{equation*}
and so by~\eqref{eq:timederbounded}, 
\begin{equation*}
\int_{\dom} |\bd_h(t,x)-\wh{\bd}_h(t,x)|^2 dx \leq \Dt^2 \norm{D_t^+\bd_h(t,\cdot)}_{L^2}^2\to 0,\quad \text{as }\, \Dt,h\to 0.
\end{equation*}
The uniform bounds~\eqref{eq:waek} and~\eqref{eq:timederbounded} in fact imply pre-compactness of the sequence $\{\bd_h\}_{h>0}$ in $L^2$, (see e.g. Chapter 6 in~\cite{lady}):
\begin{equation}\label{eq:strongconv}
\bd_h\rightarrow\bd,\quad\text{in }\, L^2([0,T]\times\dom).
\end{equation}
Therefore a subsequence of $\{\bd_h\}_{h>0}$ converges almost everywhere and so $|\bd|=1$ for almost every $(t,x)\in [0,T]\times\dom$. Since $\bd_h$ and $\wh{\bd}_h$ have the same strong limit in $L^2$ and almost everywhere, their weak gradients have the same limits too.
We write the scheme~\eqref{eq:m2} in terms of the piecewise constant functions $\bd_h$, $\bw_h$, etc.:
\begin{subequations}\label{eq:m2h}
	\begin{align}
	\label{seq:ddh}
	&D_t \bd_h=\wh{\bd}_{h}\times \wh{\bw}_{h},\\
	\label{seq:wdh}
	&\disp D_t \bw_h=\oneconst\Delta_h \wh{\bd}_{h}\times \wh{\bd}_{h}-\damp \wh{\bw}_{h}+S_h\times\wh{\bd}_h\\
	\label{seq:phidh}
	&u_{0,h}(\cdot)\in V_h,\quad\text{s.t.}\quad a_h(u_{0,h},v)=F(v)-a_h(\widetilde{g}_h,v),\quad \forall v\in V_h,\quad \vphi_h := u_{0,h}+\widetilde{g}_h,
	\end{align}
\end{subequations}
where
\begin{equation}
\label{eq:Sh}
\begin{split}
 S_h(t,x)&:=\frac{\epsi}{2|\mathcal{C}_{\bi}|}\int_{\mathcal{C}_{\bi}}\Big((\Grad\vphi_h(t^{m+1},y) \cdot\wh{\bd}_h)\Grad\vphi_h(t^m,y)\\
 &\hphantom{\frac{\epsi}{2|\mathcal{C}_{\bi}|}\int_{\mathcal{C}_{\bi}}\Big(} +(\Grad\vphi_h(t^{m},y) \cdot\wh{\bd}_h)\Grad\vphi_h(t^{m+1},y)\Big) dy,\quad x\in \mathcal{C}_{\bi},\, t\in [t^m,t^{m+1}).
\end{split}
\end{equation}
Denote by $\Pi_h:H^1_0(\dom)\to V_h$ the projection onto the finite element space $V_h$. This projection satisfies
\begin{equation}\label{eq:H1projection}
\abs{\Pi_h v-v}_{H^1(\dom)}\stackrel{h\to 0}{\longrightarrow} 0,
\end{equation}
for any function in $H^1(\dom)$, see e.g.~\cite[Rem. 1.6]{Ern2004}.
Multiplying equations~\eqref{seq:ddh} and~\eqref{seq:wdh} with test functions $\phi_1,\phi_2\in C^\infty_c([0,T)\times\dom)$ and integrating over $[0,T)\times\dom$, and using the projection $\Pi_h\psi$ of a function $\psi\in L^2(0,T; H^1_0(\dom))$ as a test function in~\eqref{seq:phidh}, we obtain
\begin{subequations}\label{eq:m2hw}
	\begin{align}
	\label{seq:ddhw}
	&\int_0^T\int_{\dom} D_t \bd_h\cdot \phi_1 dx dt=\int_0^T\int_{\dom}(\wh{\bd}_{h}\times \wh{\bw}_{h})\cdot \phi_1 dxdt,\\
	\label{seq:wdhw}
	&-\int_{\Dt}^T\int_{\dom}\disp \bw_h\cdot D_t ^-\phi_2 dx dt-\int_{\dom}\disp \bw_h(0,x)\cdot \phi_2(0,x) dx\\
	&\quad =\int_0^T\int_{\dom}\left(-\oneconst(\Grad_h \wh{\bd}_{h}\times \wh{\bd}_{h}):\Grad_h\phi_2-\damp \wh{\bw}_{h}\cdot\phi_2+(S_h\times\wh{\bd}_h)\cdot\phi_2\right) dxdt, \notag\\
	\label{seq:phidhw}
	&\int_0^T\int_{\dom} (\Grad\vphi_h+\epsii(\bd_h\cdot\Grad\vphi_h)\bd_h)\cdot\Grad\Pi_h\psi dx dt = \int_0^T\int_{\dom}f\Pi_h\psi dxdt,
	\end{align}
\end{subequations}
where $\nu$ is the outward normal vector along the boundary.
Letting $h\to 0$ and using~\eqref{eq:weakconv},~\eqref{eq:strongconv}, and~\eqref{eq:H1projection} we obtain along a subsequence,
 \begin{subequations}\label{eq:m2hwx}
 	\begin{align}
 	\label{seq:ddhwx}
 	&\int_0^T\int_{\dom} \partial_t\bd\cdot \phi_1 dx dt=\int_0^T\int_{\dom}(\bd\times \wh{\bw} )\cdot \phi_1 dxdt,\\
 	\label{seq:wdhwx}
 	&-\int_{0}^T\int_{\dom}\disp \bw\cdot \partial_t\phi_2 dx dt-\int_{\dom}\disp \bw_0(x)\cdot \phi_2(0,x) dx\\
 	&\quad =\int_0^T\int_{\dom}\left(-\oneconst(\Grad \bd\times \bd):\Grad\phi_2-\damp \wh{\bw}\cdot\phi_2+(\overline{S}\times{\bd})\cdot\phi_2\right) dxdt, \notag\\
 	\label{seq:phidhwx}
 	&\int_0^T\int_{\dom} ( \Grad \vphi+\epsii(\bd\cdot \Grad\vphi) \bd )\cdot\Grad\psi dxdt =\int_0^T\int_{\dom}f \psi dxdt,
 	\end{align}
 \end{subequations}
where $\overline{S}$ is the weak limit of $\{S_h\}_{h>0}$. 
 Hence it remains to show that $\wh{\bw}=\bw$ and $\overline{S}=\epsi(\Grad\vphi\cdot\bd)\Grad\vphi$, where the latter will follow if we can show that $\Grad\vphi$ converges strongly in $L^2([0,T]\times\dom)$. 

To show that $\overline{S}=\epsi(\Grad\vphi\cdot\bd)\Grad\vphi$, we on one hand choose $u_0:=\vphi-\widetilde{g}$ as a test function in~\eqref{seq:phidhwx} and obtain
\begin{equation}
\label{eq:limphi}
\int_0^T\int_{\dom} ( |\Grad \vphi|^2+\epsii(\bd\cdot \Grad\vphi)^2) dx=\int_0^T\int_{\dom}f u_0 \,dxdt+\int_0^T\int_{\dom} ( \Grad \vphi+\epsii(\bd\cdot \Grad\vphi) \bd )\cdot\Grad\widetilde{g}\, dxdt. 
\end{equation}
On the other hand, we can take $\psi=u_{0,h}=\vphi_h-\widetilde{g}_h$ as a test function in~\eqref{seq:phidhw}, and integrate over $[0,T]$:
\begin{equation}
\label{eq:beforelimphi}
\begin{split}
\int_0^T\int_{\dom} ( |\Grad \vphi_h|^2+\epsii(\bd_h\cdot \Grad\vphi_h)^2) dx=&\int_0^T\int_{\dom}f u_{0,h}\, dxdt\\
&+\int_0^T\int_{\dom} ( \Grad \vphi_h+\epsii(\bd_h\cdot \Grad\vphi_h) \bd_h )\cdot\Grad\widetilde{g}_h\, dxdt. 
\end{split}
\end{equation}
Subtracting~\eqref{eq:limphi} from~\eqref{eq:beforelimphi}, we obtain
\begin{multline}
\label{eq:difference}
\underbrace{\int_0^T \int_{\dom}\left(  |\Grad \vphi_h|^2-|\Grad\vphi|^2+ \epsii(\bd_h\cdot \Grad \vphi_h)^2-\epsii(\bd\cdot\Grad\vphi)^2\right)dxdt}_{A}\\
=\underbrace{\int_0^T\int_{\dom}f\cdot u_{0,h} dxdt-\int_0^T\int_{\dom}f\cdot u_0 dxdt}_{F}\\
+\underbrace{\int_0^T\int_{\dom} ( \Grad \vphi_h+\epsii(\bd_h\cdot \Grad\vphi_h) \bd_h )\cdot\Grad\widetilde{g}_h\, dxdt-  
	\int_0^T\int_{\dom} ( \Grad \vphi+\epsii(\bd\cdot \Grad\vphi) \bd )\cdot\Grad\widetilde{g}\, dxdt}_{G}.
\end{multline}
We denote $E_h:= \Grad\vphi_h$ and $E:= \Grad\vphi$. Then we can rewrite the term $A$ as
\begin{align*}
A &= \int_0^T \int_{\dom}\left(  |E_h -E|^2 +\epsii(\bd_h\cdot (E_h-E) )^2\right)dxdt\\
&\quad+\int_0^T \int_{\dom}\left( 2 E\cdot (E_h-E) +\epsii\left( 2(\bd_h\cdot (E_h-E))(\bd_h\cdot E)  +(\bd_h\cdot E)^2-(\bd\cdot E)^2\right)\right)dxdt\\
\end{align*} 
We note that for $\epsii>-1$, we have
\begin{equation*}
\int_0^T \int_{\dom}\left(  |E_h -E|^2 +\epsii(\bd_h\cdot (E_h-E) )^2\right)dxdt\geq \min\{1,1-\epsii\}\int_0^T \int_{\dom}|E_h -E|^2 dx dt.
\end{equation*}
Hence~\eqref{eq:difference} implies
\begin{multline}
\label{eq:tbestimated}
\min\{1,1-\epsii\}\norm{E-E_h}_{L^2([0,T]\times\dom)}^2\\
\leq\int_0^T \int_{\dom}\left(  |E_h -E|^2 +\epsii(\bd_h\cdot (E_h-E) )^2\right)dxdt\\
=\underbrace{ 2\int_0^T \int_{\dom}  E\cdot (E-E_h)dxdt}_{A_1}   -\underbrace{2\epsii \int_0^T \int_{\dom} (\bd_h\cdot E_h)(\bd_h\cdot E) dxdt}_{A_2}\\
 +\underbrace{\epsii  \int_0^T \int_{\dom} \left((\bd\cdot E)^2+(\bd_h\cdot E)^2\right)dxdt}_{A_3} 
 +F+G
\end{multline}
We now show that each of the terms $A_1$, $F$ and $G$ go to zero as $h\to 0$. In addition, we will see that $A_3-A_2 \to 0$. The term $A_1$ goes to zero because $\{E_h\}_{h>0}$ converges weakly up to a subsequence. 
For the term $A_2$, we use Lemma~\ref{lem:kennethlemma} componentwise for $v_m=E_{h_m}^{(i)}E^{(j)}$ and $u_m=\bd_{h_m}^{(k)}\bd_{h_m}^{(\ell)}$, $i,j,k,\ell=1,\dots,n$, since $\bd_h\in L^\infty([0,T)\times\dom)$ and converges almost everywhere, and $E_h^{(i)}E^{(j)}$ converges weakly in $L^1$ to $E^{(i)}E^{(j)}$ as $h,\Delta t\to 0$. Thus,
\begin{equation*}
A_2\longrightarrow 2\epsii\int_0^T\int_\dom (\bd\cdot E)^2 dx dt,\quad \text{ as } h,\Delta t\to 0.
\end{equation*}
For the term $A_3$, we can use the same lemma with $v_m = E^{(i)}E^{(j)}$ and $u_m=\bd_{h_m}^{(k)}\bd_{h_m}^{(\ell)}$, $i,j,k,\ell=1,\dots,n$ to obtain
\begin{equation*}
A_3\longrightarrow 2\epsii \int_0^T\int_\dom (\bd\cdot E)^2 dx dt,\quad \text{ as } h,\Delta t\to 0.
\end{equation*}
Hence $A_3-A_2\stackrel{h,\Delta t\to 0}{\longrightarrow} 0$. The term $F$ goes to zero due to the weak convergence of $u_{0,h}$,~\eqref{eq:weakphi}.

 We turn to the remaining term $G$. 
 We can write it as
 \begin{multline*}
 G = \underbrace{\int_0^T\int_{\dom} ( \Grad \vphi_h\cdot\Grad\widetilde{g}_h - \Grad \vphi\cdot\Grad\widetilde{g})\, dxdt}_{G_1}\\
 + \underbrace{\epsii\int_0^T\int_{\dom} ( (\bd\cdot \Grad\vphi_h) \bd \cdot\Grad\widetilde{g}-(\bd\cdot \Grad\vphi) \bd \cdot\Grad\widetilde{g})\, dxdt}_{G_2}\\
 + \underbrace{\epsii\int_0^T\int_{\dom} ( (\bd_h\cdot \Grad\vphi_h) \bd_h \cdot\Grad\widetilde{g}-(\bd\cdot \Grad\vphi_h) \bd \cdot\Grad\widetilde{g})\, dxdt}_{G_3}\\
 +
 \underbrace{\epsii\int_0^T\int_{\dom} ((\bd_h\cdot \Grad\vphi_h) \bd_h \cdot\Grad\widetilde{g}_h - (\bd_h\cdot \Grad\vphi_h) \bd_h \cdot\Grad\widetilde{g}\, dxdt}_{G_4}
 \end{multline*}
 $G_1$ goes to zero by the strong convergence of $\Grad\widetilde{g}_h$ in $L^2_{t,x}$ and the weak convergence of $\Grad\vphi_h$ in $L^2_{t,x}$, and $G_2$ goes to zero by the weak convergence of $\Grad\vphi_h$ in $L^2_{t,x}$ as well. For $G_3$, we use Lemma~\ref{lem:kennethlemma} componentwise with $v_m = \partial_i\vphi_{h_m}\partial_{j}\widetilde{g}$ and $u_m=\bd_{h_m}^{(k)}\bd_{h_m}^{(\ell)}-\bd^{(k)}\bd^{(\ell)}$ since $\bd_h$ converges almost everywhere and $\partial_i\vphi_{h_m}\partial_{j}\widetilde{g}$ converges weakly to $\partial_i\vphi\partial_j\widetilde{g}$ in $L^1$ as $\Dt,h\to 0$. Finally, $G_4\to 0$ due to the strong convergence of $\Grad\widetilde{g}_h$ in $L^2$, c.f.~\eqref{eq:H1projection}. Hence also $G\to 0$ as $\Dt,h\to 0$. 
 Therefore, the right hand side of~\eqref{eq:tbestimated} goes to zero as $h,\Delta t\to 0$ (up to subsequence).
Hence
\begin{equation}
\label{eq:strongE}
E_h=\Grad\vphi_h\longrightarrow E=\Grad\vphi,\quad \text{ in } L^2([0,T]\times\dom),\, \text{ as } h,\Delta t\to 0.
\end{equation}
We observe that
\begin{align*}
\int_{0}^T\int_{\dom}(S_h\times\wh{\bd}_h)\cdot\phi_2 dx dt& = \frac{\epsi}{2}\underbrace{\int_{0}^T\int_{\dom}(\Grad\vphi_h(t+\Dt,x) \cdot\wh{\bd}_h)(\Grad\vphi_h \times\wh{\bd}_h)\cdot\phi_2 dx dt}_{I}\\
& + \frac{\epsi}{2}\underbrace{\int_{0}^T\int_{\dom}(\Grad\vphi_h \cdot\wh{\bd}_h)(\Grad\vphi_h(t+\Dt,x) \times\wh{\bd}_h)\cdot\phi_2 dx dt}_{J}. 
\end{align*}
We consider the term $I$ ($J$ can be treated similarly). We rewrite it as follows
\begin{align*}
&\int_{0}^T\int_{\dom}(\Grad\vphi_h(t+\Dt,x) \cdot\wh{\bd}_h)(\Grad\vphi_h \times\wh{\bd}_h)\cdot\phi_2 dx dt\\
\quad& = \underbrace{\int_{0}^T\int_{\dom}(\Grad\vphi_h(t+\Dt,x) \cdot{\bd})(\Grad\vphi \times{\bd})\cdot\phi_2 dx dt}_{I_1} \\
&\qquad + \underbrace{\int_{0}^T\int_{\dom}\left[(\Grad\vphi_h(t+\Dt,x) \cdot\wh{\bd}_h)(\Grad\vphi \times\wh{\bd}_h)-(\Grad\vphi_h(t+\Dt,x) \cdot{\bd})(\Grad\vphi \times{\bd})\right]\cdot\phi_2 dx dt}_{I_2}\\
&\qquad+\underbrace{\int_{0}^T\int_{\dom}(\Grad\vphi_h(t+\Dt,x) \cdot\wh{\bd}_h)((\Grad\vphi_h-\Grad\vphi) \times\wh{\bd}_h)\cdot\phi_2 dx dt}_{I_3}.
\end{align*}
To analyze term $I_1$, we denote $v:=[(\Grad\vphi\times\bd)\cdot\phi_2]\bd\in L^\infty(0,T;L^2(\dom))$. We assume that $\Dt$ is small enough such that the support of $\phi_2(t,\cdot)$ (and hence $v(t,\cdot)$) is contained in $(0,T-\Dt)$. Then
\begin{align*}
I_1 & = \int_0^T\int_{\dom}\Grad\vphi_h(t+\Dt,x)\cdot v(t,x) dx dt \\
& = \int_{\Dt}^{T+\Dt}\int_{\dom}\Grad\vphi_h(t,x)\cdot v(t-\Dt,x) dx dt \\
& = \int_{\Dt}^{T}\int_{\dom}\Grad\vphi_h(t,x)\cdot (v(t-\Dt,x)-v(t,x)) dx dt + \int_{\Dt}^{T}\int_{\dom}\Grad\vphi_h(t,x)\cdot v(t,x) dx dt \\
& = \int_{\Dt}^{T}\int_{\dom}\Grad\vphi_h(t,x)\cdot (v(t-\Dt,x)-v(t,x)) dx dt + \int_{0}^{T}\int_{\dom}\Grad\vphi_h(t,x)\cdot v(t,x) dx dt\\
&\quad - \int_0^{\Dt}\int_{\dom}\Grad\vphi_h(t,x)\cdot v(t,x) dx dt,
\end{align*}
where we used for the second and third equality that $v(t,\cdot)$ is supported in $(0,T-\Dt)$.
Hence
\begin{align*}
\left|I_1-\int_{0}^{T}\int_{\dom}\Grad\vphi(t,x)\cdot v(t,x) dx dt\right|&\leq \left|\int_{\Dt}^{T}\int_{\dom}\Grad\vphi_h(t,x)\cdot (v(t-\Dt,x)-v(t,x)) dx dt\right|\\
&\quad  +\left| \int_{0}^{T}\int_{\dom}(\Grad\vphi_h-\Grad\vphi)\cdot v(t,x) dx dt\right|\\
&\quad +\left|\int_0^{\Dt}\int_{\dom}\Grad\vphi_h(t,x)\cdot v(t,x) dx dt\right|\\
&\leq \norm{\Grad\vphi_h}_{L^2([0,T]\times\dom)} \norm{v(\cdot-\Dt)-v}_{L^2([\Dt,T]\times\dom)}\\
&\quad  +\left|\int_{0}^{T}\int_{\dom}(\Grad\vphi_h-\Grad\vphi)\cdot v(t,x) dx dt\right|\\
&\quad + \Dt\norm{v}_{L^\infty(0,T;L^2(\dom))}\norm{\Grad\vphi_h}_{L^\infty(0,T;L^2(\dom))}\\
&\leq C \norm{v(\cdot-\Dt)-v}_{L^2([\Dt,T]\times\dom)} \\
&\quad  +\left|\int_{0}^{T}\int_{\dom}(\Grad\vphi_h-\Grad\vphi)\cdot v(t,x) dx dt\right|+ C\Dt ,
\end{align*}
where we used the energy inequality, Lemma~\ref{lem:discreteenergy}, for the last inequality. Now the first term on the right hand side goes to zero by the continuity of $L^2$-shifts and the second term goes to zero due to the weak convergence of $\Grad\vphi_h$,~\eqref{eq:weakphi}. Hence
\begin{equation*}
I_1\longrightarrow \int_0^T\int_{\dom} (\Grad\vphi\cdot\bd)(\Grad\vphi\times\bd)\cdot\phi_2 dx dt,\quad \text{ as } \Dt,h\to 0,
\end{equation*}
up to a subsequence. Next, we consider $I_2$:
\begin{align*}
I_2 =\sum_{i=1}^n\int_{0}^T\int_{\dom}\partial_i\vphi_h(t+\Dt,x) \cdot\left\{\wh{\bd}_h^{(i)}(\Grad\vphi \times\wh{\bd}_h) -\bd^{(i)}(\Grad\vphi \times{\bd}) \right\}\cdot\phi_2 dx dt.
\end{align*}
We apply Lemma~\ref{lem:kennethlemma} componentwise with $v_m=\partial_i\vphi_{h_m}\partial_j\vphi\phi_2^{(r)}$ and $u_m = \bd_{h_m}^{(k)}\bd_{h_m}^{(\ell)}-\bd^{(k)}\bd^{(\ell)}$, $i,j,k,\ell,r=1,\dots, n$ and obtain that $I_2\to 0$ as $\Dt,h\to 0$ because $u_m\to 0$ a.e.. Lastly, $I_3$ goes to zero due to the $L^2$-convergence of $\Grad\vphi_h$:
\begin{align*}
|I_3|&\leq \norm{\wh{\bd}_h}_{L^\infty}^2\norm{\phi_2}_{L^\infty}\norm{\Grad\vphi_h(\cdot+\Dt)}_{L^2([0,T]\times\dom)}\norm{\Grad\vphi_h-\Grad\vphi}_{L^2([0,T]\times\dom)}\\
&\leq C\norm{\Grad\vphi_h-\Grad\vphi}_{L^2([0,T]\times\dom)} \stackrel{h,\Dt\to 0}{\longrightarrow} 0,
\end{align*}
where we also used the energy inequality, Lemma~\ref{lem:discreteenergy}. Thus, we conclude that
\begin{equation*}
I\stackrel{h,\Dt\to 0}{\longrightarrow}\int_0^T\int_{\dom} (\Grad\vphi\cdot\bd)(\Grad\vphi\times\bd)\cdot\phi_2 dx dt.
\end{equation*}
In a similar way, one can show that $J$ converges to the same quantity and hence
\begin{equation*}
\int_0^T\int_{\dom}(S_h\times\wh{\bd}_h)\cdot \phi_2 dx dt \stackrel{h,\Dt\to 0}{\longrightarrow}\epsi \int_0^T\int_{\dom} (\Grad\vphi\cdot\bd)(\Grad\vphi\times\bd)\cdot\phi_2 dx dt.
\end{equation*}
It remains to show that the limits of $\wh{\bw}_h$ and $\bw_h$ agree. We note that
\begin{equation*}
\wh{\bw}_h(t,x) = \bw_h(t,x) +\frac{1}{2}\Dt D_t^+\bw_h(t,x))
\end{equation*}
and use the finite difference scheme~\eqref{seq:wd} to show that $\Dt D_t^+ \bw_h$ vanishes. We write~\eqref{seq:wd} in terms of the piecewise constant functions $\bd_h$, $\bw_h$, etc.,
\begin{equation}\label{eq:fdmwh}
\disp D_t^+ \bw_h = -\damp\wh{\bw}_h +\oneconst\Div_h(\Grad_h\wh{\bd}_h\times \wh{\bd}_h)+S_h\times\wh{\bd}_h,
\end{equation}
where $S_h$ is defined in~\eqref{eq:Sh}.
Note that $S_h\in L^\infty([0,T];)L^1(\dom))$ for all $h>0$ by the a priori estimates~\eqref{eq:uniformbounds}.
We multiply~\eqref{eq:fdmwh} by a test function $\phi\in H^2(\dom)$, integrate over the domain, and change variables,
\begin{equation}\label{eq:fdmwh2}
\disp\int_{\dom} D_t^+ \bw_h\cdot \phi dx = -\int_{\dom}\left(\damp\wh{\bw}_h\phi +\oneconst(\Grad_h\wh{\bd}_h\times \wh{\bd}_h):\Grad_h \phi\right) dx+\int_{\dom} (S_h\times\wh{\bd}_h)\cdot \phi dx.
\end{equation}
Multiplying everything with $\Dt$, we see that the right hand side goes to zero as $\Dt\to 0$ due to the uniform bounds coming from the energy estimate~\eqref{eq:discreteenergyestimate}, and hence $\Dt \int_{\dom} D_t^+\bw_h\cdot \phi dx\to 0$ for all test functions $\phi\in L^1([0,T];H^2(\dom))$. Hence $\bw=\wh{\bw}$ in $L^\infty([0,T];H^{-2}(\dom))$ for this subsequence. Using density of $H^2$-functions in $L^2$, and that $\bw,\wh{\bw}\in L^2([0,T]\times\dom)$, we have $\bw=\wh{\bw}$ in $L^2([0,T]\times\dom)$ and in particular $\bw=\wh{\bw}$ for almost every $(t,x)\in [0,T]\times \dom$. This allows us to pass to the limit in all the terms in~\eqref{eq:m2hw} and conclude that the limit is a weak solution of~\eqref{eq:angmomentum}. Since $\bd_t\in L^\infty([0,T];L^2(\dom))$, the first equation~\eqref{seq:dt} holds for a.e. $(t,x)$. Since $|\bd|=1$, this means (taking the cross product with $\bd$):
\begin{equation}\label{eq:dtw}
\bd_t\times\bd = (\bd\times\bw)\times \bd = \bw-(\bd\cdot\bw)\bd.
\end{equation}
But from the numerical method~\eqref{seq:dd}--\eqref{seq:wd}, we have
\begin{equation*}
D_t(\bd_h^m\cdot\bw_h^m)= \bd^{m+\nicefrac{1}{2}}_h\cdot D_t\bw_h^m+\bw_h^{m+\nicefrac{1}{2}}\cdot D_t \bd_h^m = 0,
\end{equation*}
and since $\bd_h^0\cdot\bw_h^0=0$, this implies that
\begin{equation*}
\bd_h\cdot \bw_h=0,\quad (t,x)\in [0,T]\times\dom.
\end{equation*}
Since $\bd_h$ converges strongly, and $\bw_h$ converges weakly in $L^2$, we obtain that in the limit
\begin{equation*}
\bd\cdot \bw=0,\quad \text{ a.e. } (t,x)\in [0,T]\times\dom.
\end{equation*}
Thus,~\eqref{eq:dtw} becomes
\begin{equation*}
\bw = \bd_t \times \bd,
\end{equation*}
and we conclude that the limit $(\bd,\vphi)$ satisfies~\eqref{eq:weakformulation}. The continuity of $\bd$ at zero,~\eqref{eq:contat0}, follows from the fact that $\bd_t\in L^2([0,T]\times\dom)$. Hence $(\bd,\bw)$ is a weak solution in the sense of Definition~\ref{def:weaksol}.
\end{proof}
\begin{remark}
	Combining the techniques used here with those used in~\cite{damped}, it should also be possible to prove convergence in the case that $\disp=0$ and $\damp>0$.
\end{remark}
\section{Solving the nonlinear system \eqref{eq:m2}}\label{sec:fixed}
The algebraic system \eqref{eq:m2} is nonlinear and implicit, therefore it is not a priori clear that it has a solution. In this section, we show that a solution can be obtained using a fixed point iteration. In 3D, this iteration converges only under a strict assumption for the CFL-condition, specifically, given the estimates on the approximations we have from the energy estimate, Lemma~\ref{lem:discreteenergy}, we require $\Dt\leq C h^\theta$, where $\theta=\max\{1,\sfrac{n}{2}\}$, for some constant $C>0$. It is possible that there are faster ways of obtaining solutions to \eqref{eq:m2} and under milder assumptions on the time step $\Dt$. This is one of our current research efforts.
The fixed point iteration yields a constructive existence result, but in practice, the algebraic system~\eqref{eq:m2} can also be solved using Newton's method or similar.
We will assume for this section that $\disp>0$, if this assumption is not made, then the CFL-condition needed for convergence of the fixed point algorithm might be stricter, c.f.~\cite{damped}.

To set up a suitable fixed point iteration, we first collect a few observations:
We first note, that the update for $\bd_{\bi}^m$,~\eqref{seq:dd}, can be written as
\begin{equation}\label{eq:V}
\bd^{m+1}_{\bi} = V(\bw_{\bi}^{m+\frac12})\bd_{\bi}^m,
\end{equation}
where $V$ is the matrix given by
\begin{equation*}
V(w)=\frac{1}{1+\frac{\Delta t^2}{4}|w|^2}\left(\left(1-\frac{\Delta t^2}{4}|w|^2\right)\mathbf{I}+\frac{\Delta t^2}{2}w\otimes w+\Delta tQ(w)\right),
\end{equation*}
with $Q(w)$ defined as the skew-symmetric matrix
\begin{equation*}
Q(w)=\begin{pmatrix}
0 & w^{(3)} & -w^{(2)}\\
-w^{(3)} & 0 & w^{(1)}\\
w^{(2)} & -w^{(1)} & 0
\end{pmatrix}.
\end{equation*}
In fact, $Q$ is such that $Q(w)v=v\times w$. In~\cite[Lemma 4.8]{Karper2014} we have shown that 
\begin{equation}\label{eq:destimate}
\norm{V(\overline{\bw}_h^1)\bd_h-V(\overline{\bw}^2)\bd_h}_{L^2(\dom)}\leq C\Delta t \norm{\bw^1_h-\bw_h^2}_{L^2(\dom)},
\end{equation}
where $\overline{\bw}_h^i=(\bw_h^i+\bw_h)/2$, $i=1,2$, for piecewise constant functions $\bw_h$, $\bd_h$, $\bw_h^i$, $i=1,2$ on the grid on $\dom$.
Similarly, $\vphi^{m+1}_h$ solves the equation~\eqref{seq:phid}, and is therefore a function of $\bd^{m+1}_{\bi}$ and hence $\bw_{\bi}^{m+1}$. So $\vphi^{m+1}_h:=A(V(\overline{\bw}_h^{m+\nicefrac{1}{2}})\bd_h^m):=\mathcal{A}(\bw_h^{m+1})$. The bilinear forms $a$ and $a_h$ satisfy
\begin{equation*}
\begin{split}
(1+\min\{0,\epsii\})\norm{v}_{H^1_0(\dom)}^2&\leq a(v,v),\quad \forall v\in H^1_0(\dom),\\
(1+\min\{0,\epsii\})\norm{v}_{H^1_0(\dom)}^2&\leq a_h(v,v),\quad \forall v\in H^1_0(\dom),
\end{split}
\end{equation*}
and 
\begin{equation*}
\begin{split}
\left|a(v,w)\right|&\leq C(\epsii) \norm{v}_{H^1_0}\norm{w}_{H^1_0},\quad \forall v,w\in H^1_0(\dom),\\
\left|a_h(v,w)\right|&\leq C(\epsii) \norm{v}_{H^1_0}\norm{w}_{H^1_0},\quad \forall v,w\in H^1_0(\dom),
\end{split}
\end{equation*}
for any $\bd$ with $\|\bd\|_\infty\leq 1$ ($\|\bd_h\|_\infty\leq 1$) uniformly in $h$, which implies ellipticity and continuity in $H^1_0$ of $a$, $a_h$ respectively and hence solvability of the elliptic equation for a given $\bd$, $\bd_h$  with $\|\bd\|_\infty\leq 1$ ($\|\bd_h\|_\infty\leq 1$)  respectively as long as $\epsii>-1$.
Furthermore, we can estimate the difference $\Grad\vphi_h^1-\Grad\vphi_h^2$ between two solutions $\vphi_h^1$ and $\vphi_h^2$ of~\eqref{seq:phidh} for different given $\bd_h^1$ and $\bd_h^2$ (but same boundary conditions and right hand side $f$): If we assume the same boundary conditions, then $\vphi_h^1-\vphi_h^2\in V_h$, hence we can use it as a test function in~\eqref{seq:phidh} for the equation for $\vphi_h^1$ and $\vphi_h^2$. Subtracting the two resulting equations, we have (we omit the time dependence to simplify the notation)
\begin{equation*}
a_h(\vphi_h^1,\vphi_h^1-\vphi_h^2)-a_h(\vphi_h^2,\vphi_h^1-\vphi_h^2) = 0.
\end{equation*}
This is the same as
\begin{multline*}
\int_{\dom}|\Grad(\vphi_h^1-\vphi_h^2)|^2 dx +\epsii\int_{\dom}\left(\Grad(\vphi_h^1-\vphi_h^2)\cdot\frac{\bd_h^1+\bd_h^2}{2}\right)^2 dx \\
+\frac{\epsii}{2}\int_\dom \left((\bd_h^1-\bd_h^2)\cdot \Grad(\vphi_h^1-\vphi_h^2)\right)\left((\Grad\vphi_h^1\cdot\bd_h^1)+(\Grad\vphi_h^2\cdot\bd_h^2)\right) dx \\
+\frac{\epsii}{4}\int_\dom \left((\Grad\vphi_h^1+\Grad\vphi_h^2)\cdot(\bd_h^1-\bd_h^2)\right)\left( (\bd_h^1+\bd_h^2)\cdot(\Grad\vphi_h^1-\Grad\vphi_h^2)\right) dx
=0.
\end{multline*}
We rearrange terms:
\begin{multline*}
\int_{\dom}|\Grad(\vphi_h^1-\vphi_h^2)|^2 dx +\epsii\int_{\dom}\left(\Grad(\vphi_h^1-\vphi_h^2)\cdot\frac{\bd_h^1+\bd_h^2}{2}\right)^2 dx \\
=-\frac{\epsii}{2}\int_\dom \left((\bd_h^1-\bd_h^2)\cdot \Grad(\vphi_h^1-\vphi_h^2)\right)\left((\Grad\vphi_h^1\cdot\bd_h^1)+(\Grad\vphi_h^2\cdot\bd_h^2)\right) dx \\
-\frac{\epsii}{4}\int_\dom \left((\Grad\vphi_h^1+\Grad\vphi_h^2)\cdot(\bd_h^1-\bd_h^2)\right)\left( (\bd_h^1+\bd_h^2)\cdot(\Grad\vphi_h^1-\Grad\vphi_h^2)\right) dx.
\end{multline*}
The left hand side of this expression, we can lower bound by
\begin{equation*}
\text{ LHS }\geq (1+\min(0,\epsii))\norm{\Grad\vphi_h^1-\Grad\vphi_h^2}_{L^2(\dom)}^2
\end{equation*}
The right hand side, we can upper bound by
\begin{equation*}
\text{ RHS }\leq C\norm{\Grad\vphi_h^1-\Grad\vphi_h^2}_{L^2(\dom)}\norm{\bd_h^2-\bd_h^1}_{L^2(\dom)}\left(\norm{\Grad\vphi_h^1}_{L^\infty}+\norm{\Grad\vphi_h^2}_{L^\infty}\right).
\end{equation*}
Combining the two, we obtain
\begin{equation*}
\norm{\Grad\vphi_h^1-\Grad\vphi^2}_{L^2(\dom)}\leq C_{\epsii} \norm{\bd_h^2-\bd_h^1}_{L^2(\dom)}\left(\norm{\Grad\vphi_h^1}_{L^\infty}+\norm{\Grad\vphi_h^2}_{L^\infty}\right).
\end{equation*}
Combining the `inverse' estimate for finite element approximations, e.g.~\cite[Thm. 3.2.6]{Ciarlet2002},
\begin{equation}\label{eq:inverse}
\norm{f_h}_{L^\infty(\dom)}\leq h^{-\frac{n}{q}}\norm{f_h}_{L^q(\dom)},
\end{equation}
for any $1\leq q<\infty$ with the energy estimate~\eqref{eq:discreteenergyestimate}, we obtain
\begin{equation*}
\norm{\Grad\vphi_h^1-\Grad\vphi^2}_{L^2(\dom)}\leq C_{\epsii} h^{-\frac{n}{2}} \norm{\bd_h^2-\bd_h^1}_{L^2(\dom)}\left(\norm{\Grad\vphi_h^1}_{L^2}+\norm{\Grad\vphi_h^2}_{L^2}\right)\leq C_{\epsii} h^{-\frac{n}{2}} \norm{\bd_h^2-\bd_h^1}_{L^2(\dom)}.
\end{equation*}
When $\bd_h^1$ and $\bd_h^2$ are given by $\bd_h^1 = V(\overline{\bw}^1_h)\bd_h$ and $\bd_h^2 = V(\overline{\bw}^2_h)\bd_h$, we also obtain, using~\eqref{eq:destimate},
\begin{equation}\label{eq:phigradestimate}
\norm{\Grad\vphi_h^1-\Grad\vphi^2}_{L^2(\dom)}\leq C_{\epsii} h^{-\frac{n}{2}}\Delta t \norm{\bw_h^2-\bw_h^1}_{L^2(\dom)}.
\end{equation}
Using the Poincar\'{e} inequality for the difference of $\vphi_h^1$ and $\vphi_h^2$ which has zero trace, we also get
\begin{equation}\label{eq:phiestimate}
\norm{\vphi_h^1-\vphi^2_h}_{L^2(\dom)}\leq C_{\epsii} h^{-\frac{n}{2}}\Delta t \norm{\bw_h^2-\bw_h^1}_{L^2(\dom)}.
\end{equation}
We can therefore write the updates $\vphi_h^{m+1}$ and $\bd_h^{m+1}$ as functions of $(\bd_h^m,\bw_h^m,\vphi_h^m)$ and $\bw_h^{m+1}$, i.e., $\bd_h^{m+1}=V(\bw_h^{m+\nicefrac{1}{2}})\bd_h^m$ and $\vphi^{m+1}_h:=A(V(\bw_h^{m+\nicefrac{1}{2}})\bd_h^m):=\mathcal{A}(\bw_h^{m+1})$ and consider the mapping $\bw_h^{\text{old}}\mapsto \bw^{\text{new}}_h=F(\bw^{\text{old}}_h)$ defined by, for given $(\bd_h,\bw_h,\vphi_h)$, (and $f$ and boundary data $g$)
\begin{equation}\label{eq:fixedpointnewtry}
\begin{split}
\overline{\bw}_h &= \frac{\bw_h+\bw_h^{\text{old}}}{2},\\
	\bd_h^{\text{new}} & = V(\overline{\bw}_h)\bd_h,\\
	\overline{\bd}_h&= \frac{\bd_h+\bd_h^{\text{new}}}{2},\\
	u_{0,h}^{\text{new}}\in V_h,\quad\text{s.t.}\quad &a_h^{\text{new}}(u_{0,h}^{\text{new}},v)=F^{m+1}(v)-a_h^{\text{new}}(\widetilde{g}_h^{m+1},v),\quad \forall v\in V_h,\quad \vphi_h^{\text{new}} := u_{0,h}^{\text{new}}+\widetilde{g}_h^{m+1},\\
\disp\frac{\bw^{\text{new}}_h-\bw_h}{\Delta t} &=\oneconst\Delta_h \overline{\bd}_h \times \overline{\bd}_h-\damp \frac{\bw_h+\bw_h^{\text{new}}}{2}+S_h\times\overline{\bd}_h\\
\end{split}
\end{equation}
where
\begin{equation*}
a_h^{\text{new}}(v,w):=\int_{\dom}\left[\Grad v\cdot \Grad w+\epsii(\bd_h^{\text{new}}\cdot\Grad v)(\bd_h^{\text{new}}\cdot\Grad w)\right] dx
\end{equation*}
(i.e. the bilinear form using the update $\bd_h^{\text{new}}$ as coefficients), and
\begin{equation*}
S_h(x):=\frac{\epsi}{2|\mathcal{C}_{\bi}|}\int_{\mathcal{C}_{\bi}}\Big((\Grad\vphi_h^{\text{new}}(y) \cdot\overline{\bd}_h)\Grad\vphi_h(y)
 +(\Grad\vphi_h(y) \cdot\overline{\bd}_h)\Grad\vphi_h^{\text{new}}(y)\Big) dy,\quad x\in \mathcal{C}_{\bi}.
\end{equation*}
We observe that if $(\bd_h,\bw_h,\vphi_h)=(\bd_h^m,\bw_h^m,\vphi_h^m)$, then $\bw_h^{m+1}$ is a fixed point of~\eqref{eq:fixedpointnewtry}. Therefore, we will show that $\bw_h^{\text{old}}\mapsto \bw^{\text{new}}_h=F(\bw^{\text{old}}_h)$ is a contraction under a suitable time step constraint and hence the fixed point exists.
\begin{lemma}
	\label{lem:contraction}
	The mapping $\bw_h^{\text{old}}\mapsto \bw^{\text{new}}_h=F(\bw^{\text{old}}_h)$ defined by~\eqref{eq:fixedpointnewtry} is a contraction
	\begin{equation*}
	\norm{F(\bw_h^1)-F(\bw_h^2)}_{L^2}\leq q \norm{\bw_h^1-\bw_h^2}_{L^2},
	\end{equation*}
	for $0<q<1$ under the time step constraint
	\begin{equation}\label{eq:CFLnew}
	\Delta t\leq \kappa h^{\max\{1,\nicefrac{n}{2}\}},
	\end{equation}
for a constant $\kappa>0$ sufficiently small.
\end{lemma}
\begin{proof}
	We denote $\bw^{1,\text{new}}_h=F(\bw^{1}_h)$ and $\bw^{2,\text{new}}_h=F(\bw^{2}_h)$ two solutions with starting values $\bw_h^i$, $i=1,2$ and $\overline{\bw}_h^i=0.5(\bw_h+\bw_h^i)$, $\bd_h^i:=\bd_h^{i,\text{new}}$ and $\vphi_h^i:=\vphi_h^{i,\text{new}}$.
	We rewrite the equation for $\bw_h^{i,\text{new}}$ as
	\begin{equation*}
	\left(\disp+\frac{\damp\Delta t}{2}\right)\bw^{i,\text{new}}_h =\Delta t\oneconst\Delta_h \overline{\bd}_h^i \times \overline{\bd}_h^i+\left(\disp-\frac{ \damp\Delta t}{2} \right)\bw_h +\Dt S_h^i\times\overline{\bd}_h^i
	\end{equation*}
	We subtract the equation for $\bw_h^{1,\text{new}}$ from the equation for $\bw_h^{2,\text{new}}$ to obtain
		\begin{equation*}
	\begin{split}
	&\left(\disp+\frac{\damp\Delta t}{2}\right)(\bw^{2,\text{new}}_h-\bw^{1,\text{new}}_h)\\
	 &=\underbrace{\frac{1}{2}\Delta t\oneconst\Delta_h \left(\bd_h^2-\bd_h^1\right)  \times \overline{\bd}_h^2}_{A_1}+\underbrace{\frac{1}{2}\Delta t\oneconst\Delta_h \overline{\bd}_h^1  \times \left(\bd_h^2-\bd_h^1\right)}_{A_2}+\underbrace{\frac{\Delta t}{2} S_h^2\times (\bd_h^2-\bd_h^1)}_{B}\\
&\quad +
\underbrace{\sum_{\bi}\frac{\epsi\Dt}{4|\mathcal{C}_{\bi}|}\int_{\mathcal{C}_{\bi}}\!\!\!\Big((\Grad\vphi_h^{2}(y) \cdot({\bd}_h^2-{\bd}_h^1))\Grad\vphi_h(y)
+(\Grad\vphi_h(y) \cdot({\bd}_h^2-{\bd}_h^1))\Grad\vphi_h^{2}(y)\Big) dy\mathbf{1}_{\mathcal{C}_{\bi}}(x) \times\overline{\bd}_h^1}_{D}\\
&\quad +
\underbrace{\sum_{\bi}\frac{\epsi\Dt}{2|\mathcal{C}_{\bi}|}\int_{\mathcal{C}_{\bi}}\!\!\!\Big((\Grad(\vphi_h^{2}-\vphi_h^1)(y) \cdot\overline{\bd}_h^1)\Grad\vphi_h(y)
	+(\Grad\vphi_h(y) \cdot\overline{\bd}_h^1)\Grad(\vphi_h^{2}-\vphi_h^{1})(y)\Big) dy\mathbf{1}_{\mathcal{C}_{\bi}}(x) \times\overline{\bd}_h^1}_{E}.
	\end{split}
	\end{equation*}
	Thus,
	\begin{equation*}
	\norm{\bw^{2,\text{new}}_h-\bw^{1,\text{new}}_h}_{L^2}\leq \frac{1}{\disp+\frac{\damp\Delta t}{2}}\left(\sum_{i=1}^2\norm{A_i}_{L^2}+\norm{B}_{L^2}+\norm{D}_{L^2}+\norm{E}_{L^2}\right),
	\end{equation*}
	and we have to bound each of these terms. Again, we will frequently use the estimate
	\begin{equation*}
	\norm{\bd_h^i}_{L^\infty}\leq 1,
	\end{equation*}
	which follows because $\bd_h^i$ satisfy the unit length constraint (see~\cite{Karper2014} for a proof of this fact via analyzing the matrix $V$ in~\eqref{eq:V}). We will denote by $C$ a generic constant that does not depend on $h$ or $\Delta t$. For $A_1$, we have
	\begin{equation}
	\label{eq:A1}
	\norm{A_1}_{L^2}\leq \frac{k\Delta t}{2}\norm{\Delta_h(\bd_h^2-\bd_h^1)}_{L^2}\norm{\overline{\bd}_h^2}_{L^\infty}\leq \frac{C \Delta t}{h^2}\norm{\bd_h^2-\bd_h^1}_{L^2}\stackrel{\eqref{eq:destimate}}{\leq } \frac{C \Delta t^2}{h^2}\norm{\bw_h^2-\bw_h^1}_{L^2}.
	\end{equation}
	Similarly, we find for $A_2$,
	\begin{equation}
	\label{eq:A2}
	\norm{A_2}_{L^2}\leq \frac{k\Delta t}{2}\norm{\bd_h^2-\bd_h^1}_{L^2}\norm{\Delta_h\overline{\bd}_h^2}_{L^\infty}\leq \frac{C \Delta t}{h^2}\norm{\bd_h^2-\bd_h^1}_{L^2}\stackrel{\eqref{eq:destimate}}{\leq } \frac{C \Delta t^2}{h^2}\norm{\bw_h^2-\bw_h^1}_{L^2}.
	\end{equation}
	For $B$, we have
	\begin{equation*}
\begin{split}
	\norm{B}_{L^2}&\leq C\Delta t \norm{\bd_h^2-\bd_h^1}_{L^2}\norm{\Grad\vphi^2_h}_{L^\infty}\norm{\Grad\vphi_h}_{L^\infty}\norm{\overline{\bd}_h^2}_{L^\infty}\\
	&\leq \frac{C\Delta t}{h^n} \norm{\bd_h^2-\bd_h^1}_{L^2}\norm{\Grad\vphi^2_h}_{L^2}\norm{\Grad\vphi_h}_{L^2},
\end{split}
	\end{equation*}
	using the inverse estimate~\eqref{eq:inverse}. Using then the energy estimate~\eqref{eq:discreteenergyestimate} and~\eqref{eq:destimate}, we obtain
	\begin{equation}
	\label{eq:B1}
	\norm{B}_{L^2}\leq \frac{C \Delta t^2}{h^n}\norm{\bw_h^2-\bw_h^1}_{L^2}.
	\end{equation}
	We turn to the term $D$:
	\begin{align}\label{eq:D1}
	\begin{split}
		\norm{D}_{L^2}&\leq C\Delta t \norm{\bd_h^2-\bd_h^1}_{L^2}\norm{\Grad\vphi_h}_{L^\infty}\norm{\Grad\vphi_h^2}_{L^\infty}\norm{\overline{\bd}_h^1}_{L^\infty}\\
	&\leq \frac{C\Delta t}{h^n} \norm{\bd_h^2-\bd_h^1}_{L^2}\norm{\Grad\vphi_h}_{L^2}\norm{\Grad\vphi_h^2}_{L^2}\\
	&\leq \frac{C\Delta t^2}{h^n} \norm{\bw_h^2-\bw_h^1}_{L^2},	
	\end{split}
\end{align}
	using again the inverse bound~\eqref{eq:inverse}, the energy estimate and~\eqref{eq:destimate}. 
	We next estimate term $E$:
	\begin{align*}
	\norm{E}_{L^2}&\leq C\Delta t\norm{\overline{\bd}_h^1}_{L^\infty}^2 \norm{\Grad\vphi_h}_{L^\infty}\norm{\Grad(\vphi_h^2-\vphi_h^1)}_{L^2(\dom)}\\
     &\stackrel{\eqref{eq:inverse}}{\leq}\frac{C\Delta t}{h^{\frac{n}{2}}}\norm{\Grad\vphi_h}_{L^2}\norm{\Grad(\vphi_h^2-\vphi_h^1)}_{L^2(\dom)}\\
     & \leq \frac{C\Delta t}{h^{\frac{n}{2}}}\norm{\Grad(\vphi_h^2-\vphi_h^1)}_{L^2(\dom)},
	\end{align*}
	where we used the energy estimate for the last inequality. Now using~\eqref{eq:phigradestimate}, we obtain
	\begin{equation}
	\label{eq:E1}
	\norm{E}_{L^2}\leq \frac{C\Delta t^2}{h^n}\norm{\bw_h^1-\bw_h^2}_{L^2}.
	\end{equation}
	Combining~\eqref{eq:A1}, \eqref{eq:A2}, \eqref{eq:B1}, \eqref{eq:D1},  and \eqref{eq:E1}, we obtain
	\begin{equation*}
	\norm{\bw_h^{2,\text{new}}-\bw_h^{1,\text{new}}}_{L^2}\leq \frac{C}{\disp+\frac{\damp\Delta t}{2}}\left(\frac{\Delta t^2}{h^2} + \frac{\Delta t^2}{h^n}\right)\norm{\bw_h^2-\bw_h^1}_{L^2}.
	\end{equation*}
	So if we take $\Delta t\leq \kappa h^{\max\{1,\nicefrac{n}{2}\}}$ for $\kappa$ sufficiently small 
	we obtain that $F$ is a contraction.
\end{proof}
The previous lemma will imply that the following algorithm converges. We denote $\theta =\max\{1,\nicefrac{n}{2}\}$.
\equationbox{
	\begin{definition}\label{def:fixedp}
		Given $h >0 $, $\Delta t = \kappa h^\theta$, and functions 
		$(\bd^m_h, \bw^m_h,\vphi_h^m)$ satisfying \eqref{eq:m2}, we approximate
		the next time-step $(d^{m+1}_h, w^{m+1}_h,\vphi_h^{m+1})$ to a given 
		tolerance $\varepsilon > 0$ by the following procedure: Set
		\begin{equation*}
		(\bd_h^{m, 0}, \bw_h^{m,0},\vphi_h^{m,0}) = (\bd_h^m, \bw_h^m,\vphi_h^m),
		\end{equation*}
		and iteratively solve $(\bd_h^{m,s+1}, \bw_h^{m,s+1},\vphi_h^{m,s+1})$ satisfying
		\begin{equation}\label{num:fixed}
		\begin{split}
		\frac{\bd^{m,s+1}_h-\bd^m_h}{\Delta t} &=  \bd_h^{m,s+\nicefrac{1}{2}}  \times \frac{1}{2}\left(\bw^{m}_h + \bw^{m,s}_h\right), \\
		u_{0,h}^{m,s+1}\in V_h,\quad\text{s.t.}\quad &a_h^{m,s+1}(u_{0,h}^{m,s+1},v)=F^{m+1}(v)-a_h^{m,s+1}(\widetilde{g}_h^{m+1},v),\quad \forall v\in V_h,\\
		\vphi_h^{m,s+1}& := u_{0,h}^{m,s+1}+\widetilde{g}_h^{m+1},\\		
		\disp\frac{\bw^{m,s+1}_h-\bw^m_h}{\Delta t} &= \oneconst \Delta_h \bd_h^{m,s+\nicefrac{1}{2}}  \times  \bd_h^{m,s+\nicefrac{1}{2}}-\damp \frac{\bw^{m,s+1}_h+\bw_h^m}{2}+S_h\times\bd_h^{m,s+\nicefrac{1}{2}}\end{split}
		\end{equation}
		where $\bd^{m,s+\nicefrac{1}{2}}_{h}:=(\bd^m_h+ \bd^{m,s+1}_h)/2$,
		\begin{equation*}
		a_h^{m,s+1}(v,w):=\int_{\dom}\left[\Grad v\cdot \Grad w+\epsii(\bd_h^{m,s+1}\cdot\Grad v)(\bd_h^{m,s+1}\cdot\Grad w)\right] dx,
		\end{equation*}
	and
		\begin{equation*}
		S_h(x):=\frac{\epsi}{2|\mathcal{C}_{\bi}|}\int_{\mathcal{C}_{\bi}}\Big((\Grad\vphi_h^{m,s+1} \cdot{\bd}^{m,s+\nicefrac{1}{2}}_h)\Grad\vphi_h^m
		+(\Grad\vphi_h^m \cdot\bd^{m,s+\nicefrac{1}{2}}_h)\Grad\vphi_h^{m,s+1}\Big) dy,\quad x\in \mathcal{C}_{\bi}.
		\end{equation*} 
		until the following stopping criterion is met:
		\begin{equation}\label{eq:stop}
		\left\|\bw_h^{m, s+1} - \bw_h^{m,s}\right\|_{L^2(\dom)} + \left\|\Grad_h \bd_h^{m,s+1} - \Grad_h \bd_h^{m,s}\right\|_{L^2(\dom)}+	\left\|\Grad \vphi_h^{m, s+1} - \Grad \vphi_h^{m,s}\right\|_{L^2(\dom)} < \varepsilon.
		\end{equation}		
	\end{definition}
}

We note that~\eqref{num:fixed} corresponds to the nonlinear mapping $\bw_h^{m,s}\mapsto F(\bw_h^{m,s})=\bw_h^{m,s+1}$.
\begin{theorem}
	\label{lem:convalgo}
	For every $\varepsilon_0>0$ there is a number of iterations $s\in\mathbb{N}$ such that the stopping criterion~\eqref{eq:stop} is satisfied with $\varepsilon=\varepsilon_0$ and moreover
	\begin{equation*}
	\norm{\bw_h^{m+1}-\bw_h^{m,s+1}}_{L^2}+\norm{\Grad_h\bd_h^{m+1}-\Grad_h\bd_h^{m,s+1}}_{L^2}+\norm{\Grad\vphi_h^{m+1}-\Grad\vphi_h^{m,s+1}}_{L^2}<\varepsilon_0.
	\end{equation*}
\end{theorem}
\begin{proof}
We let $\varepsilon>0$ arbitrary. Using that $\bw_h^{m+1}$ is a fixed point of $F$, we have
\begin{equation*}
\begin{split}
	\norm{\bw_h^{m+1}-\bw_h^{m,s}}_{L^2}& = 	\norm{F(\bw_h^{m+1})-F(\bw_h^{m,s-1})}_{L^2}\\
	&\stackrel{\text{Lem. } \ref{lem:contraction}}{\leq} q	\norm{\bw_h^{m+1}-\bw_h^{m,s-1}}_{L^2}\\
	&\leq q^{s}	\norm{\bw_h^{m+1}-\bw_h^{m}}_{L^2},
\end{split}
\end{equation*}
applying Lemma~\ref{lem:contraction} $s-1$ more times. Using the energy estimate, we then obtain
\begin{equation*}
\norm{\bw_h^{m+1}-\bw_h^{m,s}}_{L^2}\leq 2 q^s \sqrt{\frac{E_0}{\disp}}.
\end{equation*}
Using an inverse estimate and~\eqref{eq:destimate}, we also obtain
\begin{equation}
\begin{split}
\norm{\Grad_h\bd_h^{m,s}-\Grad_h\bd_h^{m+1}}_{L^2}&\leq \frac{C}{h}\norm{V(\bw_{h}^{m+\nicefrac{1}{2}})\bd_h^m-V\left((\bw_h^{m}+\bw_h^{m,s-1})/2\right)\bd_h^m}_{L^2}\\
&\leq \frac{C\Delta t}{h}\norm{\bw_h^{m+1}-\bw_h^{m,s-1}}_{L^2}\\
&\leq \frac{C\Delta t}{h} q^{s-1} \sqrt{\frac{E_0}{\disp}}\\
&\leq C q^s \sqrt{E_0},
\end{split}
\end{equation}
where the last inequality follows from the CFL-condition~\eqref{eq:CFLnew}.
Similarly, using~\eqref{eq:phigradestimate}, we obtain
\begin{equation*}
\begin{split}
\norm{\Grad\vphi^{m+1}_h-\Grad\vphi_h^{m,s}}_{L^2}&\leq  C \frac{\Delta t}{h^{\frac{n}{2}}}\norm{\bw_h^{m+1}-\bw_h^{m,s-1}}_{L^2}\\
&\leq \frac{C\Delta t}{h^{\frac{n}{2}}} q^{s-1} \sqrt{\frac{E_0}{\disp}}\\
&\leq C q^s \sqrt{E_0},
\end{split}
\end{equation*}
where the last inequality again follows from the CFL-condition. Thus,
\begin{equation*}
\norm{\bw^{m+1}_h-\bw_h^{m,s}}_{L^2}+\norm{\Grad\vphi^{m+1}_h-\Grad\vphi_h^{m,s}}_{L^2}+\norm{\Grad_h\bd^{m+1}_h-\Grad_h\bd_h^{m,s}}_{L^2}\leq C q^s \sqrt{E_0}.
\end{equation*}
Next, by triangle inequality
\begin{equation*}
\norm{\bw^{m,s+1}_h-\bw_h^{m,s}}_{L^2}\leq \norm{\bw^{m+1}_h-\bw_h^{m,s+1}}_{L^2}+\norm{\bw^{m+1}_h-\bw_h^{m,s}}_{L^2}\leq 4 q^s \sqrt{\frac{E_0}{\disp}}.
\end{equation*}
By the same argument, we obtain,
\begin{equation*}
\norm{\Grad_h\bd^{m,s+1}_h-\Grad_h\bd_h^{m,s}}_{L^2}\leq 2 Cq^s \sqrt{E_0}\quad\text{ and }\quad \norm{\Grad\vphi^{m,s+1}_h-\Grad\vphi_h^{m,s}}_{L^2}\leq 2 Cq^s \sqrt{E_0}.
\end{equation*}
Hence
\begin{equation*}
\norm{\bw^{m,s+1}_h-\bw_h^{m,s}}_{L^2}+\norm{\Grad\vphi^{m,s+1}_h-\Grad\vphi_h^{m,s}}_{L^2}+\norm{\Grad_h\bd^{m,s+1}_h-\Grad_h\bd_h^{m,s}}_{L^2}\leq C q^s \sqrt{E_0},
\end{equation*}
and we can pick $s$ large enough such that $C q^s \sqrt{E_0}<\varepsilon_0$. 
\end{proof}
This proves that the system~\eqref{eq:m2} has a unique solution that can be found for example by fixed point iteration or some other iterative procedure.
\section{Numerical experiments}\label{sec:num}
In the following, we test the numerical scheme on several examples for a 2-dimensional spatial domain. We assume that there are no free charges and set $f\equiv 0$. For all experiments, we will use the parameters
\begin{equation*}
h = \frac{1}{64},\quad \Dt = h\frac{\sqrt{\damp h^2+\disp}}{10},\quad  \dom = [-0.5,0.5]^2,\quad \disp=\frac12,\quad \oneconst =1,\quad \varepsilon = \frac{h^2}{20};
\end{equation*}
and
\begin{equation*}
\widetilde{g}(t,x,y) = 10\sin(2\pi t+0.2) (x+0.5)\sin(\pi y),\quad \bd_1\equiv \bw_0\equiv 0.
\end{equation*}
while varying the other parameters and initial data.

\medskip

The MATLAB code used to run the numerical experiments can be found at\\
 \url{https://github.com/fraenschii/director_efield_dynamics}.
\subsection{Alignment with electric field}
In the first numerical experiment, we test the claim that the director field aligns parallel to the director field when $\eps_1,\eps_2>0$ (corresponding to $\eps_a>0$) and perpendicular when $\eps_1,\eps_2<0$ (corresponding to $\eps_a<0$). We use $\damp =2$, and run the experiment up to time $T=2$. As initial data, we use 
\begin{equation*}
\bd_0(x)\equiv \frac{1}{\sqrt{2}}\begin{pmatrix}
1\\1\\0.
\end{pmatrix}
\end{equation*}
(So here the director field can only vary in the plane.) We first use $\epsi=5$ and $\epsii=0.5$ and then in a second experiment $\epsi=-5$ and $\epsii=-0.5$. We plot the approximations to the director field and the electric field in Figure~\ref{fig:exp1pos} and Figure~\ref{fig:exp1neg}.
\begin{figure}[ht]
	\centering
\begin{tabular}{lcr}
	\hspace{-1.5em}\subfloat[$t=0.25$]{\includegraphics[width=0.4\linewidth]{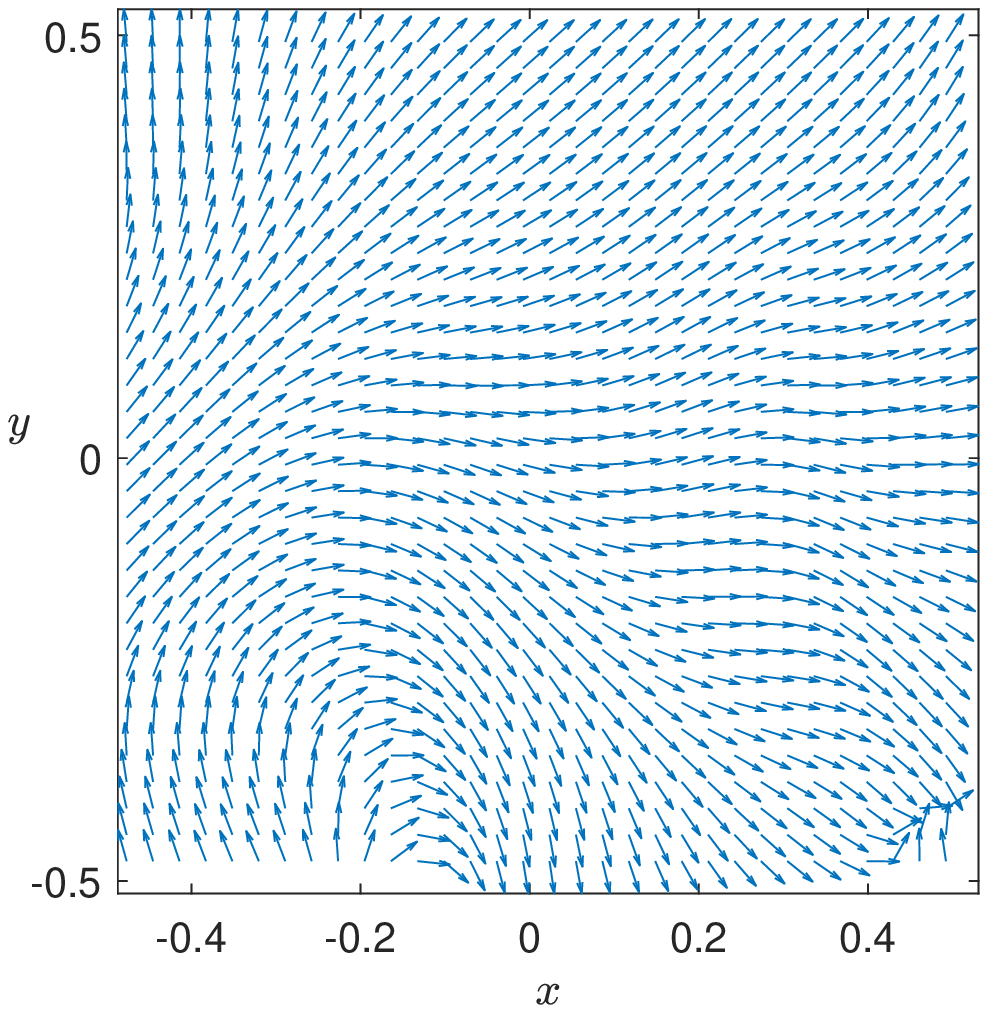}}\hspace{-2.25em}
	\subfloat[$t=0.5$]{\includegraphics[width=0.4\linewidth]{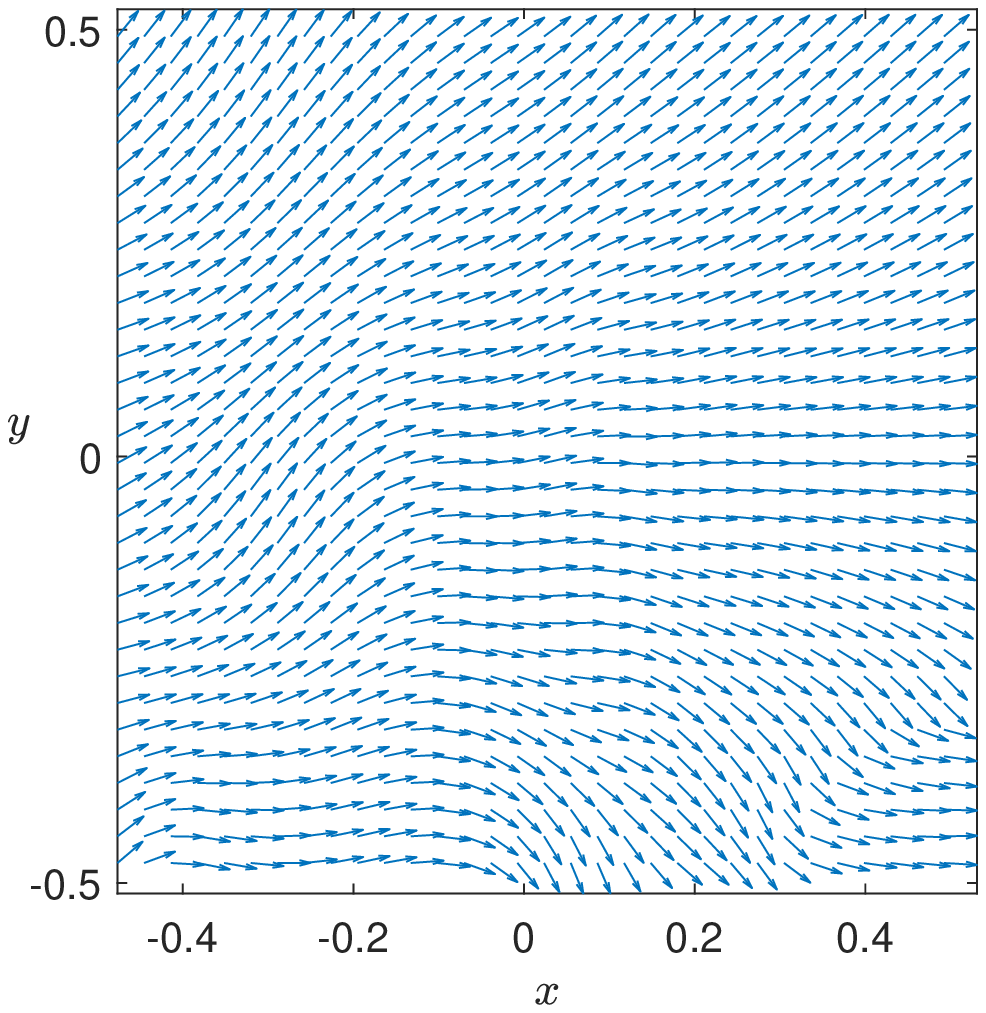}}\hspace{-2.25em}
	\subfloat[$t=2$]{\includegraphics[width=0.4\linewidth]{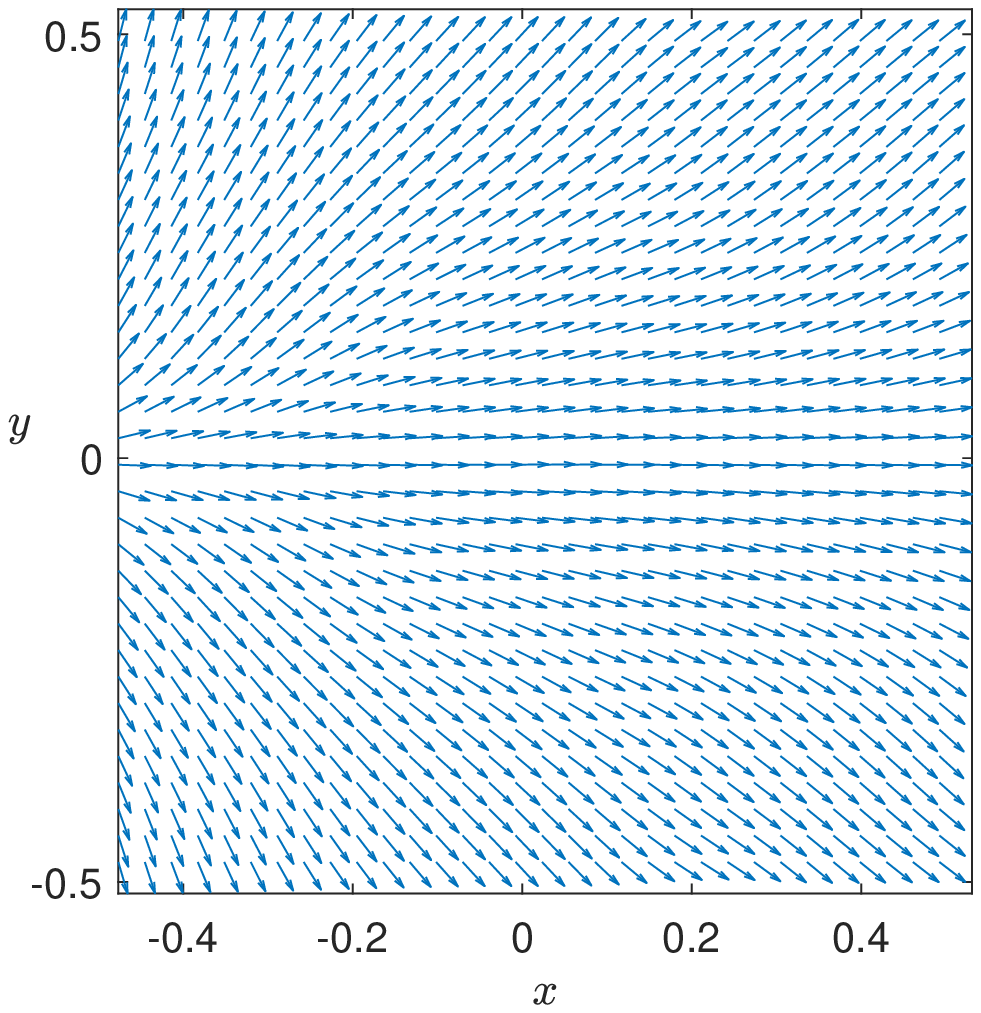}} \\
	\hspace{-1.5em}\subfloat[$t=0.25$]{\includegraphics[width=0.4\linewidth]{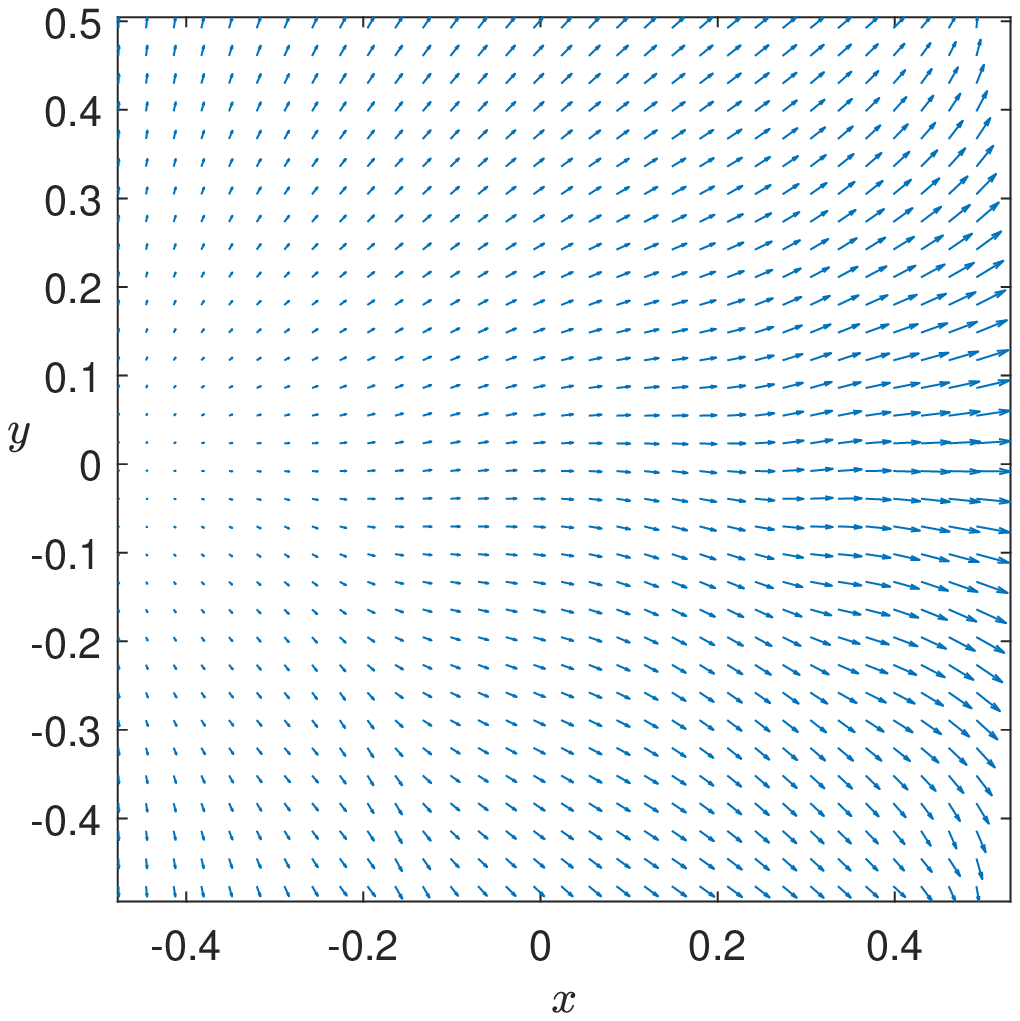}}\hspace{-2.25em}
	\subfloat[$t=0.5$]{\includegraphics[width=0.4\linewidth]{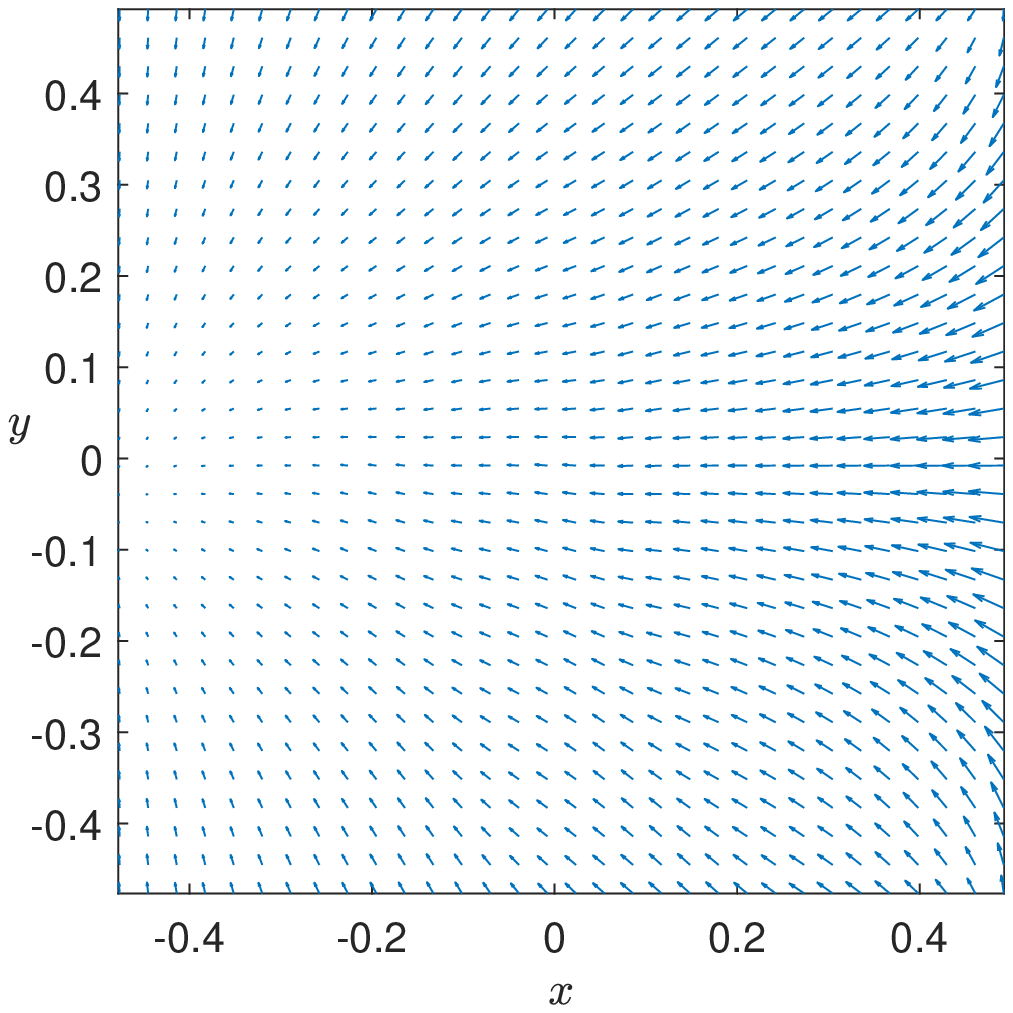}}\hspace{-2.25em}
	\subfloat[$t=2$]{\includegraphics[width=0.4\linewidth]{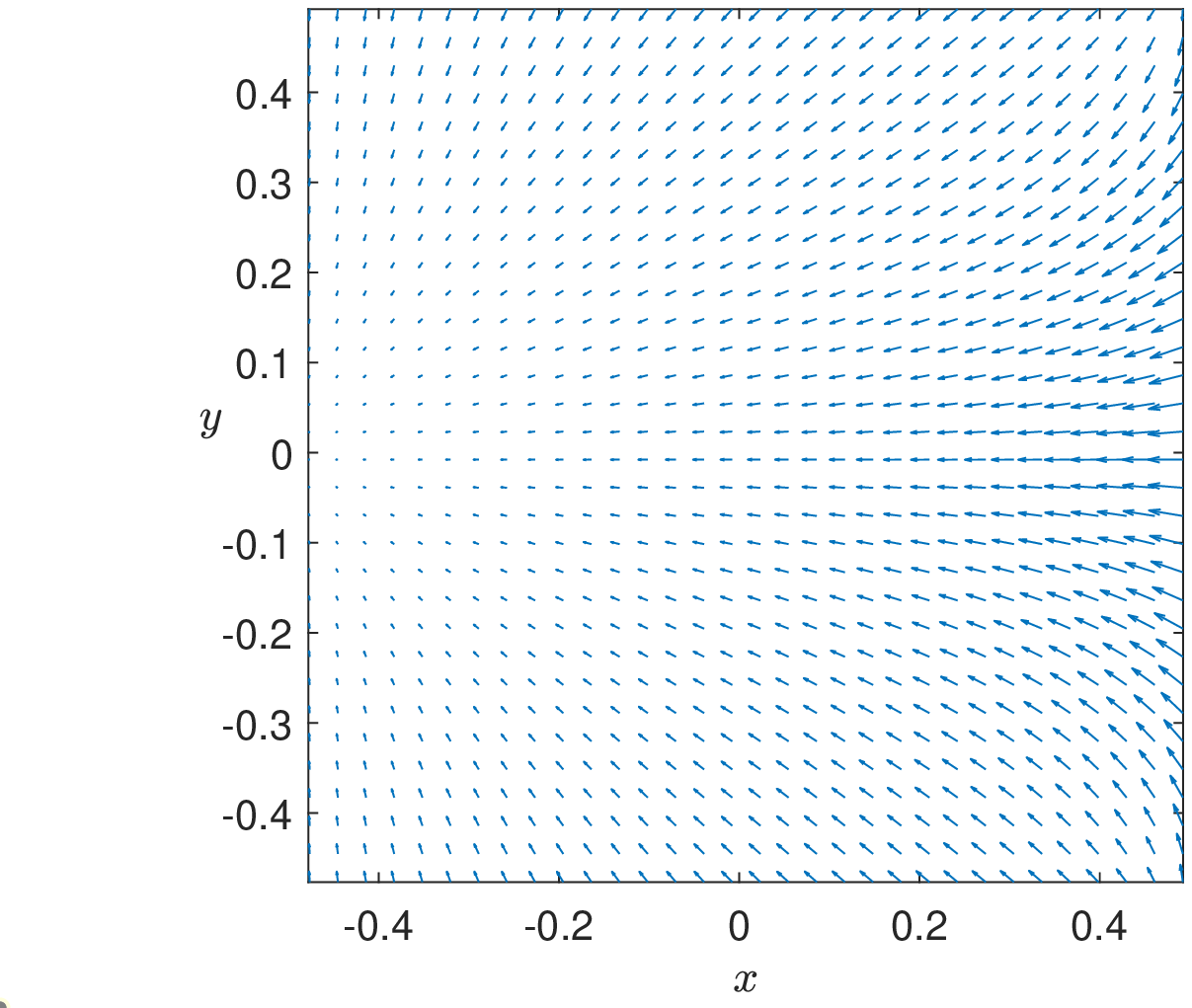}} 
\end{tabular}
	\caption{Experiment 1 with $\epsi=5$ and $\epsii=0.5$. Top row: Director field $\bd_h$, bottom row: Electric field $\Grad\vphi_h$.}
	\label{fig:exp1pos}
\end{figure}
\begin{figure}[ht]
	\centering
	\begin{tabular}{lcr}
		 \hspace{-1.5em}\subfloat[$t=0.25$]{\includegraphics[width=0.4\linewidth]{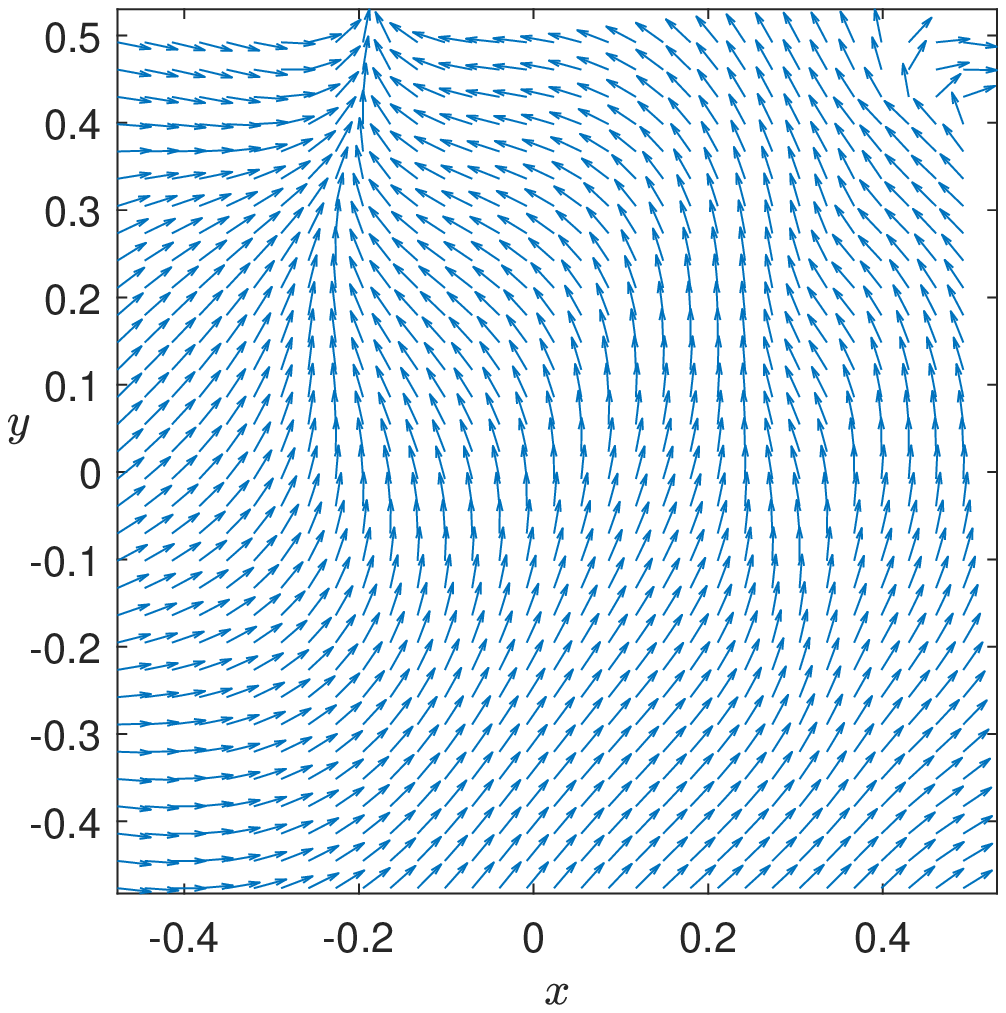}}\hspace{-2.25em}
		 		\subfloat[$t=0.5$]{\includegraphics[width=0.4\linewidth]{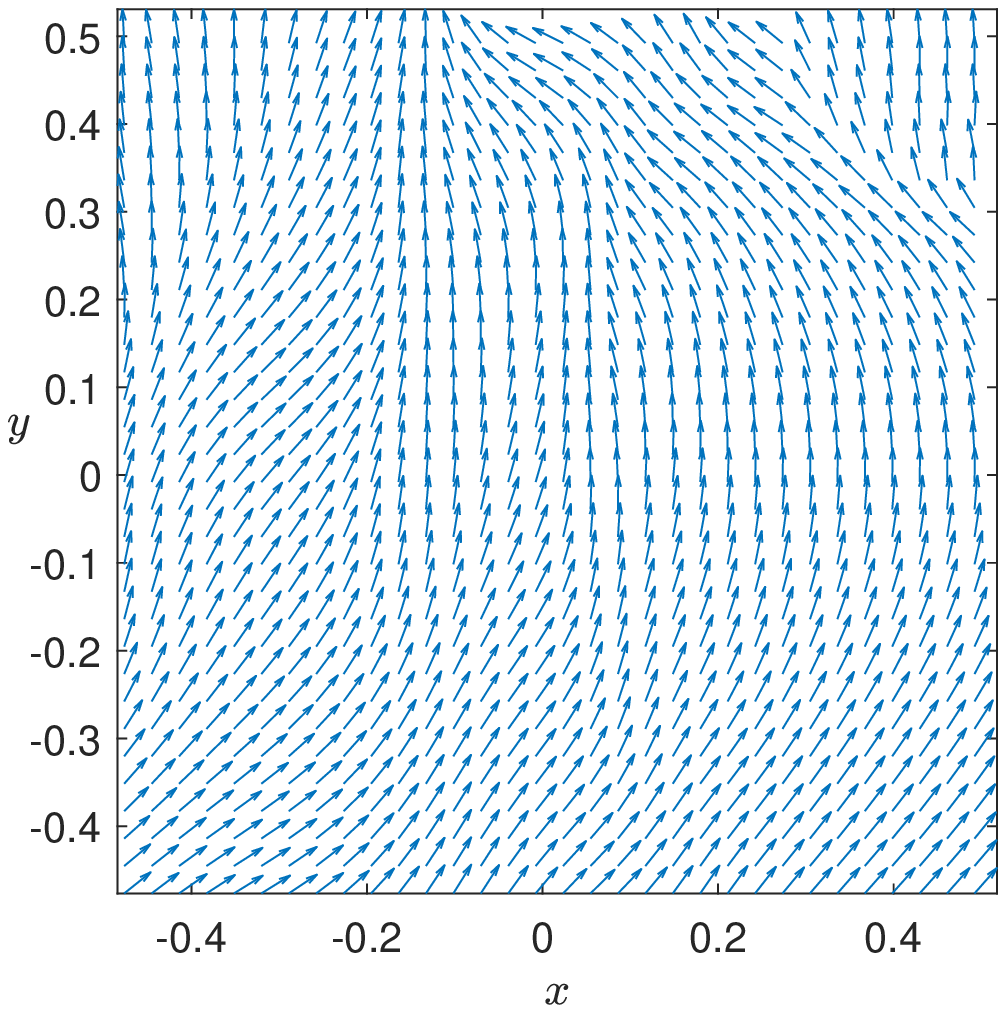}}\hspace{-2.25em}
		\subfloat[$t=2$]{\includegraphics[width=0.4\linewidth]{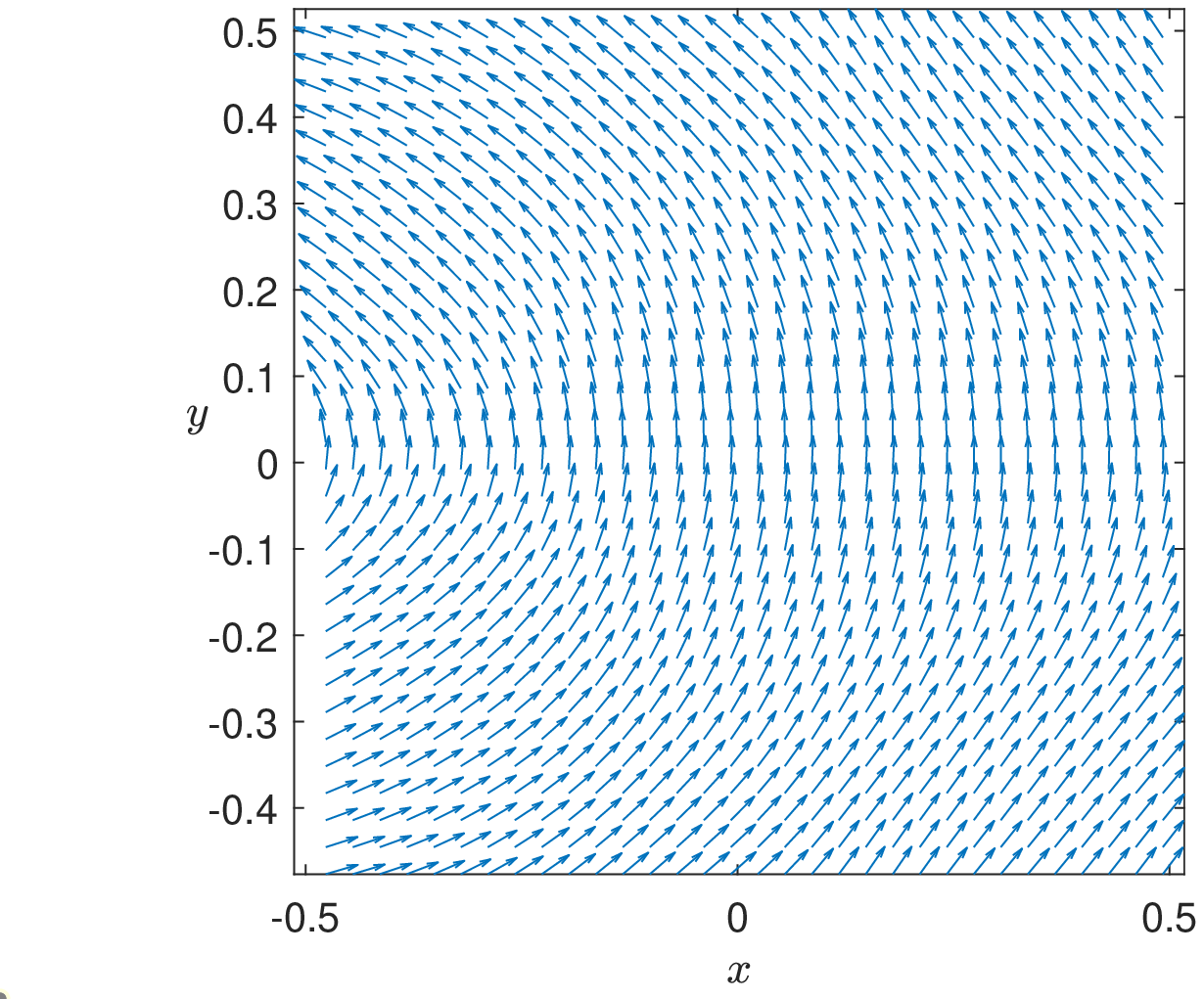}} \\
		 \hspace{-1.5em}\subfloat[$t=0.25$]{\includegraphics[width=0.4\linewidth]{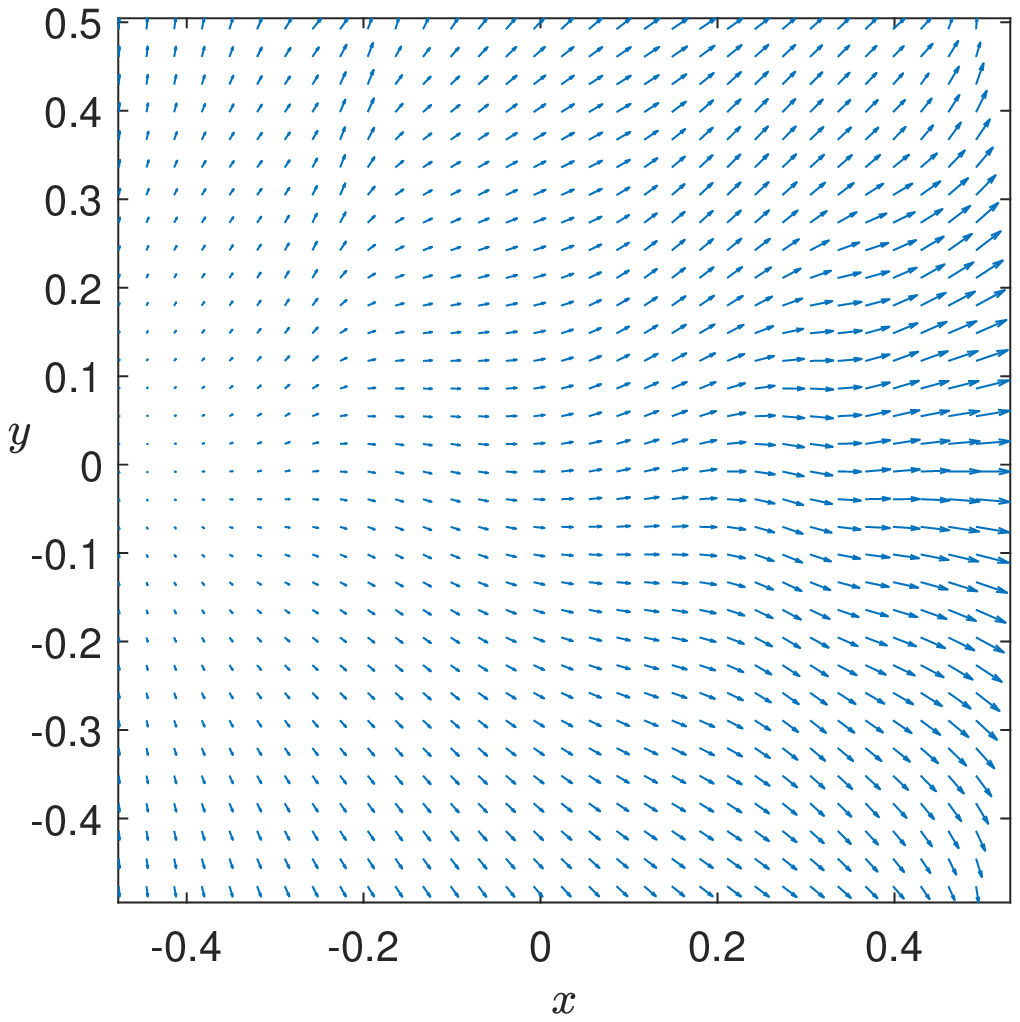}}\hspace{-2.25em}
		\subfloat[$t=0.5$]{\includegraphics[width=0.4\linewidth]{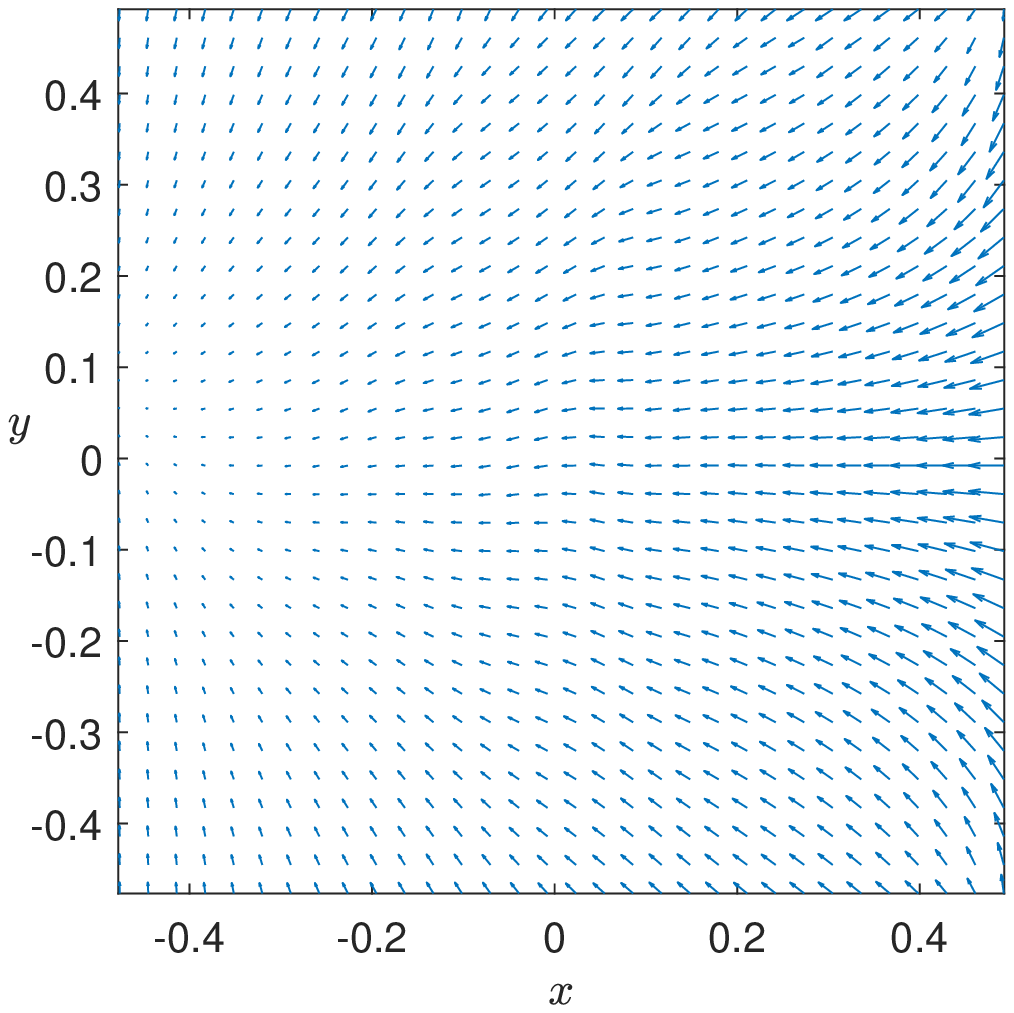}}\hspace{-2.25em}
		\subfloat[$t=2$]{\includegraphics[width=0.4\linewidth]{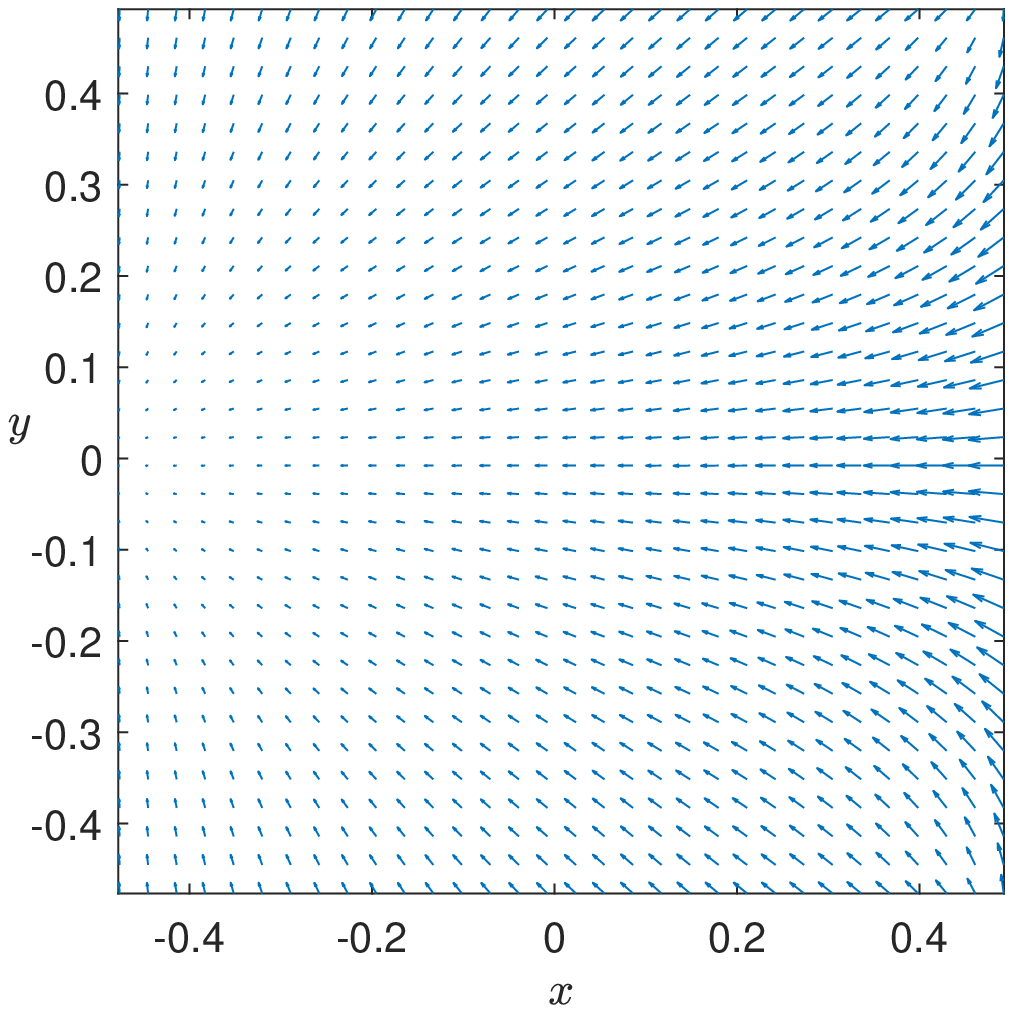}} 
	\end{tabular}
	\caption{Experiment 1 with $\epsi=-5$ and $\epsii=-0.5$. Top row: Director field $\bd_h$, bottom row: Electric field $\Grad\vphi_h$.}
	\label{fig:exp1neg}
\end{figure}

The electric field does not appear to be impacted by the director field very much, while the director field reacts quickly to the electric field and starts aligning in parallel, when the parameters $\eps_i$ are positive, and perpendicular when the parameters are negative, as predicted. 
In Figure~\ref{fig:exp1energy}, we also plotted the evolution of the `reduced' energy
\begin{equation}\label{eq:reducedenergy}
E_m = \frac12\int_{\dom} \left[k|\Grad_h\bd_h^m|^2+\disp|\bw_h^m|^2\right] dx
\end{equation}
and the effect of the damping
\begin{equation}\label{eq:damping}
\text{damping}(t) = \damp \int_0^t\int_{\dom} |\wh{\bw}_h|^2 dx ds,
\end{equation}
over time. We observe that after an initial increase of the energy due to the onset of the electric field, the damping factor leads to a significant decrease of the energy despite the oscillating electric field.
\begin{figure}[ht]
	\begin{tabular}{lr}
		\subfloat[]{\includegraphics[width=0.5\linewidth]{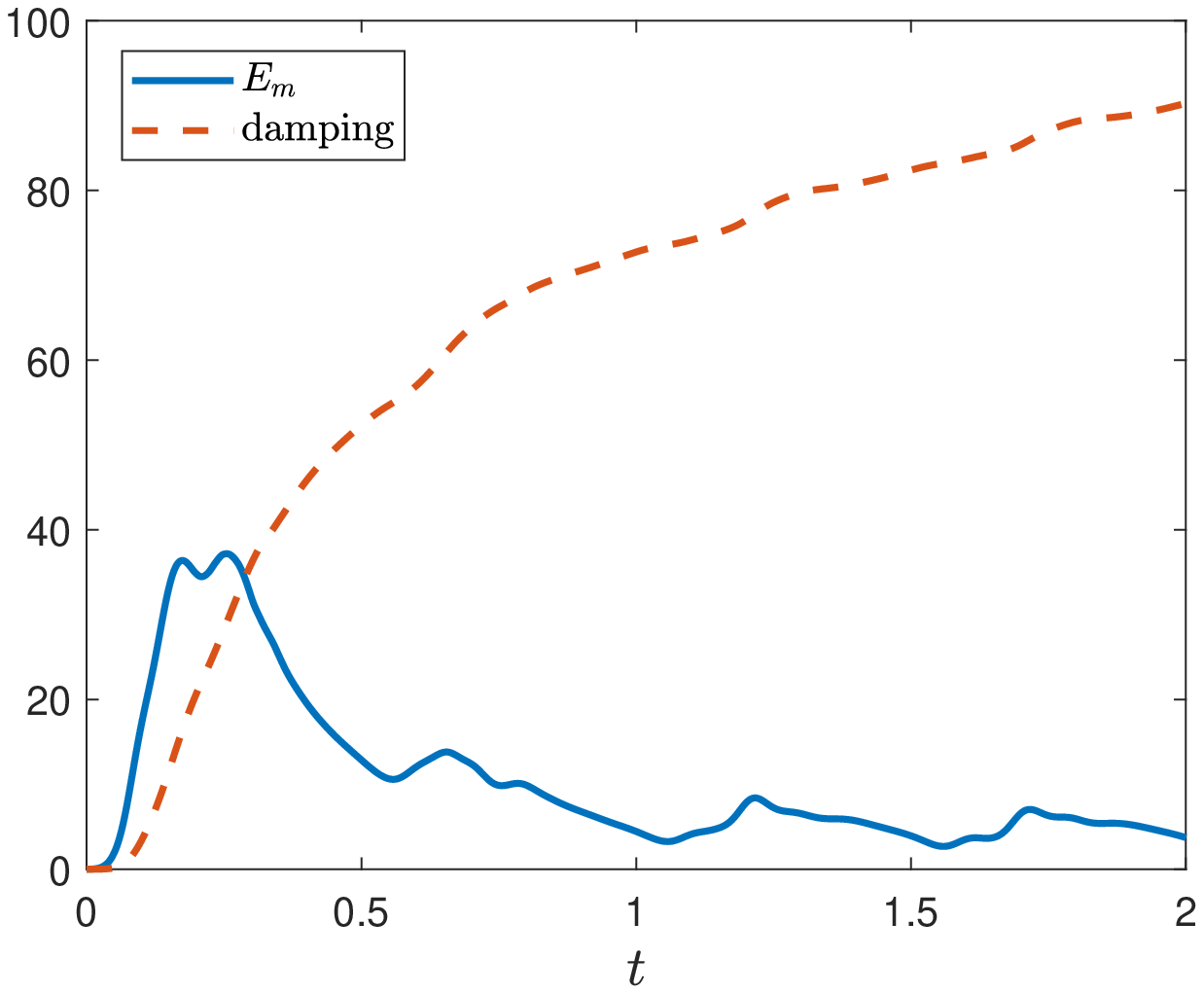}} 
		\subfloat[]{\includegraphics[width=0.5\linewidth]{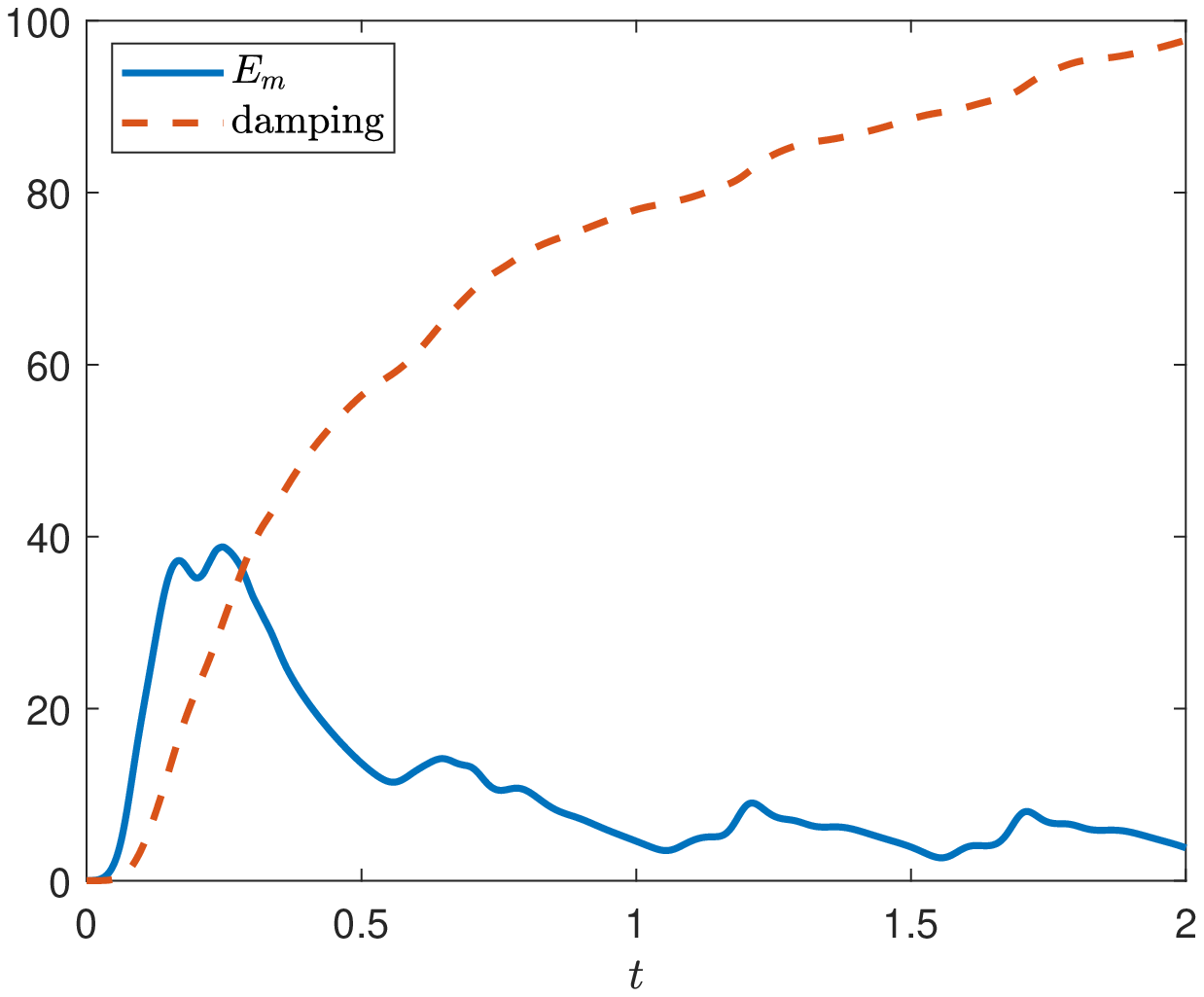}}
	\end{tabular}
	\caption{Experiment 1: Evolution of energy and damping term over time. Left: $\epsi=5$ and $\epsii=0.5$, right: $\epsi=-5$ and $\epsii=-0.5$.}
	\label{fig:exp1energy}
\end{figure}

\subsection{Development of singularities and damping}
For our second experiment, we use the initial data
\begin{equation}\label{eq:exp3init}
\begin{split}
\bd_0(x)=\begin{cases}
(0,0,-1)^\top,\quad &r\geq 1/2,\\
(2 x_1 a,2 x_2 a, a^2-r^2)^\top /(a^2+r^2),\quad &r<1/2,
\end{cases}
\end{split}   
\end{equation}
where $r=\sqrt{x_1^2+x_2^2}$ and $a(r)=(1-2r)^4$. 
\begin{figure}[ht]
	\begin{tabular}{lr}
		\subfloat[$\bd_0$]{\includegraphics[width=0.5\linewidth]{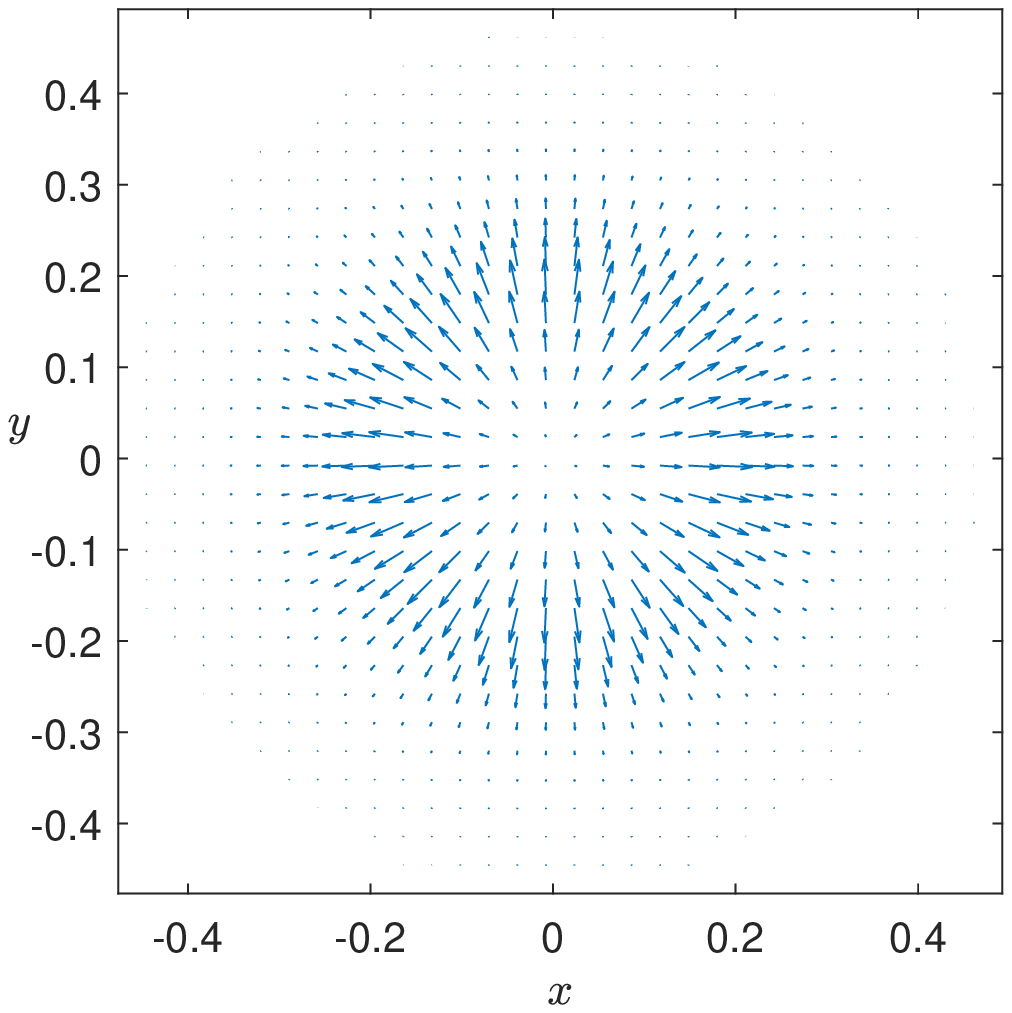}} 
		\subfloat[$\Grad\vphi_h^0$]{\includegraphics[width=0.5\linewidth]{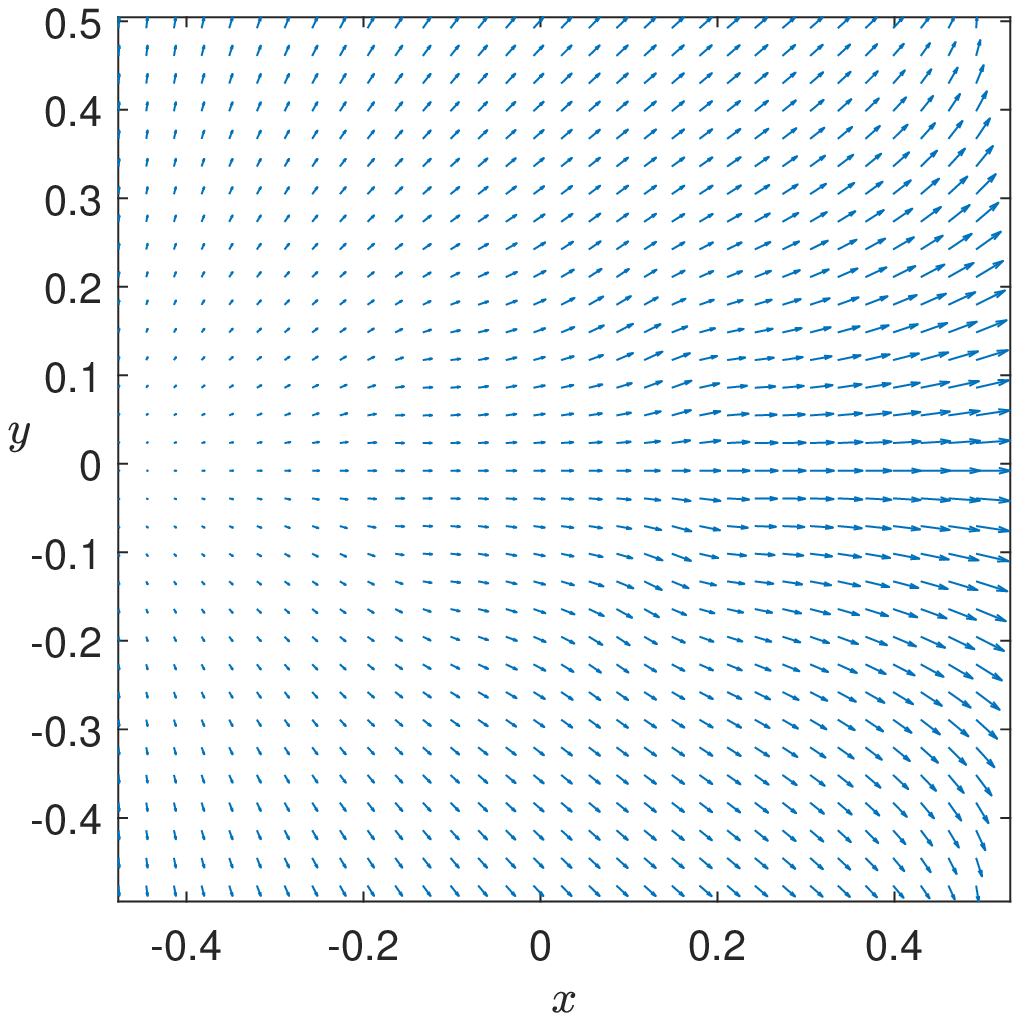}}
	\end{tabular}
	\caption{Experiment 2: Initial data for $\bd$~\eqref{eq:exp3init} and the initial approximation of the electric field $\Grad\vphi_h^0$.}
	\label{fig:exp2initdata}
\end{figure}
This initial data has been used in~\cite{Bartels2007,Karper2014} to demonstrate development of singularities in the wave map equation to the sphere. We use $\epsi=-5$ and $\epsii=-0.5$ and $\damp=0.5, 3$. We compute up to time $T=1$. 
\begin{figure}[ht]
	\centering
	\begin{tabular}{lccr}
		\hspace{-2.em}\subfloat[$t=0.25$]{\includegraphics[width=0.31\linewidth]{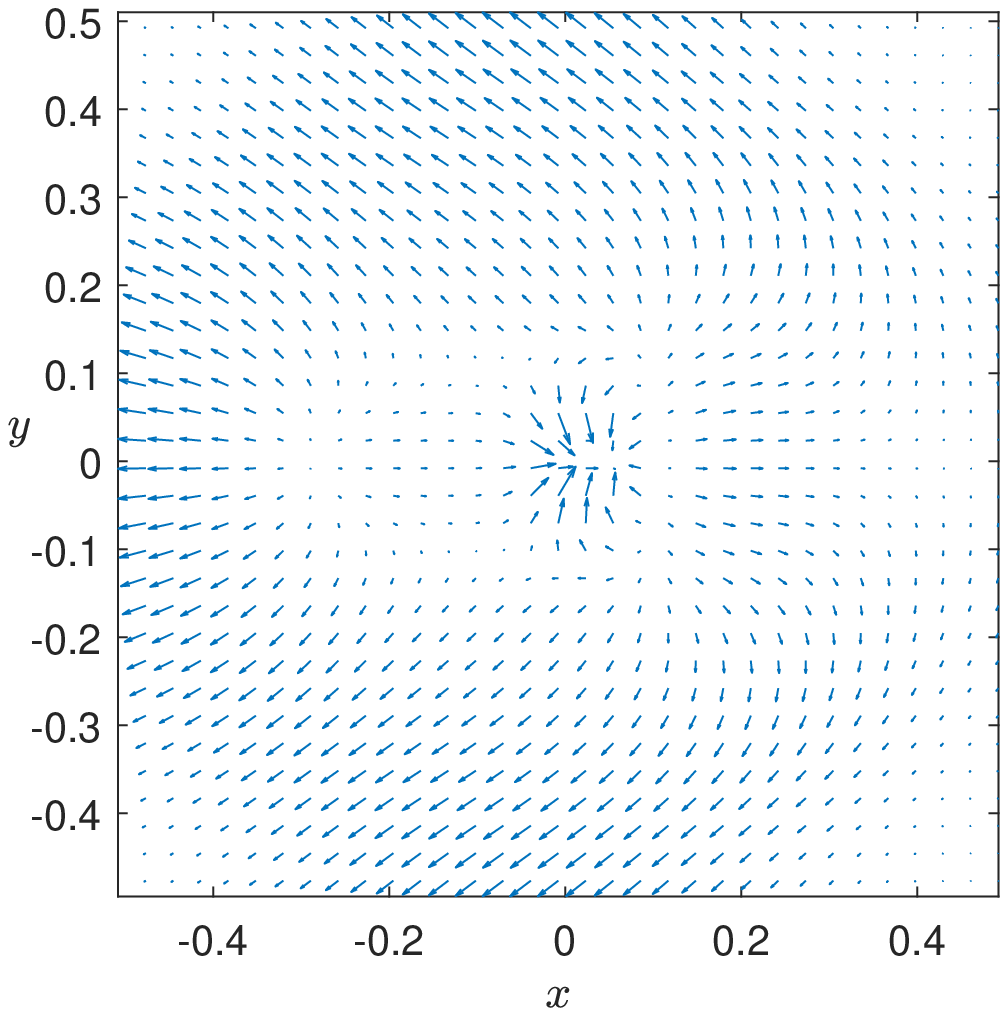}}\hspace{-2.25em}
		\subfloat[$t=0.5$]{\includegraphics[width=0.31\linewidth]{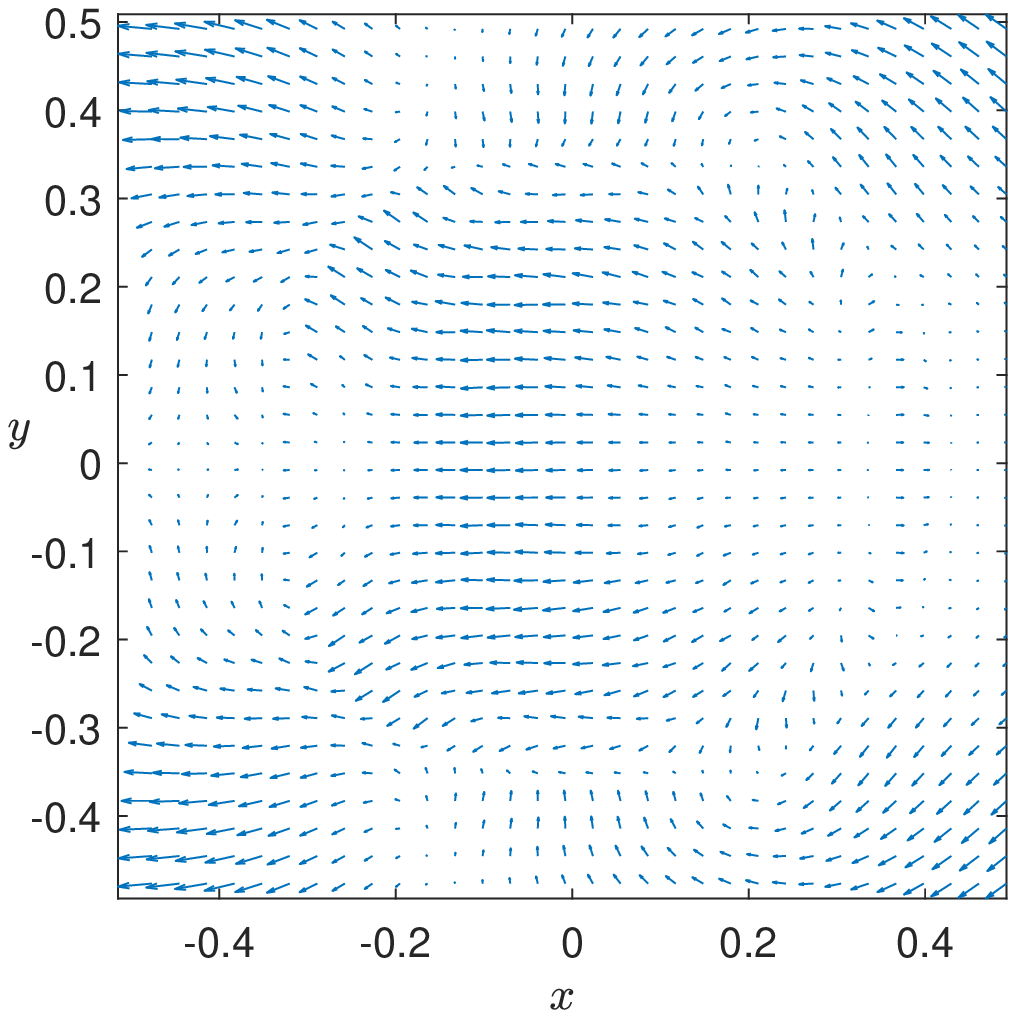}}\hspace{-2.25em}
		\subfloat[$t=0.75$]{\includegraphics[width=0.31\linewidth]{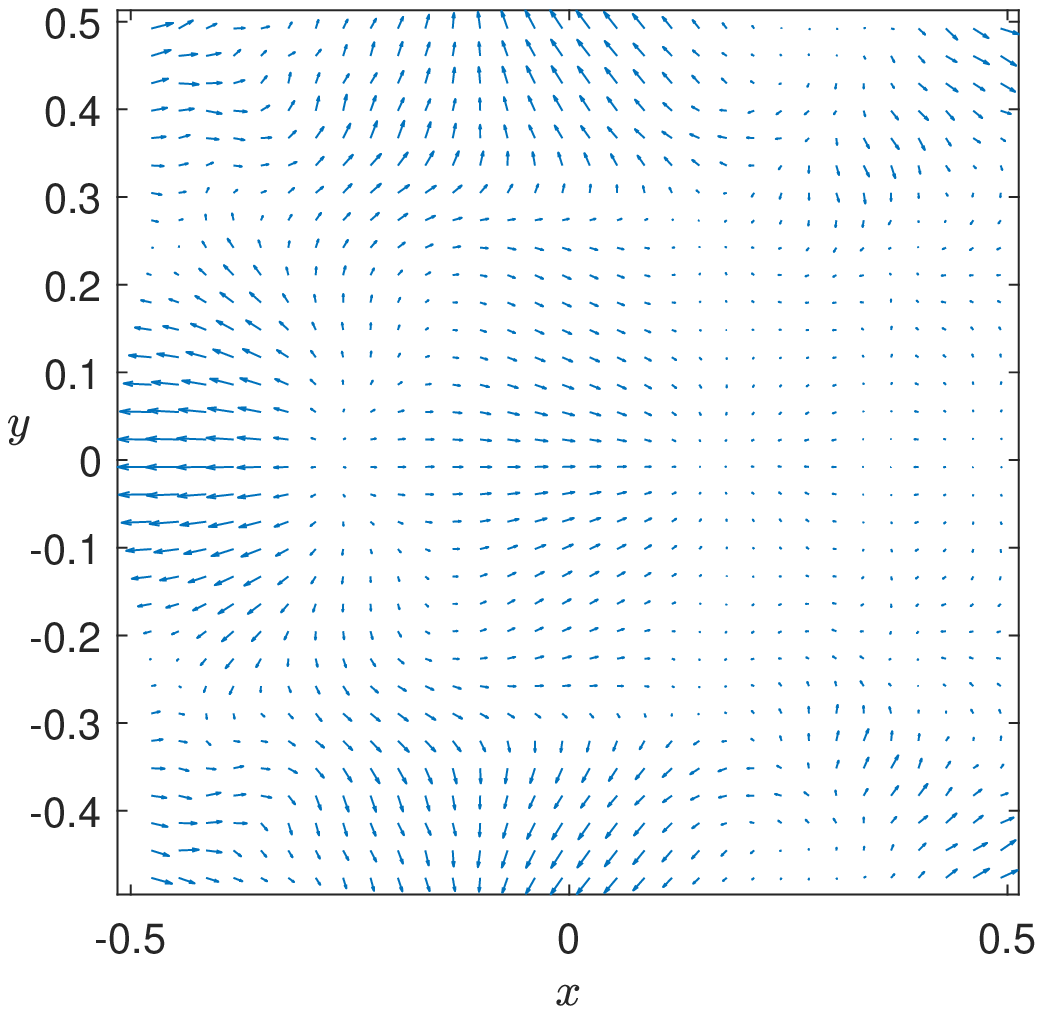}}\hspace{-2.25em}
		\subfloat[$t=1$]{\includegraphics[width=0.31\linewidth]{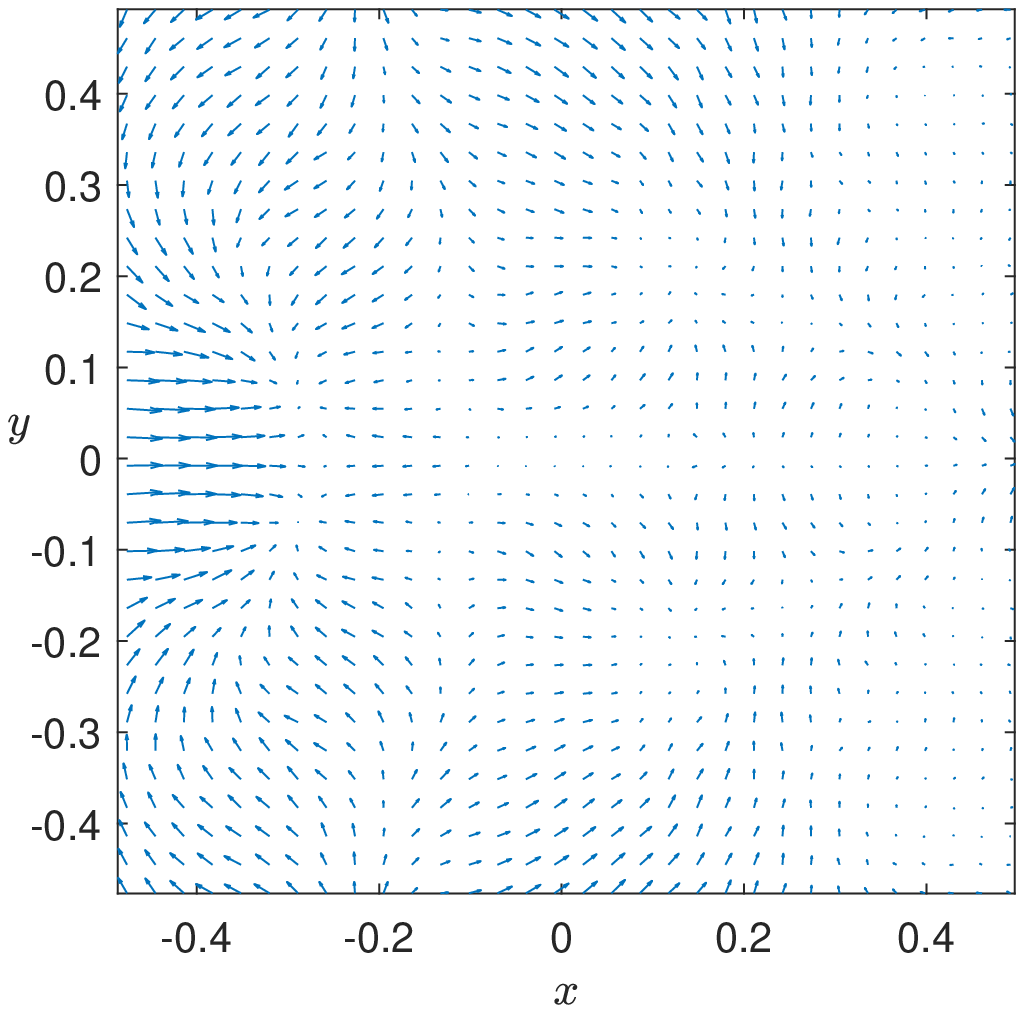}} \\
		\hspace{-2.em}\subfloat[$t=0.25$]{\includegraphics[width=0.31\linewidth]{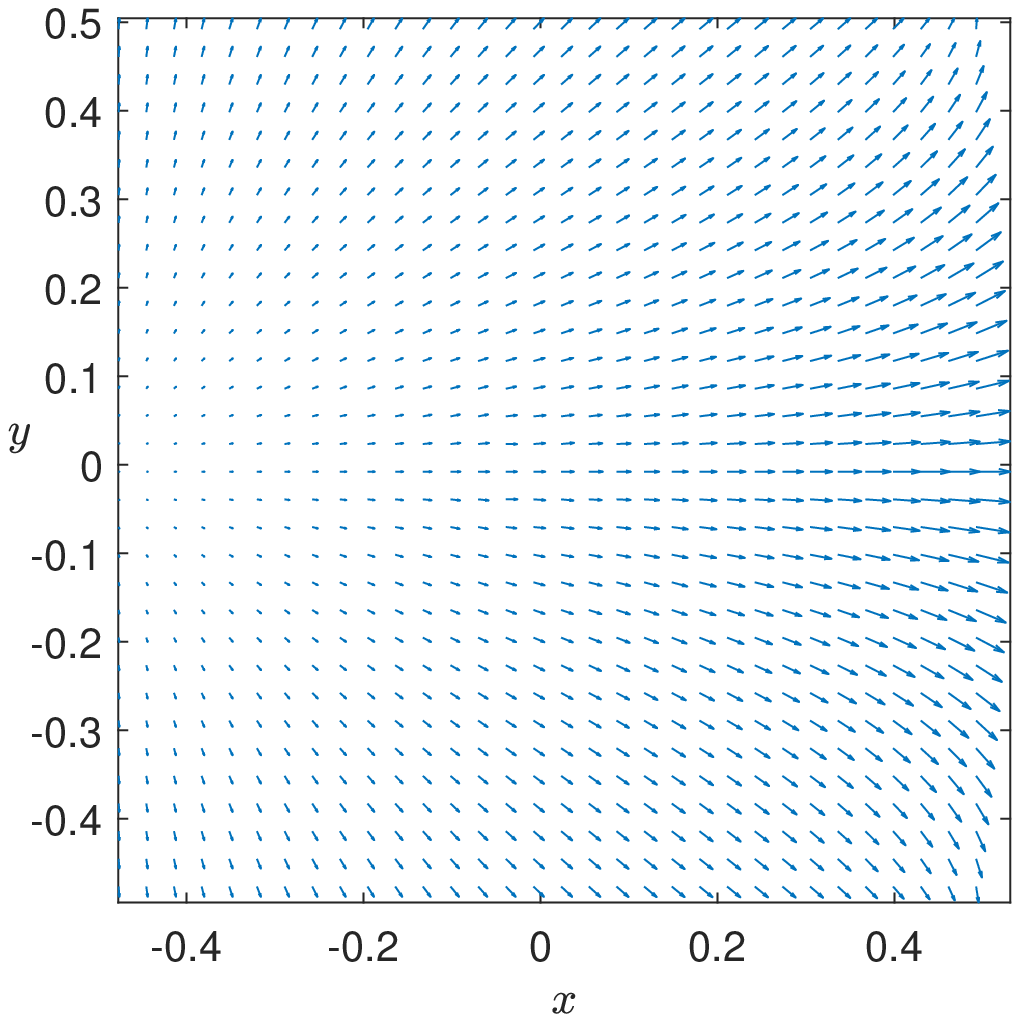}}\hspace{-2.25em}
		\subfloat[$t=0.5$]{\includegraphics[width=0.31\linewidth]{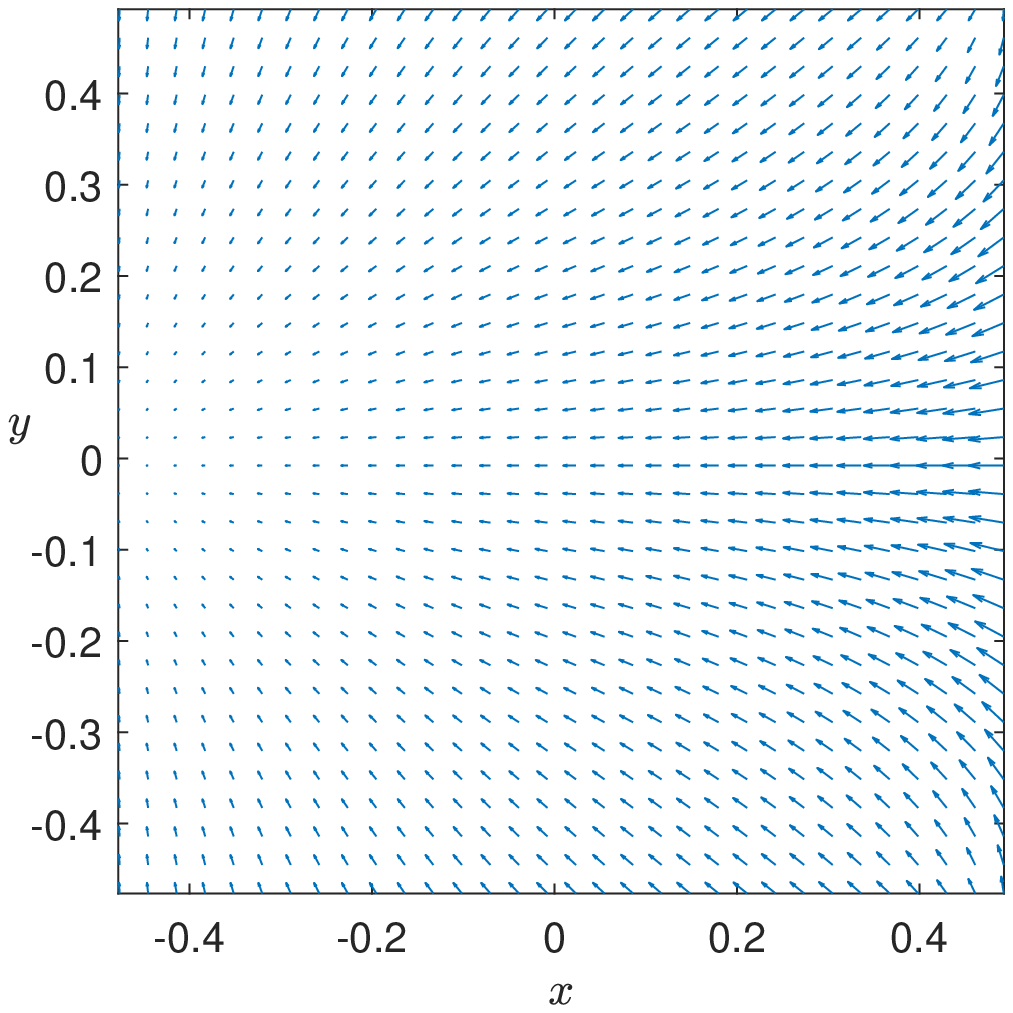}}\hspace{-2.25em}
		\subfloat[$t=0.75$]{\includegraphics[width=0.31\linewidth]{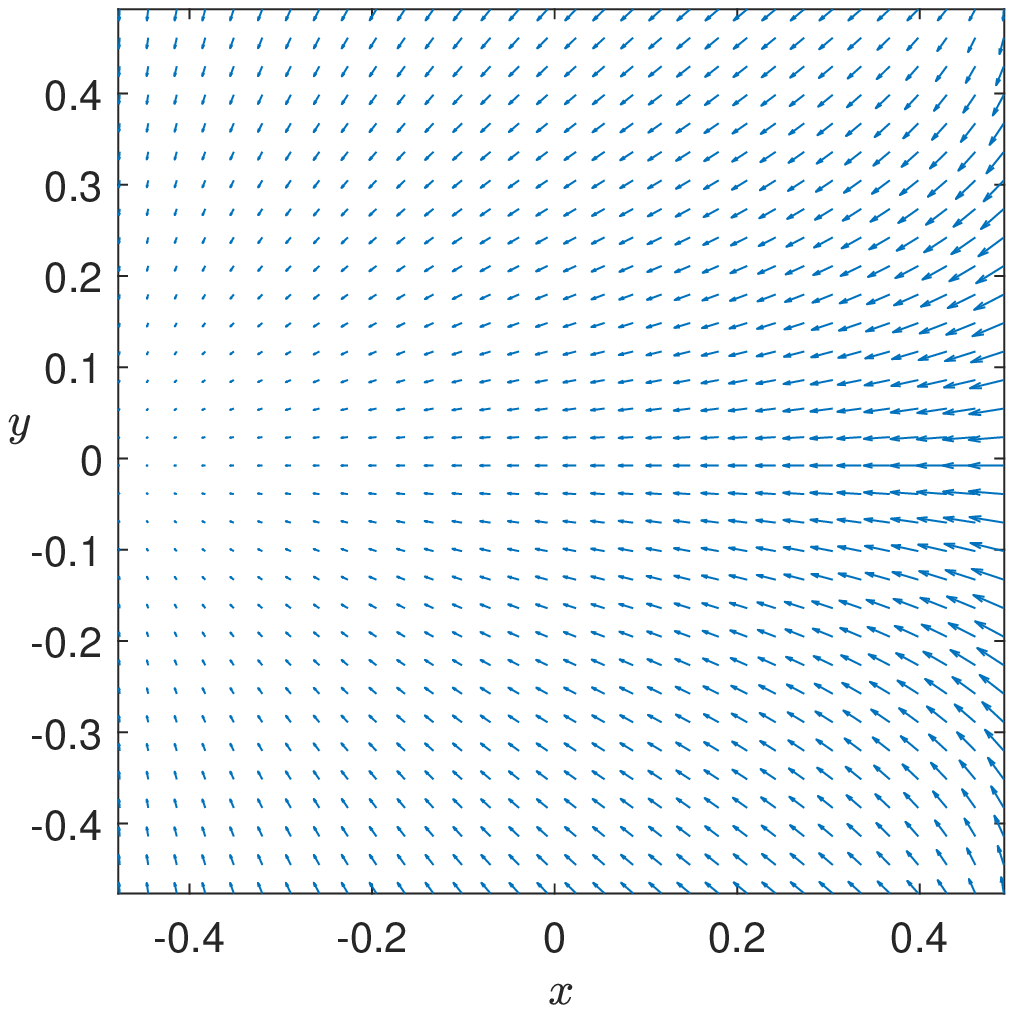}}\hspace{-2.25em}
		\subfloat[$t=1$]{\includegraphics[width=0.31\linewidth]{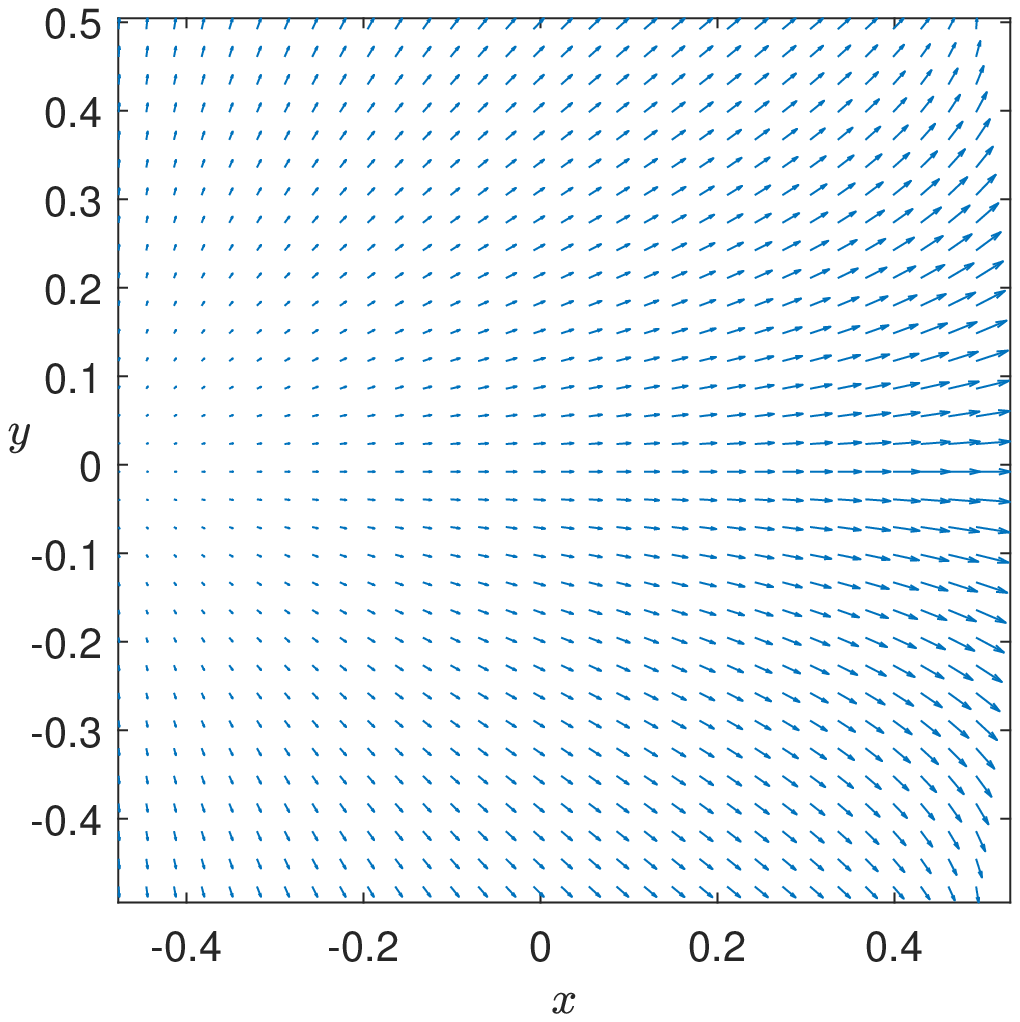}} 
	\end{tabular}
	\caption{Experiment 2 with $\damp = 0.5$. Top row: Director field $\bd_h$, bottom row: Electric field $\Grad\vphi_h$.}
	\label{fig:exp2smalldamp}
\end{figure}

\begin{figure}[ht]
	\centering
	\begin{tabular}{lccr}
		\hspace{-2.em}\subfloat[$t=0.25$]{\includegraphics[width=0.31\linewidth]{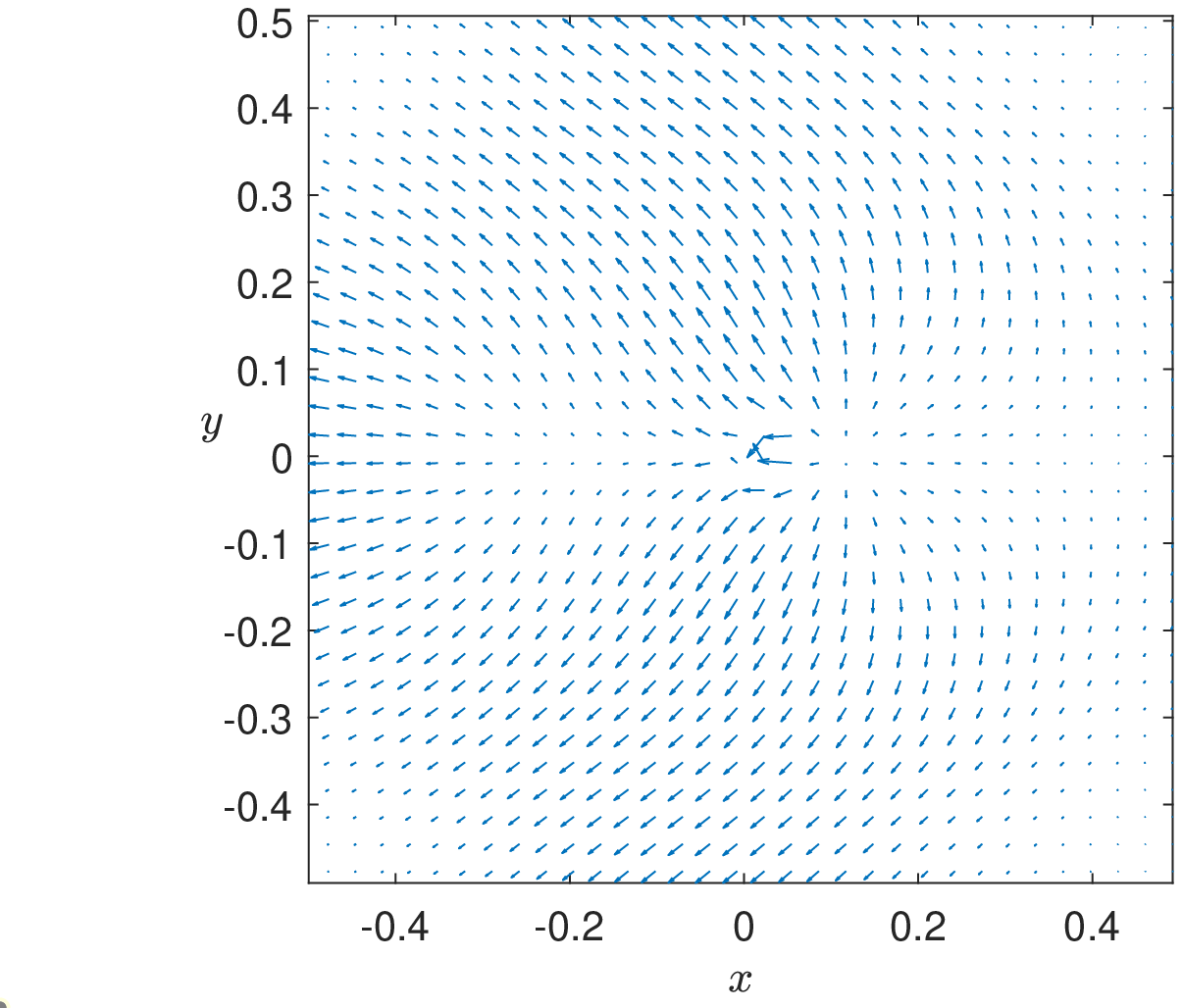}}\hspace{-2.25em}
		\subfloat[$t=0.5$]{\includegraphics[width=0.31\linewidth]{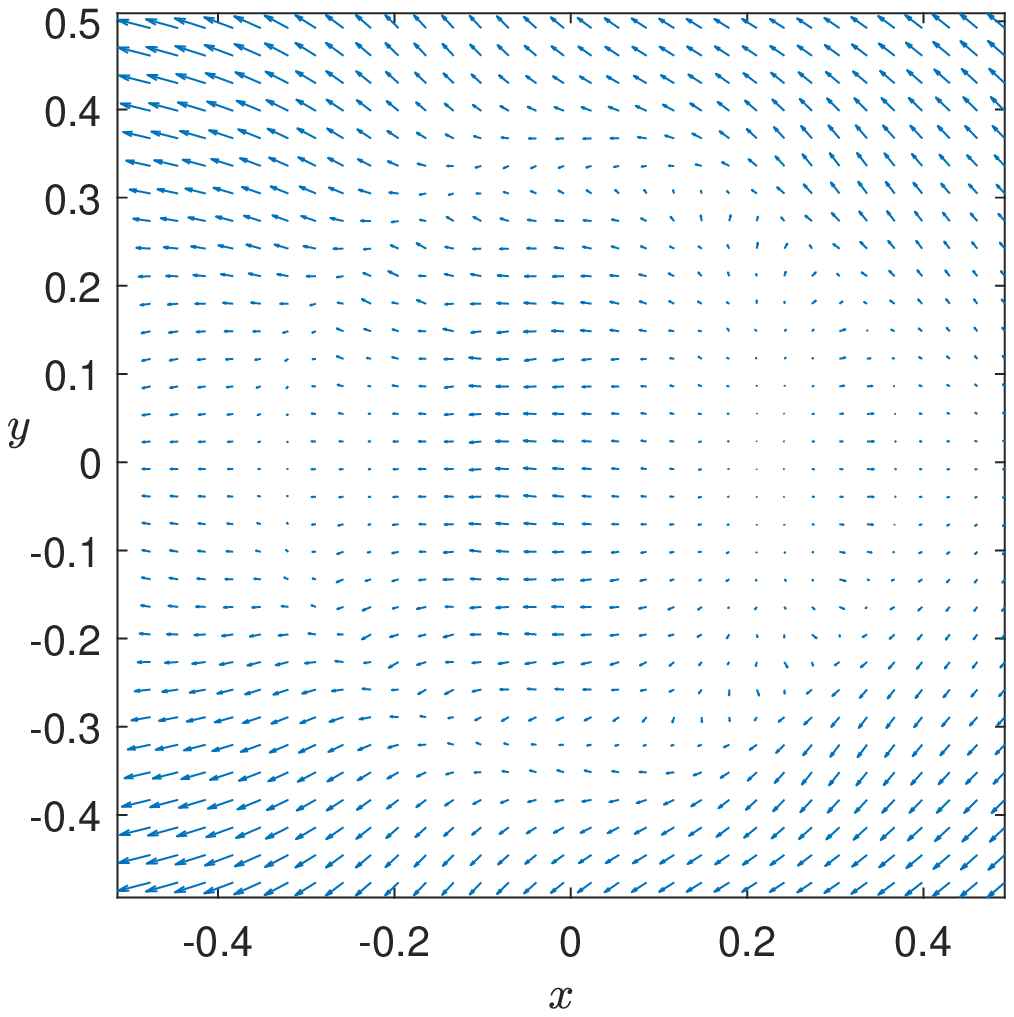}}\hspace{-2.25em}
		\subfloat[$t=0.75$]{\includegraphics[width=0.31\linewidth]{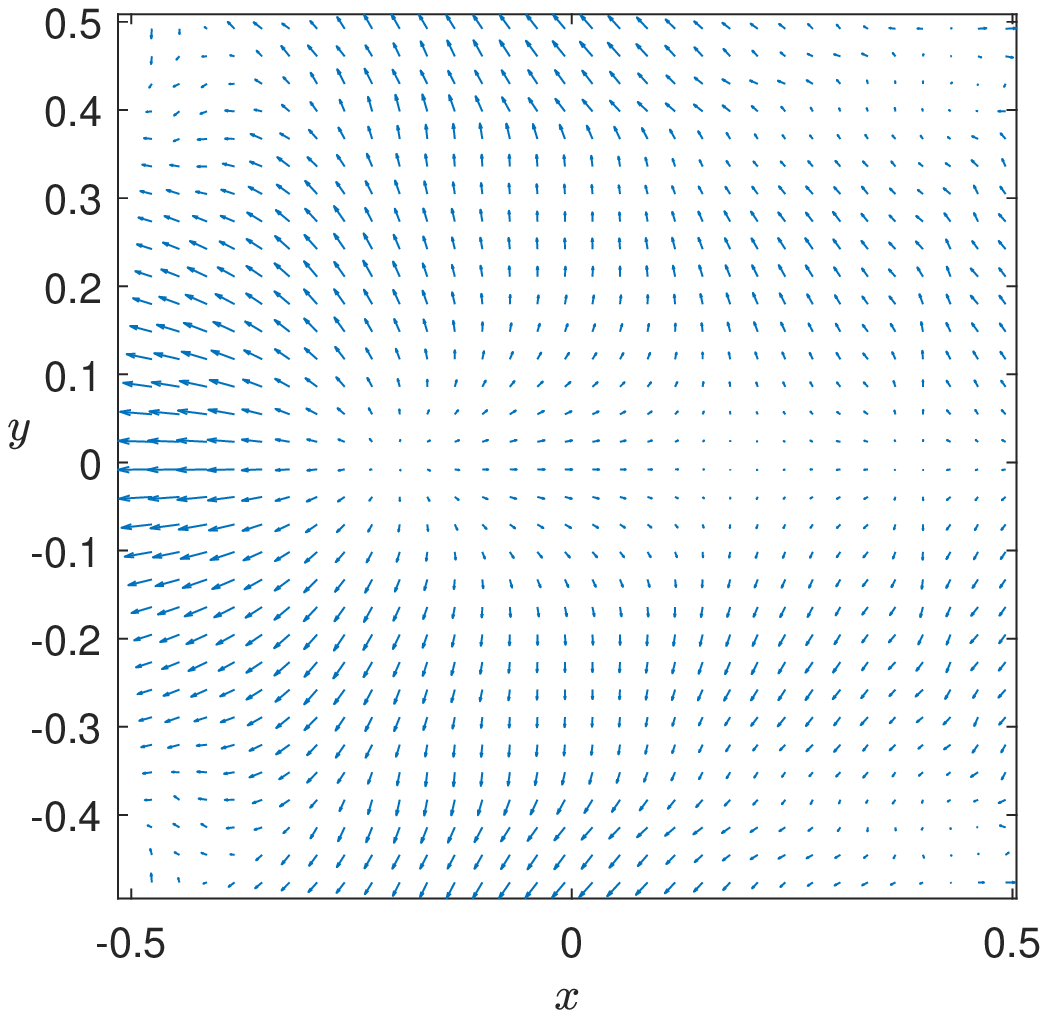}}\hspace{-2.25em}
		\subfloat[$t=1$]{\includegraphics[width=0.31\linewidth]{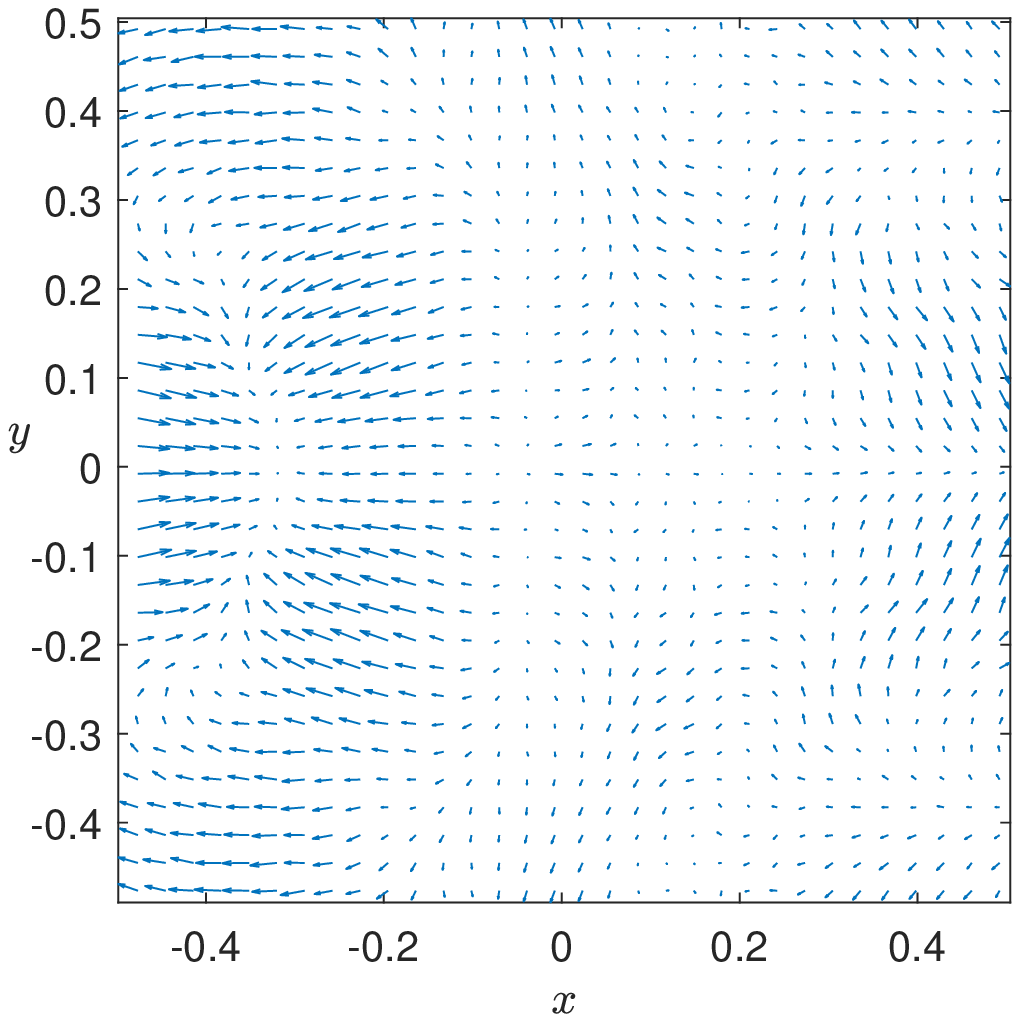}} \\
		\hspace{-2.em}\subfloat[$t=0.25$]{\includegraphics[width=0.31\linewidth]{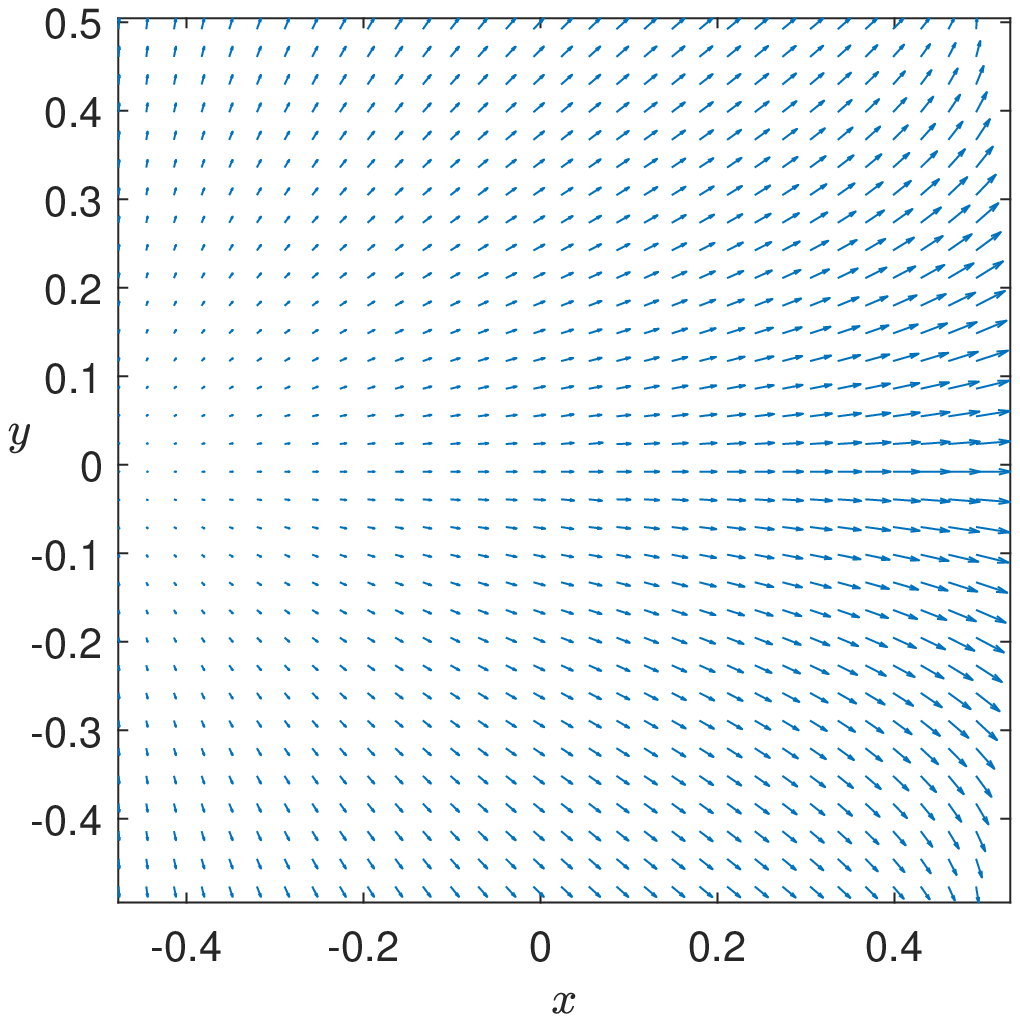}}\hspace{-2.25em}
		\subfloat[$t=0.5$]{\includegraphics[width=0.31\linewidth]{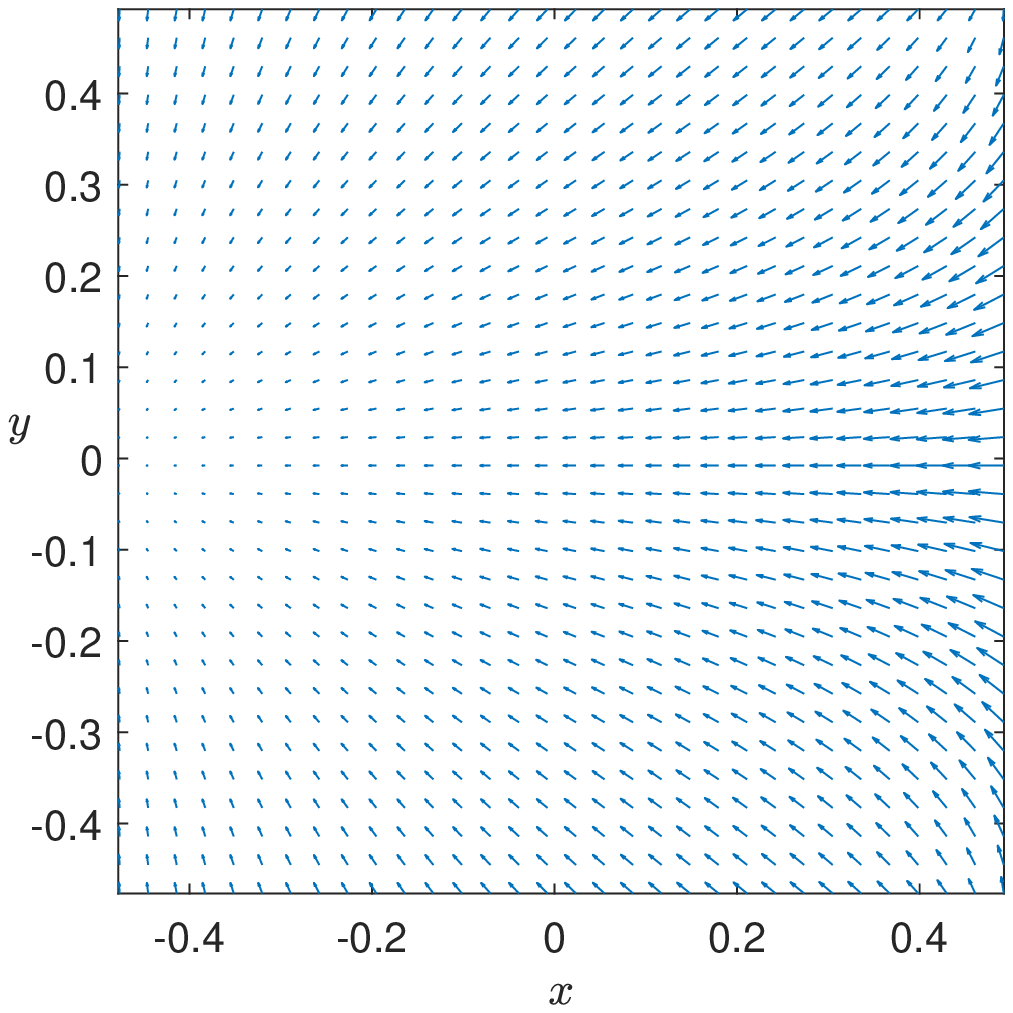}}\hspace{-2.25em}
		\subfloat[$t=0.75$]{\includegraphics[width=0.31\linewidth]{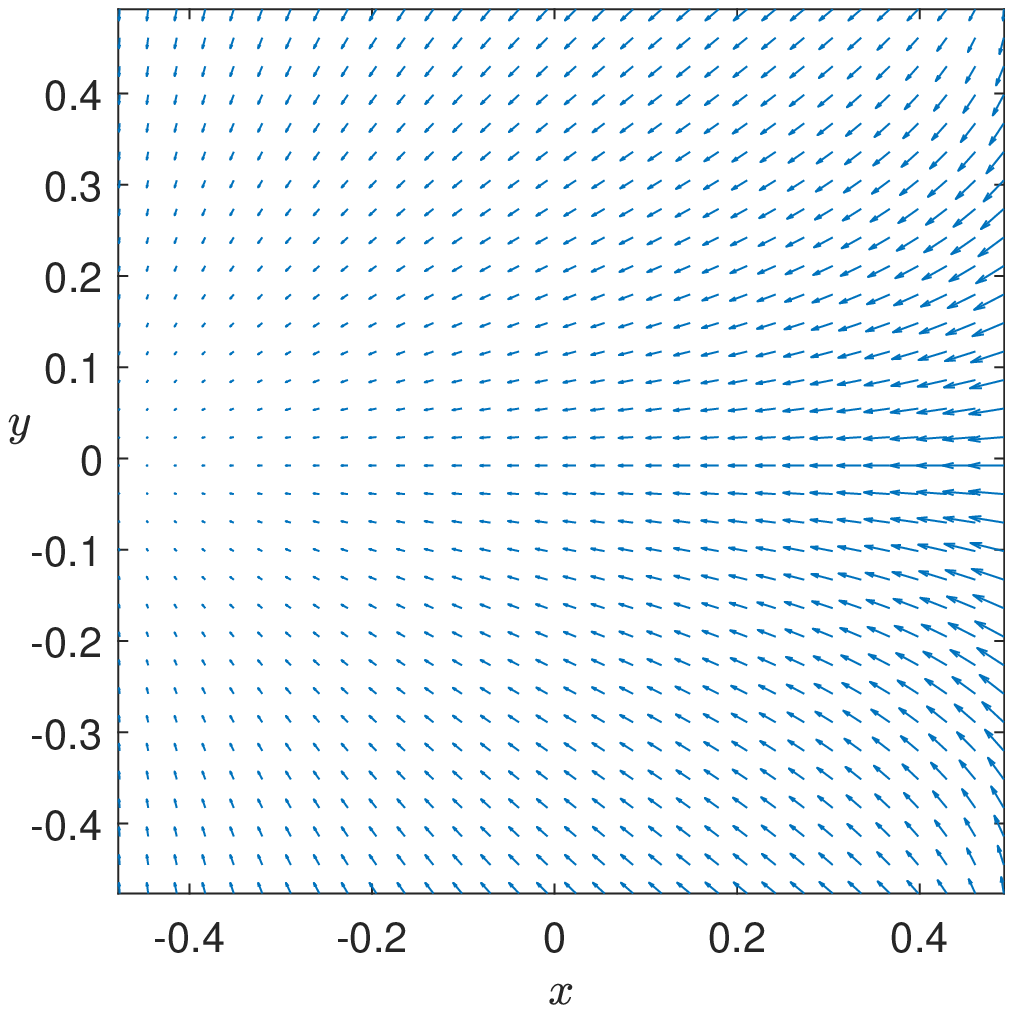}}\hspace{-2.25em}
		\subfloat[$t=1$]{\includegraphics[width=0.31\linewidth]{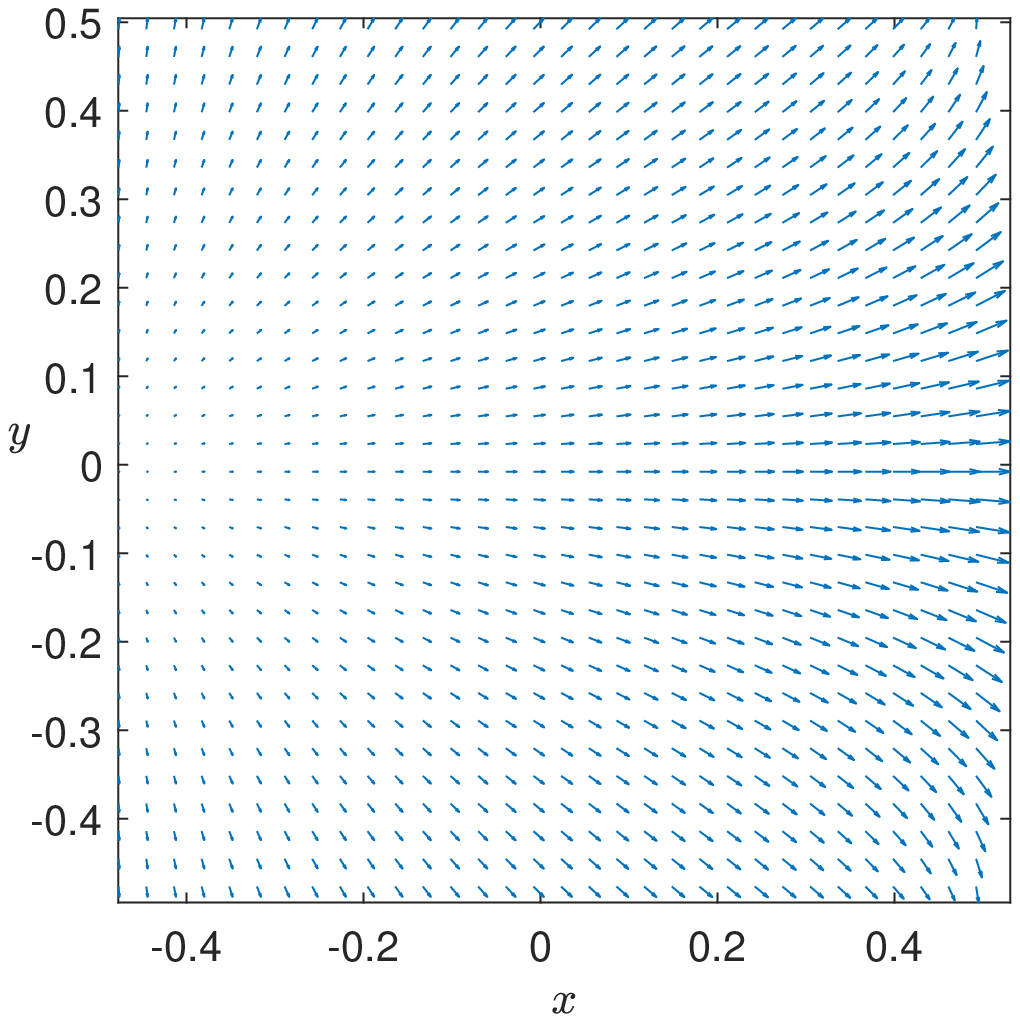}} 
	\end{tabular}
	\caption{Experiment 2 with $\damp = 3$. Top row: Director field $\bd_h$, bottom row: Electric field $\Grad\vphi_h$.}
	\label{fig:exp2largedamp}
\end{figure}

We observe in Figures~\ref{fig:exp2smalldamp} ($\damp=0.5$) and~\ref{fig:exp2largedamp} ($\damp=3$) that a singularity still seems to develop initially but then the effect of the electric field comes into play and forces the director field into a perpendicular alignment to the electric field. In the previous experiment, the director field only moved in the $x-y$-plane, and was so constrained into one direction that is perpendicular to the electric field. However, in this experiment, the director field moves in the whole $\mathbb{S}^2$ and can align perpendicular to the electric field in a 2-dimensional subspace of $\R^3$ which is what appears to happen and also leads to more dynamic behavior before the director field relaxes to an equilibrium state, see Figure~\ref{fig:exp2energy} for a plot of the evolution of the quantities~\eqref{eq:reducedenergy} and~\eqref{eq:damping} for $\damp=0.5$ and $\damp=3$. Again, the electric field does not appear to be perturbed by the director field very much, which could however also be due to our choice of parameters $\epsi$ and $\epsii$.

\begin{figure}[h]
	\begin{tabular}{lr}
		\subfloat[$\damp=0.5$]{\includegraphics[width=0.5\linewidth]{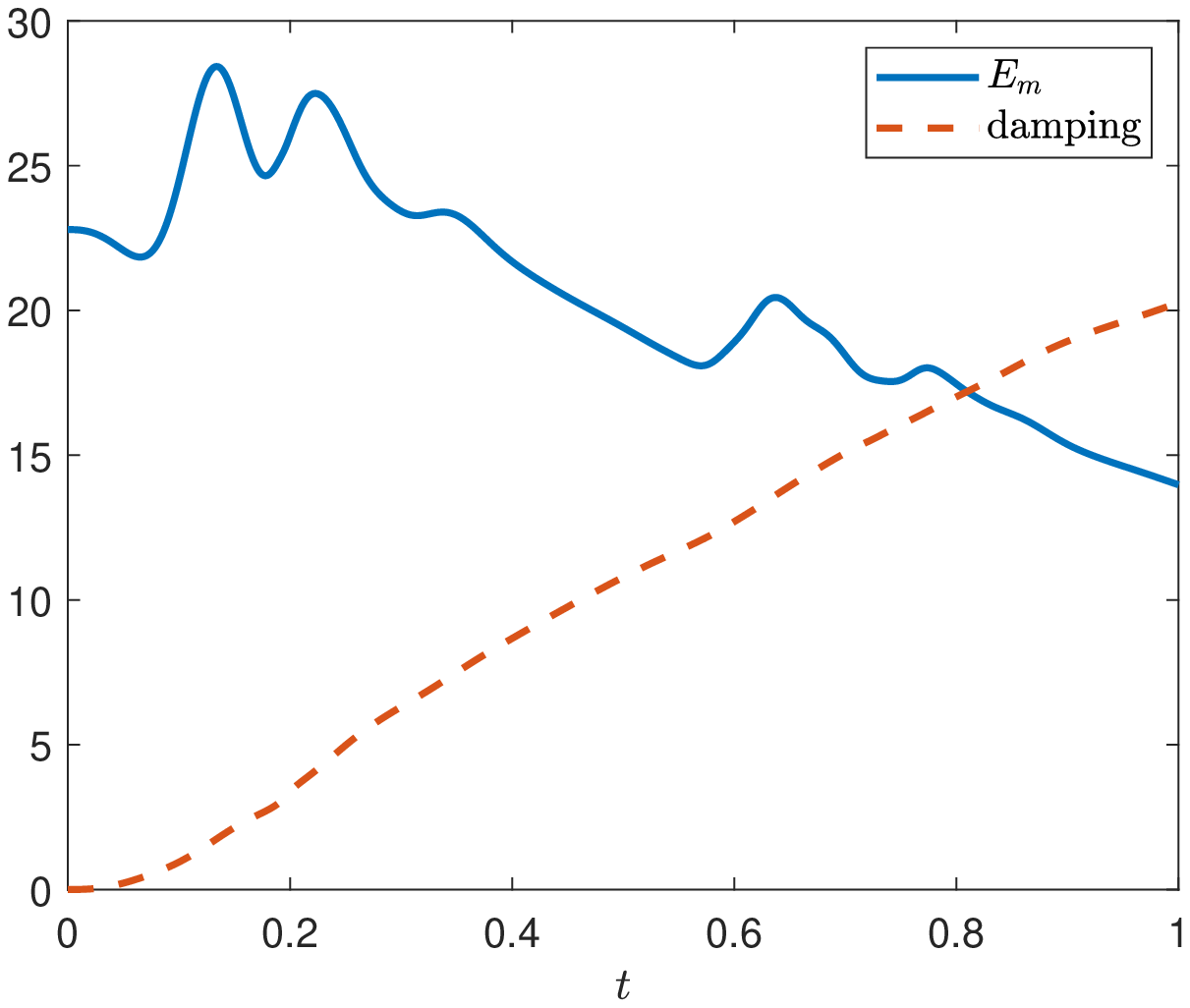}} 
		\subfloat[$\damp=3$]{\includegraphics[width=0.5\linewidth]{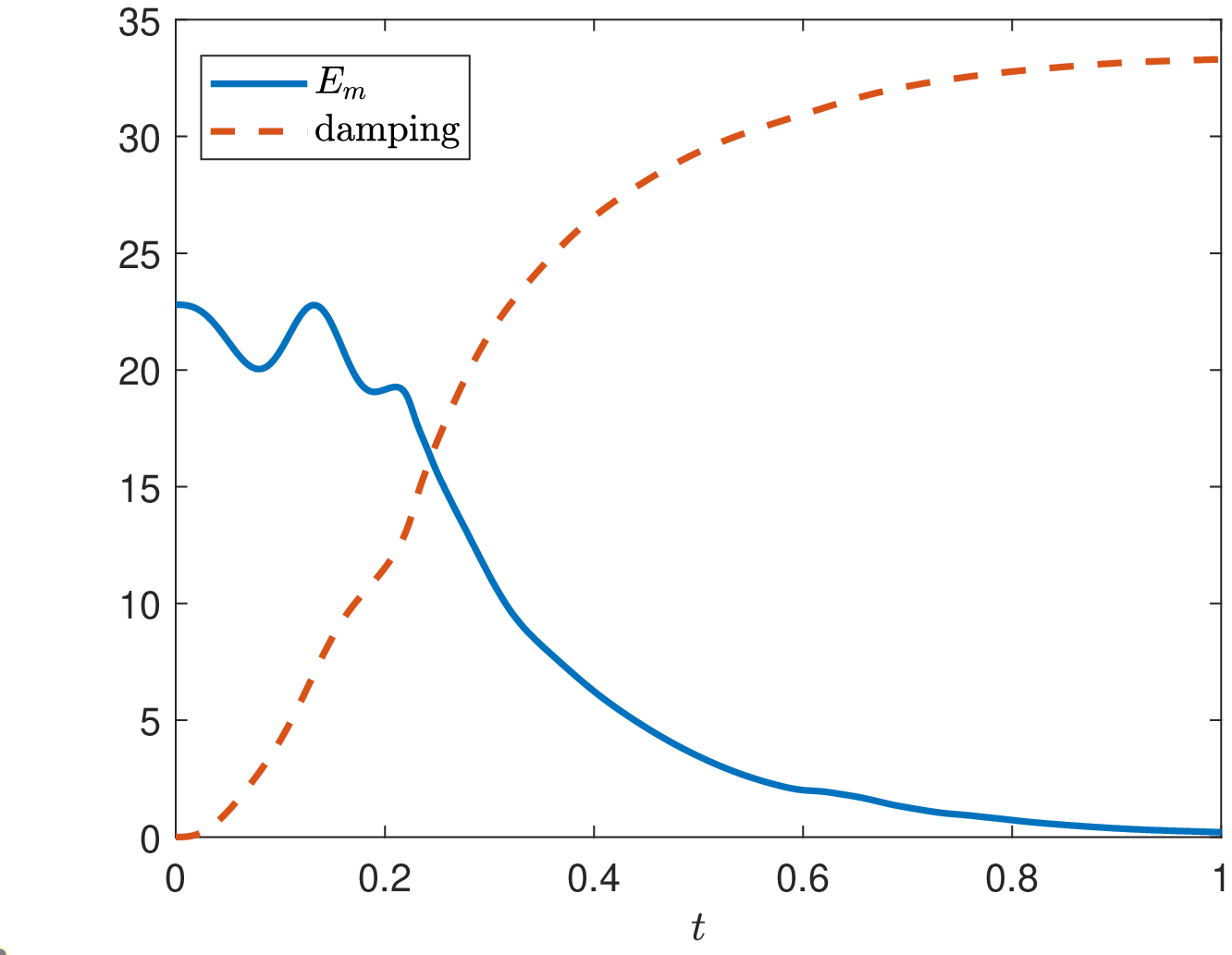}}
	\end{tabular}
	\caption{Experiment 2: Evolution of the energy~\eqref{eq:reducedenergy} and~\eqref{eq:damping} for $\damp=0.5$ and $3$.}
	\label{fig:exp2energy}
\end{figure}

\section{Acknowledgment}
I would like to thank Juan Pablo Borthagaray and Andreas Prohl for encouraging to finally write up this paper.
\appendix
\section{Real Analysis Folklore}
In this section we prove a lemma which can be found in the lecture notes of Kenneth Karlsen~\cite{kenneth} but is not proved there. The result is standard, but we provide its proof here for convenience.
\begin{lemma}\label{lem:kennethlemma}
	Let $u_m, v_m:\dom\to\R$ be sequences of measurable functions, such that
	\begin{align*}
	&u_m\rightarrow u,\quad \text{a.e. in } \dom, \quad \norm{u_m}_{L^\infty(\dom)}\leq C\, \quad\forall \, m\\
	&v_m\weak v,\quad \text{in } L^1(\dom), 
	\end{align*} 
	for some functions $u\in L^\infty(\dom)$, $v\in L^1(\dom)$. Then
	\begin{equation*}
	u_m v_m\rightharpoonup u\, v,\quad \text{in } L^1(\dom),\quad \text{as } m\rightarrow\infty.
	\end{equation*}
\end{lemma}
\begin{proof}
	For functions $v\in L^1(\dom)$, we define the truncation operator
	\begin{equation}\label{eq:truncation}
		T_M(v)=\begin{cases}
		 v,&\quad v\in [-M,M],\\
		 M,&\quad v> M,\\
		 -M,&\quad v<-M.
		\end{cases}
	\end{equation}
Then we let $\phi\in  L^\infty(\dom)$ be an arbitrary test function and use triangle inequality to estimate the difference
\begin{equation*}
\begin{split}
&\left| \int_{\dom} v_m u_m \phi\, dx - \int_{\dom} v u \phi\, dx\right|\leq  \left| \int_{\dom} \left(v-v_m\right) u\phi\, dx\right| +  \left|\int_{B_R(0)}  T_M(v_m) (u - u_m)\phi \,dx \right|\\
&\qquad\quad +  \left|\int_{\dom\backslash B_R(0)}  T_M(v_m) (u - u_m)\phi\, dx \right|+ \left| \int_{\dom} \left(v_m - T_M(v_m) \right)(u-u_m)\phi \, dx \right|\\
&:= A + B + D+ E.
\end{split}
\end{equation*}	
We fix $\eps>0$ and chooe $M>0$ so large that 
\begin{align*}
E & \leq  \norm{\phi}_{L^\infty} \left(\norm{u}_{L^\infty} + \norm{u_m}_{L^\infty}\right)\int_{\dom} |v_m - T_M(v_m) | dx\\
 & \leq C \int_{\{|v_m|\geq M\}} |v_m| \, dx<\frac{\eps}{4},
\end{align*}
which is possible since $v_m$ is weakly compact and therefore eqiintegrable. Thanks to the equiintegrability, we can also choose $R>0$ so large that 
\begin{equation*}
D\leq \left(\norm{u}_{L^\infty}+ \norm{u_m}_{L^\infty}\right) \norm{\phi}_{L^\infty} \int_{\dom\backslash B_R(0)} |T_M(v_m)|\, dx < \frac{\eps}{4}.
\end{equation*}
Then we let $N_0$ so large that for $m\geq N_0$, 
\begin{equation*}
A= \left| \int_{\dom} v_m u \phi\, dx - \int_{\dom} v u \phi\, dx\right|< \frac{\eps}{4},
\end{equation*}
which is possible since $v_m$ converges weakly in $L^1$ and $u\phi\in L^\infty$.
Then we choose $N_1\geq N_0$ so large that for $n\geq N_1$,
\begin{align*}
B\leq M \norm{\phi}_{L^\infty} \int_{B_R(0)} |u - u_m| \, dx<\frac{\eps}{4}, 
\end{align*}
which is possible by dominated convergence theorem, since $B_R(0)$ is bounded, $u_m$ converges strongly to $u$ almost everywhere and the sequence $u_m$ is uniformly bounded. Hence for $M,R, N_1$ large enough, we have
\begin{equation*}
\left| \int_{\dom} v_m u_m \phi\, dx - \int_{\dom} v u \phi\, dx\right| <\eps.
\end{equation*}
Since $\eps>0$ was arbitrary, this implies the claim.
\end{proof}	

\bibliographystyle{abbrv} 
\bibliography{efieldbib}

\end{document}